%% file: main.tex
\documentclass[12pts]{article}
\usepackage{amsmath,amsfonts,amssymb,amsthm,mathtools}
\usepackage{microtype}
\usepackage{lipsum}
\usepackage{url,hyperref}
\usepackage{indentfirst}
\hypersetup{colorlinks, linkcolor={red}, citecolor={blue}, urlcolor={black}}
\usepackage{comment}
\usepackage[margin=1.25in]{geometry}
\usepackage{enumerate,enumitem}
\usepackage{subcaption}
\usepackage{tikz}
\usepackage{array,multirow}

\allowdisplaybreaks

\DeclareMathOperator\pf{pf}
\DeclareMathOperator\Pf{Pf}
\DeclareMathOperator\pk{pk}
\DeclareMathOperator\Pk{Pk}
\DeclareMathOperator\Av{Av}
\DeclareMathOperator\ev{ev}
\DeclareMathOperator\pev{pev}

\newtheorem{theorem}{Theorem}[section]
\newtheorem{lemma}[theorem]{Lemma}

\newtheorem{corollary}[theorem]{Corollary}

\newtheorem{claim}[theorem]{Claim}

\newcounter{propcounter}

\title{Results on pattern avoidance in parking functions}
\author{Jun Yan\thanks{Mathematics Institute, University of Warwick, UK. Email: {\tt jun.yan@warwick.ac.uk}. Supported by the Warwick Mathematics Institute CDT and funding from the UK EPSRC (Grant number: EP/W523793/1).}}
\date{}

\begin{document}

\maketitle

\begin{abstract}
In this paper, we mainly study two notions of pattern avoidance in parking functions. First, for any collection of length 3 patterns, we compute the number of parking functions of size $n$ that avoid them under the first notion. This is motivated by the recent work of Adeniran and Pudwell, who obtained analogous results using a second notion of pattern avoidance. Then, we provide new purely bijective proofs for two of their results, and improve the formula of another one. Finally, we apply similar enumeration techniques to the work of Novelli and Thibon on certain Hopf algebras of generalised parking functions, and compute their graded dimensions. 
\end{abstract}

\section{Introduction}
Let $f:[n]\to[n]$ be a function. Suppose that $n$ cars enter a one-way parking lot with $n$ parking spots and attempt to park one by one. For every $i\in[n]$, the $i$-th car prefers the $f(i)$-th parking spot, and it drives straight there. If the $f(i)$-th spot is still available, it parks there, otherwise it continues down the parking lot and parks at the first available spot, or exits without parking if no such spot exists. We say that $f:[n]\to[n]$ is a \textit{parking function} of \textit{size} $n$ if all $n$ cars can park successfully. See Figure \ref{fig:parkingfunctionexample} for an example.

It is well known and easy to show that $f:[n]\to[n]$ is a parking function if and only if at least $i$ cars prefer the first $i$ parking spots for every $i\in[n]$. In other words, if and only if the following condition holds:
\stepcounter{propcounter}
\begin{enumerate}[label = \textbf{\Alph{propcounter}}]
\item\label{parkingcondition}  for every $i\in [n]$, $|\{j\mid f(j)\leq i\}|\geq i$.
\end{enumerate}

The study of pattern avoidance in parking functions is motivated by the notion of pattern avoidance in permutations. For $m\leq n$ and $\sigma\in S_m$, $\pi\in S_n$, we say that $\pi$ \textit{contains} $\sigma$ \textit{as a pattern} if there exists $1\leq i_1<\cdots<i_m\leq n$, such that $\pi(i_a)<\pi(i_b)$ if and only if $\sigma(a)<\sigma(b)$ for all $a,b\in[m]$, and we say $\pi$ \textit{avoids} $\sigma$ otherwise. For any collection $\sigma_1,\cdots,\sigma_k$ of permutations, we denote by $\Av_n(\sigma_1,\cdots,\sigma_k)$ the set of all permutations in $S_n$ containing none of $\sigma_1,\cdots,\sigma_k$ as a pattern. Depending on how one associates a permutation to each parking function, there are several possible notions of pattern avoidance in parking functions, and we study two of these in this paper. 

The first such notion looks at the final parking positions. For a parking function $f:[n]\to[n]$, the \textit{parking permutation} associated to $f$ is defined to be the permutation $\rho_f\in S_n$ satisfying that the $i$-th spot in the parking lot is occupied by the $\rho_f(i)$-th car. See Figure \ref{fig:parkingfunctionexample} for an example. Note that different parking functions could have the same associated parking permutation. 

For a collection $\sigma_1,\cdots,\sigma_k$ of permutations, let $\Pk_n(\sigma_1,\cdots,\sigma_k)$ be the set of parking functions $f:[n]\to[n]$ such that its associated parking permutation $\rho_f$ contains none of $\sigma_1,\cdots,\sigma_k$ as a pattern, or equivalently, $\Pk_n(\sigma_1,\cdots,\sigma_k)=\{f:[n]\to[n]\mid f\text{ is a parking function and }\rho_f\in\Av_n(\sigma_1,\cdots,\sigma_k)\}$. Let $\pk_n(\sigma_1,\cdots,\sigma_k)=|\Pf_n(\sigma_1,\cdots,\sigma_k)|$. 

In Section \ref{computepk} we systematically compute the values of $\pk_n(\sigma_1,\cdots,\sigma_k)$ for all collections $\sigma_1,\cdots,\sigma_k$ of permutations in $S_3$. Moreover, the formulas we prove are all explicit and non-recursive, except for $\pk_n(\sigma)$ when $\sigma\in\{132, 231, 312, 321\}$, which are more difficult. This section is motivated by the work of Adeniran and Pudwell in \cite{AP}, where they obtained analogous results using a different notion of pattern avoidance in parking functions that we define now. 

This second notion is introduced by Qiu and Remmel \cite{Q,QR}, and is motivated by their work on pattern avoidance in ordered set partitions. A parking function $f:[n]\to[n]$ is represented in \textit{block notation} by a sequence of $n$ sets $(B_1,\cdots, B_n)$, each called a \textit{block}, where $B_i=f^{-1}(i)$ for all $i\in[n]$. This block notation of parking function is introduced by Adeniran and Pudwell \cite{AP}, and is a simpler form of the labelled Catalan path notation introduced by Garsia and Haiman \cite{GH}. See Figure \ref{fig:parkingfunctionexample} for an example. Note that by \ref{parkingcondition}, a sequence of $n$ disjoint subsets $(B_1,\cdots, B_n)$ of $[n]$ satisfying $\cup_{j=1}^nB_i=[n]$ is a parking function represented in block notation if and only if the following holds:
\stepcounter{propcounter}
\begin{enumerate}[label = \textbf{\Alph{propcounter}}]
\item\label{blockcondition} for every $i\in[n]$, $|\cup_{j=1}^iB_i|\geq i$.
\end{enumerate}

Given a parking function $f:[n]\to[n]$ represented in block notation by $(B_1,\cdots, B_n)$, write down the elements in $B_1,\cdots,B_n$ in this order, where for each $B_i$ with size at least 2, we write the elements in $B_i$ down in increasing order. In this way, we obtain a permutation $\pi_f\in S_n$, where $\pi_f(i)$ is the $i$-th element written down, and we call $\pi_f$ the \textit{block permutation} associated to $f$. See Figure \ref{fig:parkingfunctionexample} for an example. Note that again, the same permutation could be the block permutation associated to several distinct parking functions. 

For a collection $\sigma_1,\cdots,\sigma_k$ of permutations, let $\Pf_n(\sigma_1,\cdots,\sigma_k)$ be the set of parking function $f:[n]\to[n]$ such that its associated block permutation $\pi_f$ contains none of $\sigma_1,\cdots,\sigma_k$ as a pattern, or equivalently, $\Pf_n(\sigma_1,\cdots,\sigma_k)=\{f:[n]\to[n]\mid f\text{ is a parking function and }\pi_f\in\Av_n(\sigma_1,\cdots,\sigma_k)\}$. Let $\pf_n(\sigma_1,\cdots,\sigma_k)=|\Pf_n(\sigma_1,\cdots,\sigma_k)|$.

%We say that a parking function $f:[n]\to[n]$ \textit{contains} $\sigma$ \textit{as a block pattern} if the block permutation $\pi_f\in S_n$ associated to $f$ does, and we say $f$ \textit{avoids} $\sigma$ \textit{as a block pattern} otherwise. 

As an example, $\pf_n(12)=1$ as the only parking function in $\Pf_n(12)$ is the one with block notation $(\{n\},\{n-1\},\cdots,\{1\})$. It is also relatively easy to see that $\pf_n(21)=C_n$, the $n$-th Catalan number. Indeed, the block permutation associated to any parking function in $\Pf_n(21)$ must be the identity, and given a parking function in $\Pf_n(21)$ with block notation $B_1,B_2,\cdots,B_n$, if from left to right we draw an up-step for every element encountered in the blocks, and draw a down-step every time we reach the end of a set $B_i$, then condition \ref{blockcondition} implies that this gives a bijection with the set $\mathcal{C}_n$ of Catalan paths of length $2n$. Since parking functions in $\Pf_n(21)$ are exactly the size $n$ parking functions that are increasing, this also gives a bijection between size $n$ parking functions and the Catalan paths of length $2n$, which we will use later. 

\begin{figure}[t]
    \centering
    \input{parkingfunctionexample}
    \caption{An example parking function $f:[7]\to[7]$ illustrating relevant concepts. Note that $\rho_f$ contains 312 as a pattern, but avoids 132, while $\pi_f$ contains 132 as a pattern, but avoids 213.}\label{fig:parkingfunctionexample}
\end{figure}

The first non-trivial result under this notion of pattern avoidance is due to Qiu and Remmel \cite{Q}, who computed the values of $\pf_n(123)$. Adeniran and Pudwell \cite{AP} systematically computed all values of $\pf_n(\sigma_1,\cdots,\sigma_k)$, where $\sigma_1,\cdots,\sigma_k$ is any collection of at least two permutations in $S_3$. The values of $\pf_n(\sigma)$ for $\sigma\in S_3\setminus\{123\}$ remains open. Two of the results proven by Adeniran and Pudwell \cite{AP} are as follows. 

\begin{theorem}\label{thm:main1}
$\pf_n(123,132)$ is equal to the number of ordered rooted trees with $n+1$ edges and odd root degrees.
\end{theorem}

\begin{theorem}\label{thm:main2}
$\pf_n(123,213)=C_{n+1}-C_n$, which is equal to the number of ordered rooted trees with $n+1$ edges and root degrees at least 2. 
\end{theorem}

The proofs of these two results in \cite{AP} are both algebraic, specifically by setting up and solving recurrence relations using generating function or computer-assisted induction. In view of the simple combinatorial interpretation of the numbers $\pf_n(123,132)$ and $\pf_n(123,213)$, Adeniran and Pudwell \cite{AP} posed the question of finding purely bijective proofs of these two results. We answer this question by presenting bijective proofs of these two results in Section \ref{123132} and \ref{123213}, respectively.

Finally, we apply similar algebraic techniques as the ones employed in the more difficult cases in Section \ref{computepk} to some related enumeration problems on parking functions. We first improve a result of Adeniran and Pudwell in \cite{AP} by proving a more explicit formula for $\pf_n(312, 321)$ in Section \ref{compute312321}. Then, in Section \ref{related}, building on the work of Novelli and Thibon in \cite{NT}, we provide several new formulas on the graded dimensions of certain Hopf algebras of generalised parking functions, which is also the number of congruence classes of generalised parking functions under certain congruence relations. In the process, we confirm several conjectures. 

\section{Pattern avoidance in parking permutations}\label{computepk}
In this section, analogous to the work by Adeniran and Pudwell \cite{AP} on enumerating parking functions whose block permutations avoid certain length 3 patterns, we systematically compute $\pk_n(\sigma_1,\cdots,\sigma_k)$, the number of parking functions whose parking permutations avoid a certain collection $\sigma_1,\cdots,\sigma_k$ of permutations in $S_3$. Under this notion of pattern avoidance, determining $\pk_n(\sigma_1,\cdots,\sigma_k)$ is in general slightly easier due to the following lemma, a proof of which can be found in \cite{CHJKRSV}, and its immediate corollary. 

\begin{lemma}
For any $\rho\in S_n$ and $i\in[n]$, let $\ell(i,\rho)=\max\{\ell\mid\rho(j)\leq\rho(i)\text{ for all }i-\ell+1\leq j\leq i\}$, and let $\ell(\rho)=\prod_{i=1}^n\ell(i,\rho)$. Then, the number of parking function $f:[n]\to[n]$ with $\rho_f=\rho$ is $\ell(\rho)$.
\end{lemma}

\begin{corollary}\label{generalpkformula}
$$\pk_n(\sigma_1,\cdots,\sigma_k)=\sum_{\rho\in\Av_n(\sigma_1,\cdots,\sigma_k)}\ell(\rho)=\sum_{\rho\in\Av_n(\sigma_1,\cdots,\sigma_k)}\prod_{i=1}^n\ell(i,\rho).$$
\end{corollary}

Most of the proofs below directly apply Corollary \ref{generalpkformula} to the descriptions of permutation in $\Av_n(\sigma_1,\cdots,\sigma_k)$ given by Adeniran and Pudwell \cite{AP}. As such, we now introduce several notations they used to describe permutations. We denote the identity permutation by $I_n$, and denote the permutation $\sigma\in S_n$ satisfying $\sigma(i)=n+1-i$ for all $i\in[n]$ by $J_n$. For permutations $\sigma\in S_n, \tau\in S_m$, we define the \textit{direct sum} $\sigma\oplus\tau$ and \textit{skew sum} $\sigma\ominus\tau$ to be the permutations in $S_{n+m}$ given as follows.
$$(\sigma\oplus\tau)_i=\begin{cases}
    \sigma_i, &1\leq i\leq n,\\
    \tau_{i-n}+n, &n+1\leq i\leq n+m.
\end{cases}$$
$$(\sigma\ominus\tau)_i=\begin{cases}
    \sigma_i+m, &1\leq i\leq n,\\
    \tau_{i-n}, &n+1\leq i\leq n+m.
\end{cases}$$
For $S\subset[n]$, we also define $I_S$ to be the elements of $S$ listed in increasing order, and $J_S$ to be the elements of $S$ listed in decreasing order. If several lists of elements partitioning $[n]$ are written directly next to one another, then that denotes the permutation in $S_n$ obtained by concatenating these lists. For example, $I_{\{2,5\}}6J_{\{1,3,4\}}$ denotes the permutation $\sigma=256431\in S_6$.

As two examples, by Corollary \ref{generalpkformula}, $\pk_n(12)=1$ as $\Av_n(12)=\{J_n\}$ and $\ell(J_n)=\prod_{i=1}^n\ell(i,J_n)=\prod_{i=1}^n1=1$, while $\pk_n(21)=n!$ as $\Av_n(21)=\{I_n\}$ and $\ell(I_n)=\prod_{i=1}^n\ell(i,I_n)=\prod_{i=1}^ni=n!$

Note that by the Erd\H{o}s-Szekeres Theorem, any permutation containing neither 123 nor 321 as a pattern must have length at most 4, so in the computations below we will only consider collections $\sigma_1,\cdots,\sigma_k$ that do not contain 123 and 321 together. We also follow the convention that $\Av_0(\sigma_1,\cdots,\sigma_k)$ has size 1 and contains the unique empty permutation, and $\pk_0(\sigma_1,\cdots,\sigma_k)=1$. 

Finally, in some of the more difficult cases below that require the use of generating functions, we use $[x^n]F(x)$ to denote the coefficient of the term $x^n$ in the formal power series $F(x)$.

\subsection{Avoiding five patterns}
\begin{center}
\begin{tabular}
{|c|c|c|c|}
\hline  Patterns $\sigma_1,\sigma_2,\sigma_3,\sigma_4,\sigma_5$  & $\pk_n(\sigma_1,\sigma_2,\sigma_3,\sigma_4,\sigma_5)$, $1\leq n\leq8$ & OEIS & Result \\\hline
  123, 132, 213, 231, 312   & 1, 3, 1, 1, 1, 1, 1, 1 & A000012 ($n\geq3$) & Theorem \ref{2.1}\\\hline

  132, 213, 231, 312, 321   & 1, 3, 6, 24, 120, 720, 5040, 40320 & A338112 & Theorem \ref{2.2}\\\hline
\end{tabular}
\end{center}

\begin{theorem}\label{2.1}
$$\pk_n(123,132,213,231,312)=\begin{cases}
      1, & \text{if }n\not=2, \\
      3, & \text{otherwise}.
    \end{cases}$$
\end{theorem}
\begin{proof}
The cases when $n=1,2$ are easy to verify, so we assume $n\geq 3$. By Theorem 1 in \cite{AP}, $\Av_n(123,132,213,231,312)=\{J_n\}$. Thus, $\pk_n(123,132,213,231,312)=\ell(J_n)=1$. 
\end{proof}

\begin{theorem}\label{2.2}
$$\pk_n(132,213,231,312,321)=\begin{cases}
      n!, & \text{if }n\not=2, \\
      3, & \text{otherwise}.
    \end{cases}$$    
\end{theorem}
\begin{proof}
The cases when $n=1,2$ are easy to verify, so we assume $n\geq 3$. By Theorem 2 in \cite{AP}, $\Av_n(132,213,231,312,321)=\{I_n\}$. Thus, $\pk_n(123,132,213,231,312)=\ell(I_n)=n!$.
\end{proof}

\subsection{Avoiding four patterns}
\begin{center}
\begin{tabular}{|c|c|c|c|}
\hline  Patterns $\sigma_1,\sigma_2,\sigma_3,\sigma_4$  & $\pk_n(\sigma_1,\sigma_2,\sigma_3,\sigma_4)$, $1\leq n\leq8$ & OEIS & Result \\\hline
      123, 132, 213, 231   & \multirow{3}{*}{1, 3, 3, 3, 3, 3, 3, 3} & \multirow{3}{*}{A122553} & \multirow{3}{*}{Theorem \ref{2.3}}\\
  123, 132, 213, 312   & & &\\
  123, 213, 231, 312   & & &\\\hline
  123, 132, 231, 312   & 1, 3, 4, 5, 6, 7, 8, 9 & A065475 & Theorem \ref{2.4}\\\hline
  132, 213, 231, 312   & 1, 3, 7, 25, 121, 721, 5041, 40321 & A038507 ($n\geq2$) & Theorem \ref{2.5}\\\hline
  132, 213, 231, 321   & \multirow{3}{*}{1, 3, 8, 30, 144, 840, 5760, 45360} & \multirow{3}{*}{A059171} & \multirow{3}{*}{Theorem \ref{2.6}}\\
  132, 213, 312, 321   & & &\\
  213, 231, 312, 321   & & &\\\hline
  132, 231, 312, 321  & 1, 3, 9, 36, 180, 1080, 7560, 60480 & A070960 & Theorem \ref{2.7}\\\hline
\end{tabular}
\end{center}

\begin{theorem}\label{2.3}
$$\pk_n(123, 132, 213, 231)=\pk_n(123, 132, 213, 312)=\pk_n(123, 213, 231, 312)=\begin{cases}
    3, & \text{if }n\not=1, \\
    1, & \text{otherwise}.    
\end{cases}$$
\end{theorem}
\begin{proof}
The cases when $n=1$ are clear, so we assume $n\geq2$ from now on. 

By Theorem 3 in \cite{AP}, $\Av_n(123, 132, 213, 231)=\{J_n,J_{n-2}\ominus I_2\}$, so $\pk_n(123, 132, 213, 231)=\ell(J_n)+\ell(J_{n-2}\ominus I_2)=1+2=3$. 

By Theorem 4 in \cite{AP}, $\Av_n(123, 132, 213, 312)=\{J_n,I_2\ominus J_{n-2}\}$, so $\pk_n(123, 132, 213, 312)=\ell(J_n)+\ell(I_2\ominus J_{n-2})=1+2=3$. 

By Theorem 4 in \cite{AP}, $\Av_n(123, 213, 231, 312)=\{J_n,1\oplus J_{n-1}\}$, so $\pk_n(123, 213, 231, 312)=\ell(J_n)+\ell(1\oplus J_{n-1})=1+2=3$. 
\end{proof}

\begin{theorem}\label{2.4}
$$\pk_n(123, 132, 231, 312)=\begin{cases}
    n+1, & \text{if }n\not=1, \\
    1, & \text{otherwise}.    
\end{cases}$$    
\end{theorem}
\begin{proof}
The case when $n=1$ is clear. For $n\geq2$, by Theorem 3 in \cite{AP}, $\Av_n(123, 132, 231, 312)=\{J_n,J_{n-1}\oplus 1\}$, so $\pk_n(123, 132, 231, 312)=\ell(J_n)+\ell(J_{n-1}\oplus 1)=1+n$. 
\end{proof}

\begin{theorem}\label{2.5}
$$\pk_n(132, 213, 231, 312)=\begin{cases}
    n!+1, & \text{if }n\not=1, \\
    1, & \text{if }n=1.\\
\end{cases}$$    
\end{theorem}
\begin{proof}
The case when $n=1$ is clear. For $n\geq3$, by Theorem 5 in \cite{AP}, $\Av_n(132, 213, 231, 312)=\{I_n,J_n\}$, so $\pk_n(132, 213, 231, 312)=\ell(I_n)+\ell(J_n)=n!+1$. 
\end{proof}

\begin{theorem}\label{2.6}
\begin{align*}
\pk_n(132, 213, 231, 321)&=\pk_n(132, 213, 312, 321)=\pk_n(213, 231, 312, 321)=\begin{cases}
    \frac{(n+1)!}{n}, & \text{if }n\not=1, \\
    1, & \text{otherwise}.    
\end{cases}
\end{align*}
\end{theorem}
\begin{proof}
The cases when $n=1$ are clear, so we assume $n\geq2$ from now on. 

By Theorem 6 in \cite{AP}, $\Av_n(132, 213, 231, 321)=\{I_n,1\ominus I_{n-1}\}$, so $\pk_n(132, 213, 231, 321)=\ell(I_n)+\ell(1\ominus I_{n-1})=n!+(n-1)!=\frac{(n+1)!}{n}$. 

By Theorem 7 in \cite{AP}, $\Av_n(132, 213, 312, 321)=\{I_n,I_{n-1}\ominus 1\}$, so $\pk_n(132, 213, 312, 321)=\ell(I_n)+\ell(I_{n-1}\ominus 1)=n!+(n-1)!=\frac{(n+1)!}{n}$. 

By Theorem 7 in \cite{AP}, $\Av_n(213, 231, 312, 321)=\{I_n,I_{n-2}\oplus J_2\}$, so $\pk_n(213, 231, 312, 321)=\ell(J_n)+\ell(I_{n-2}\oplus J_2)=n!+(n-1)!=\frac{(n+1)!}{n}$. 
\end{proof}

\begin{theorem}\label{2.7}
$$\pk_n(132, 231, 312, 321)=\begin{cases}
    \frac32n!, & \text{if }n\not=1, \\
    1, & \text{if }n=1,
\end{cases}$$    
\end{theorem}
\begin{proof}
The case when $n=1$ is clear. For $n\geq2$, by Theorem 6 in \cite{AP}, $\Av_n(132, 231, 312, 321)=\{I_n,J_2\oplus I_{n-2}\}$, so $\pk_n(132, 231, 312, 321)=\ell(I_n)+\ell(J_2\oplus I_{n-2})=n!+\frac12n!$. 
\end{proof}

\subsection{Avoiding three patterns}
\begin{center}
\begin{tabular}{|c|c|c|c|}
\hline  Patterns $\sigma_1,\sigma_2,\sigma_3$  & $\pk_n(\sigma_1,\sigma_2,\sigma_3)$, $1\leq n\leq8$ & OEIS & Result \\\hline
  123, 132, 231   & \multirow{3}{*}{1, 3, 6, 10, 15, 21, 28, 36} & \multirow{3}{*}{A000217} & \multirow{3}{*}{Theorem \ref{2.8}}\\
  123, 132, 312   & & &\\
  123, 231, 312   & & &\\\hline
  123, 213, 231   & \multirow{2}{*}{1, 3, 5, 7, 9, 11, 13, 15} & \multirow{2}{*}{A005408} & \multirow{2}{*}{Theorem \ref{2.9}}\\
  123, 213, 312   & & &\\\hline
  123, 132, 213   & 1, 3, 5, 11, 21, 43, 85, 171 & A001045 & Theorem \ref{2.10}\\\hline
  132, 213, 231   & \multirow{3}{*}{1, 3, 9, 33, 153, 873, 5913, 46233} & \multirow{3}{*}{A007489} & \multirow{3}{*}{Theorem \ref{2.11}}\\
  123, 213, 312   & & &\\
  213, 231, 312   & & &\\\hline
  132, 231, 312   & 1, 3, 10, 41, 206, 1237, 8660, 69281 & A002627 & Theorem \ref{2.12}\\\hline
  132, 231, 321   & \multirow{2}{*}{1, 3, 11, 50, 274, 1764, 13068, 109584} & \multirow{2}{*}{A000254} & \multirow{2}{*}{Theorem \ref{2.13}}\\
  132, 312, 321   & & &\\\hline
  132, 213, 321   & \multirow{2}{*}{1, 3, 10, 40, 192, 1092, 7248, 55296} & \multirow{2}{*}{A136128} & \multirow{2}{*}{Theorem \ref{2.14}}\\
  213, 231, 321   & & &\\\hline
  213, 312, 321   & 1, 3, 10, 42, 216, 1320, 9360, 75600 & A007680 & Theorem \ref{2.15}\\\hline
  231, 312, 321   & 1, 3, 11, 53, 309, 2119, 16687, 148329 & A000255 & Theorem \ref{2.16}\\\hline
\end{tabular}
\end{center}

\begin{theorem}\label{2.8}
$\pk_n(123, 132, 231)=\pk_n(123, 132, 312)=\pk_n(123, 231, 312)=\binom{n+1}2.$
\end{theorem}
\begin{proof}
By Theorem 8 in \cite{AP}, $\Av_n(123, 132, 231)=\{J_{n-k}\ominus(J_{k-1}\oplus1)\mid k\in[n]\}$, so $\pk_n(123, 132, 231)=\sum_{k=1}^n\ell(J_{n-k}\ominus(J_{k-1}\oplus1))=\sum_{k=1}^nk=\binom{n+1}2$. 

By Theorem 9 in \cite{AP}, $\Av_n(123, 132, 312)=\{(J_{k-1}\oplus1)\ominus J_{n-k}\mid k\in[n]\}$, so $\pk_n(123, 132, 312)=\sum_{k=1}^n\ell((J_{k-1}\oplus1)\ominus J_{n-k})=\sum_{k=1}^nk=\binom{n+1}2$.

By Theorem 9 in \cite{AP}, $\Av_n(123, 231, 312)=\{J_k\oplus J_{n-k}\mid k\in[n]\}$, so $\pk_n(123, 231, 312)=\sum_{k=1}^n\ell(J_k\oplus J_{n-k})=\sum_{k=1}^{n-1}(k+1)+1=\binom{n+1}2$. 
\end{proof}

\begin{theorem}\label{2.9}
$\pk_n(123, 213, 231)=\pk_n(123, 213, 312)=2n-1.$
\end{theorem}
\begin{proof}
By Theorem 9 in \cite{AP}, $\Av_n(123, 213, 231)=\{J_{k-1}\ominus(1\oplus J_{n-k})\mid k\in[n]\}$, so $\pk_n(123, 213, 231)=\sum_{k=1}^n\ell(J_{k-1}\ominus(1\oplus J_{n-k}))=\sum_{k=1}^{n-1}2+1=2n-1$. 

By Theorem 10 in \cite{AP}, $\Av_n(123, 213, 312)=\{(1\oplus J_{n-k})\ominus J_{k-1}\mid k\in[n]\}$, so $\pk_n(123, 213, 312)=\sum_{k=1}^n\ell((1\oplus J_{n-k})\ominus J_{k-1})=\sum_{k=1}^{n-1}2+1=2n-1$.
\end{proof}

\begin{theorem}\label{2.10}
$\pk_n(123, 132, 213)=\frac232^n+\frac13(-1)^n$.
\end{theorem}
\begin{proof}
The cases when $n=1,2$ are easy to verify, so assume $n\geq 3$. By Theorem 11 in \cite{AP}, $\Av_n(123, 132, 213)=\{1\ominus\sigma\mid\sigma\in\Av_{n-1}(123, 132, 213)\}\cup\{12\ominus\sigma\mid\sigma\in\Av_{n-2}(123, 132, 213)\}$, so \begin{align*}
\pk_n(123, 132, 213)&=\sum_{\sigma\in\Av_n(123, 132, 213)}\ell(\sigma)\\
&=\sum_{\sigma\in\Av_{n-1}(123, 132, 213)}\ell(1\ominus\sigma)+\sum_{\sigma\in\Av_{n-2}(123, 132, 213)}\ell(12\ominus\sigma)\\
&=\sum_{\sigma\in\Av_{n-1}(123, 132, 213)}\ell(\sigma)+\sum_{\sigma\in\Av_{n-2}(123, 132, 213)}2\ell(\sigma)\\
&=\pk_{n-1}(123, 132, 213)+2\pk_{n-2}(123, 132, 213).
\end{align*}
Solving this standard linear recurrence gives $\pk_n(123, 132, 213)=\frac232^n+\frac13(-1)^n$.
\end{proof}

\begin{theorem}\label{2.11}
$\pk_n(132, 213, 231)=\pk_n(132, 213, 312)=\pk_n(213, 231, 312)=\sum_{k=1}^nk!.$
\end{theorem}
\begin{proof}
By Theorem 12 in \cite{AP}, $\Av_n(132, 213, 231)=\{J_{k-1}\ominus I_{n-k+1}\mid k\in[n]\}$, so $\pk_n(132, 213, 231)=\sum_{k=1}^n\ell(J_{k-1}\ominus I_{n-k+1})=\sum_{k=1}^n(n-k+1)!=\sum_{k=1}^nk!$. 

By Theorem 13 in \cite{AP}, $\Av_n(132, 213, 312)=\{I_k\ominus J_{n-k}\mid k\in[n]\}$, so $\pk_n(132, 213, 312)=\sum_{k=1}^n\ell(I_k\ominus J_{n-k})=\sum_{k=1}^nk!$. 

By Theorem 13 in \cite{AP}, $\Av_n(213, 231, 312)=\{I_{k-1}\oplus J_{n-k+1}\mid k\in[n]\}$, so $\pk_n(213, 231, 312)=\sum_{k=1}^n\ell(I_{k-1}\oplus J_{n-k+1})=\sum_{k=1}^nk!$. 
\end{proof}

\begin{theorem}\label{2.12}
$\pk_n(132, 231, 312)=\sum_{k=1}^n\frac{n!}{k!}$.
\end{theorem}
\begin{proof}
By Theorem 12 in \cite{AP}, $\Av_n(132, 231, 312)=\{J_k\oplus I_{n-k}\mid k\in[n]\}$, so $\pk_n(132, 231, 312)=\sum_{k=1}^n\ell(J_k\oplus I_{n-k})=\sum_{k=1}^n\frac{n!}{k!}$.
\end{proof}

\begin{theorem}\label{2.13}
$\pk_n(132, 231, 321)=\pk_n(132, 312, 321)=\sum_{k=1}^n\frac{n!}{k}.$
\end{theorem}
\begin{proof}
By Theorem 14 in \cite{AP}, $\Av_n(132, 231, 321)=\{(1\ominus I_{k-1})\oplus I_{n-k}\mid k\in[n]\}$, so $\pk_n(132, 231, 321)\\=\sum_{k=1}^n\ell((1\ominus I_{k-1})\oplus I_{n-k})=\sum_{k=1}^n\frac{n!}{k}$. 

By Theorem 15 in \cite{AP}, $\Av_n(132, 312, 321)=\{(I_{k-1}\ominus 1)\oplus I_{n-k}\mid k\in[n]\}$, so $\pk_n(132, 312, 321)=\sum_{k=1}^n\ell((I_{k-1}\ominus 1)\oplus I_{n-k})=\sum_{k=1}^n\frac{n!}{k}$.
\end{proof}

\begin{theorem}\label{2.14}
$\pk_n(132, 213, 321)=\pk_n(213, 231, 321)=\sum_{k=1}^nk!(n-k)!$.
\end{theorem}
\begin{proof}
By Theorem 15 in \cite{AP}, $\Av_n(132, 213, 321)=\{I_k\ominus I_{n-k}\mid k\in[n]\}$, so $\pk_n(132, 213, 321)=\sum_{k=1}^n\ell(I_k\ominus I_{n-k})=\sum_{k=1}^nk!(n-k)!$. 

By Theorem 15 in \cite{AP}, $\Av_n(213, 231, 321)=\{I_{k-1}\oplus(1\ominus I_{n-k})\mid k\in[n]\}$, so $\pk_n(213, 231, 321)=\sum_{k=1}^n\ell(I_{k-1}\oplus(1\ominus I_{n-k}))=\sum_{k=1}^nk!(n-k)!$.
\end{proof}

\begin{theorem}\label{2.15}
$\pk_n(213, 312, 321)=(2n-1)(n-1)!$.
\end{theorem}
\begin{proof}
By Theorem 16 in \cite{AP}, $\Av_n(213, 312, 321)=\{I_{k-1}\oplus(I_{n-k}\ominus1)\mid k\in[n]\}$, so $\pk_n(213, 312, 321)\\=\sum_{k=1}^n\ell(I_{k-1}\oplus(I_{n-k}\ominus1))=\sum_{k=1}^{n-1}(n-1)!+n!=(2n-1)(n-1)!$.
\end{proof}

\begin{theorem}\label{2.16}
$\pk_n(231, 312, 321)=\sum_{k=0}^n(-1)^k\frac{n!}{k!}(n-k+1)$.
\end{theorem}
\begin{proof}
The cases when $n=1,2$ are easy to verify, so we assume $n\geq 3$. By Theorem 17 in \cite{AP}, $\Av_n(231, 312, 321)=\{\sigma\oplus1\mid\sigma\in\Av_{n-1}(231, 312, 321)\}\cup\{\sigma\oplus21\mid\sigma\in\Av_{n-2}(231, 312, 321)\}$, so \begin{align*}
\pk_n(231, 312, 321)&=\sum_{\sigma\in\Av_{n-1}(231, 312, 321)}\ell(\sigma\oplus1)+\sum_{\sigma\in\Av_{n-2}(231, 312, 321)}\ell(\sigma\oplus21)\\
&=n\pk_{n-1}(231, 312, 321)+(n-1)\pk_{n-2}(231, 312, 321).
\end{align*}
For brevity, let $p_n=\pk_n(231, 312, 321)$. Consider the exponential generating functions $P(x)=\sum_{n\geq0}p_n\frac{x^n}{n!}$ and $Q(x)=\sum_{n\geq0}p_n\frac{x^{n+2}}{n!(n+2)}$. Note that $Q'(x)=xP(x)$. Moreover, from the recurrence relation above, we have
\begin{align*}
\frac1xQ'(x)&=P(x)=1+x+\sum_{n\geq2}p_n\frac{x_n}{n!}\\
&=1+x+\sum_{n\geq2}p_{n-1}\frac{x^n}{(n-1)!}+\sum_{n\geq2}p_{n-2}\frac{x^n}{n(n-2)!}\\
&=1+x+x(P(x)-1)+Q(x)=1+Q'(x)+Q(x).
\end{align*}
Solving this differential equation and using $Q(0)=0$, we get $Q(x)=e^{-x}(1-x)^{-1}-1$ and $P(x)=\frac1xQ'(x)=e^{-x}(1-x)^{-2}$. Thus, $\pk_n(231, 312, 321)=p_n=n![x^n]P(x)=\sum_{k=0}^n(-1)^k\frac{n!}{k!}(n-k+1)$.
\end{proof}

\subsection{Avoiding two patterns}
\begin{center}
\begin{tabular}{|c|c|c|c|}
\hline  Patterns $\sigma_1,\sigma_2$  & $\pk_n(\sigma_1,\sigma_2)$, $1\leq n\leq8$ & OEIS & Result\\\hline
  123, 231   & \multirow{2}{*}{1, 3, 8, 17, 31, 51, 78, 113} & \multirow{2}{*}{A105163} & \multirow{2}{*}{Theorem \ref{2.17}}\\
  123, 312   & & &\\\hline
  123, 132   & 1, 3, 8, 21, 55, 144, 377, 987 & A001906 & Theorem \ref{2.18}\\\hline
  123, 213   & 1, 3, 7, 17, 41, 99, 239, 577 & A001333 & Theorem \ref{2.19}\\\hline
  132, 231   & \multirow{3}{*}{1, 3, 12, 60, 360, 2520, 20160, 181440} & \multirow{3}{*}{A001710} & \multirow{3}{*}{Theorem \ref{2.20}}\\
  132, 312   & & &\\
  231, 312   & & &\\\hline
  132, 213   & \multirow{2}{*}{1, 3, 11, 47, 231, 1303, 8431, 62391} & \multirow{2}{*}{A051296} & \multirow{2}{*}{Theorem \ref{2.21}}\\
  213, 231   & & &\\\hline
  132, 321   & 1, 3, 13, 68, 412, 2844, 22116, 191904 & New & Theorem \ref{2.22}\\\hline
  213, 321   & 1, 3, 12, 56, 300, 1836, 12768, 100224 & New & Theorem \ref{2.23}\\\hline
  213, 312   & 1, 3, 11, 49, 261, 1631, 11743, 95901 & A001339 & Theorem \ref{2.24}\\\hline
  231, 321   & 1, 3, 13, 71, 461, 3447, 29093, 273343 & A003319 & Theorem \ref{2.25}\\\hline
  312, 321   & 1, 3, 13, 73, 501, 4051, 37633, 394353 & A000262 & Theorem \ref{2.26}\\\hline
\end{tabular}
\end{center}

\begin{theorem}\label{2.17}
$\pk_n(123, 231)=\pk_n(123, 312)=\frac16n(n-1)(n+4)+1.$   
\end{theorem}
\begin{proof}
By Theorem 18 in \cite{AP}, $\Av_n(123,231)=\{J_{n-a-b}\ominus(J_a\oplus J_{b})\mid a,b\geq 1, a+b\leq n\}\cup\{J_n\}$, so $\pk_n(123, 231)=\sum_{a=1}^{n-1}\sum_{b=1}^{n-a}\ell(J_{n-a-b}\ominus(J_a\oplus J_{b}))+\ell(J_n)=\sum_{a=1}^{n-1}\sum_{b=1}^{n-a}(a+1)+1=\frac16(n-1)n(n+4)+1$.  

By Theorem 19 in \cite{AP}, $\Av_n(123,312)=\{(J_a\oplus J_{b})\ominus J_{n-a-b}\mid a,b\geq 1, a+b\leq n\}\cup\{J_n\}$, so $\pk_n(123, 312)=\sum_{a=1}^{n-1}\sum_{b=1}^{n-a}\ell((J_a\oplus J_{b})\ominus J_{n-a-b}))+\ell(J_n)=\sum_{a=1}^{n-1}\sum_{b=1}^{n-a}(a+1)+1=\frac16(n-1)n(n+4)+1$.
\end{proof}

\begin{theorem}\label{2.18}
$\pk_n(123, 132)=\frac1{\sqrt5}(\frac{3+\sqrt5}2)^n-\frac1{\sqrt5}(\frac{3-\sqrt 5}2)^n$.
\end{theorem}
\begin{proof}
By Theorem 20 in \cite{AP}, $\Av_n(123, 132)=\{(J_{k-1}\oplus1)\ominus\sigma\mid k\in[n], \sigma\in\Av_{n-k}(123,132)\}$, so \begin{align*}
\pk_n(123, 132)&=\sum_{k=1}^n\sum_{\sigma\in\Av_{n-k}(123,132)}\ell((J_{k-1}\oplus1)\ominus\sigma)\\
&=\sum_{k=1}^n\sum_{\sigma\in\Av_{n-k}(123,132)}k\ell(\sigma)=\sum_{k=1}^{n}k\pk_{n-k}(123, 132).
\end{align*}
Let $p_n=\pk_n(123, 132)$ and consider the generating function $P(x)=\sum_{n\geq0}p_nx^n$. The recurrence relation above implies that $P(x)=1+\sum_{n\geq1}\sum_{k=1}^nkp_{n-k}x^n=1+\sum_{k\geq1}kx^k\sum_{n\geq k}p_{n-k}x^{n-k}=1+\frac{x}{(1-x)^2}P(x)$. Thus, $P(x)=\frac{(1-x)^2}{1-3x+x^2}=1+\frac{x}{x^2-3x+1}$. It follows that for $n\geq1$, $\pk_n(123, 132)=p_n=[x^{n-1}](x^2-3x+1)^{-1}=\frac1{\sqrt5}(\frac{3+\sqrt5}2)^n-\frac1{\sqrt5}(\frac{3-\sqrt 5}2)^n$.
\end{proof}

\begin{theorem}\label{2.19}
$\pk_n(123, 213)=\frac12(1-\sqrt2)^n+\frac12(1+\sqrt2)^n$.   
\end{theorem}
\begin{proof}
By Theorem 21 in \cite{AP}, $\Av_n(123, 213)=\{\sigma\ominus(1\oplus J_{k-1})\mid k\in[n], \sigma\in\Av_{n-k}(123, 213)\}$, so
\begin{align*}
\pk_n(123, 213)&=\sum_{k=1}^n\sum_{\sigma\in\Av_{n-k}(123,213)}\ell(\sigma\ominus(1\oplus J_{k-1}))\\
&=\sum_{\sigma\in\Av_{n-1}(123,213)}\ell(\sigma)+\sum_{k=2}^n\sum_{\sigma\in\Av_{n-k}(123,213)}2\ell(\sigma)\\
&=\pk_{n-1}(123, 213)+\sum_{k=2}^{n}2\pk_{n-k}(123, 213).
\end{align*}
Let $p_n=\pk_n(123, 213)$, it follows that for $n\geq2$
\begin{align*}
p_{n+1}&=p_n+\sum_{k=2}^{n+1}p_{n+1-k}=p_n+2p_{n-1}+\sum_{k=2}^{n}2p_{n-k}\\
&=p_n+2p_{n-1}+p_n-p_{n-1}=2p_n+p_{n-1}.
\end{align*}
Solving this standard linear recurrence using the easily verified initial conditions $p_1=1,p_2=3$ gives $\pk_n(123, 213)=p_n=\frac12(1-\sqrt2)^n+\frac12(1+\sqrt2)^n$. 
\end{proof}

\begin{theorem}\label{2.20}
$\pk_n(132, 231)=\pk_n(132, 312)=\pk_n(231, 312)=\frac12(n+1)!$.   
\end{theorem}
\begin{proof}
By Theorem 22 in \cite{AP}, $\Av_n(132, 231)=\{J_{[n]\setminus(S\cup\{1\})}1I_S\mid S\subset[n]\setminus\{1\}\}$, so $\pk_n(132, 231)=\sum_{S\subset[n]\setminus\{1\}}\ell(J_{[n]\setminus(S\cup\{1\})}1I_S)=\sum_{S\subset[n]\setminus\{1\}}\prod_{i\in S}i=\prod_{i=2}^n(1+i)=\frac12(n+1)!$.

By Theorem 23 in \cite{AP}, $\Av_n(132, 312)=\{(\sigma\oplus1)\ominus J_{k-1}\mid k\in[n],\sigma\in\Av_{n-k}(132, 312)\}$, so 
\begin{align*}
\pk_n(132, 312)&=\sum_{k=1}^n\sum_{\sigma\in\Av_{n-k}(132, 312)}\ell((\sigma\oplus1)\ominus J_{k-1})=\sum_{k=1}^n\sum_{\sigma\in\Av_{n-k}(132, 312)}(n-k+1)\ell(\sigma)\\
&=\sum_{k=1}^n(n-k+1)\pk_{n-k}(132, 312)=\sum_{k=0}^{n-1}(k+1)\pk_k(132,312).
\end{align*}
It follows that $\pk_{n+1}(132, 312)-\pk_n(132, 312)=(n+1)\pk_n(132, 312)$ for $n\geq1$, so $\pk_{n+1}(132, 312)=(n+2)\pk_n(132, 312)$. An easy induction gives $\pk_n(132, 312)=\frac12(n+1)!$.

By Theorem 23 in \cite{AP}, $\Av_n(231, 312)=\{\sigma\oplus J_k\mid k\in[n],\sigma\in\Av_{n-k}(231, 312)\}$, so \begin{align*}
\pk_n(231, 312)&=\sum_{k=1}^n\sum_{\sigma\in\Av_{n-k}(231, 312)}\ell(\sigma\oplus J_k)=\sum_{k=1}^n\sum_{\sigma\in\Av_{n-k}(231, 312)}(n-k+1)\ell(\sigma)\\
&=\sum_{k=1}^n(n-k+1)\pk_{n-k}(231, 312)=\sum_{k=0}^{n-1}(k+1)\pk_k(231,312).
\end{align*}
Since this is the same recurrence relation as the case above, and it is easy to verify the initial values match, we have $\pk_n(231, 312)=\frac12(n+1)!$ as well.
\end{proof}

\begin{theorem}\label{2.21}
$$\pk_n(132, 213)=\pk_n(213, 231)=\sum_{k=1}^n\sum_{\substack{a_1+\cdots a_k=n\\a_1,\cdots,a_k\geq1}}a_1!\cdots a_k!.$$   
\end{theorem}
\begin{proof}
By Theorem 23 in \cite{AP}, $\Av_n(132, 213)=\{I_k\ominus\sigma\mid k\in[n],\sigma\in\Av_{n-k}(132, 213)\}$, so
\begin{align*}
\pk_n(132, 213)&=\sum_{k=1}^n\sum_{\sigma\in\Av_{n-k}(132, 213)}\ell(I_k\ominus\sigma)\\
&=\sum_{k=1}^n\sum_{\sigma\in\Av_{n-k}(132, 213)}k!\ell(\sigma)=\sum_{k=1}^nk!\pk_{n-k}(132, 213).
\end{align*}
The result then follows from induction on $n$.

By Theorem 23 in \cite{AP}, $\Av_n(213, 231)=\{I_{k-1}\oplus(1\ominus\sigma)\mid k\in[n],\sigma\in\Av_{n-k}(213, 231)\}$, so \begin{align*}
\pk_n(213, 231)&=\sum_{k=1}^n\sum_{\sigma\in\Av_{n-k}(213, 231)}\ell(I_{k-1}\oplus(1\ominus\sigma))\\
&=\sum_{k=1}^n\sum_{\sigma\in\Av_{n-k}(213, 231)}k!\ell(\sigma)=\sum_{k=1}^nk!\pk_{n-k}(213, 231).
\end{align*}
Again, the result follows from induction on $n$.
\end{proof}

\begin{theorem}\label{2.22}
$$\pk_n(132, 321)=n!+\sum_{\substack{a+b\leq n\\a,b\geq1}}\frac{n!}{\binom{a+b}{a}}.$$   
\end{theorem}
\begin{proof}
By Theorem 24 in \cite{AP}, $\Av_n(132, 321)=\{I_n\}\cup\{(I_a\ominus I_b)\oplus I_{n-a-b}\mid a,b\geq 1, a+b\leq n\}$, so 
\begin{align*}
\pk_n(132, 321)&=\ell(I_n)+\sum_{a=1}^{n-1}\sum_{b=1}^{n-a}\ell((I_a\ominus I_b)\oplus I_{n-a-b})\\
&=n!+\sum_{a=1}^{n-1}\sum_{b=1}^{n-a}a!b!\frac{n!}{(a+b)!}=n!+\sum_{a=1}^{n-1}\sum_{b=1}^{n-a}\frac{n!}{\binom{a+b}{a}},    
\end{align*}
as claimed.
\end{proof}

\begin{theorem}\label{2.23}
$$\pk_n(213, 321)=n!+n!\sum_{k=1}^{n-1}\frac{k}{\binom{n}{k}}.$$   
\end{theorem}
\begin{proof}
By Theorem 25 in \cite{AP}, $\Av_n(213, 321)=\{I_n\}\cup\{I_{n-a-b}\oplus(I_a\ominus I_b)\mid a,b\geq 1, a+b\leq n\}$, so 
\begin{align*}
\pk_n(213, 321)&=\ell(I_n)+\sum_{a=1}^{n-1}\sum_{b=1}^{n-a}\ell(I_{n-a-b}\oplus(I_a\ominus I_b))\\
&=n!+\sum_{a=1}^{n-1}\sum_{b=1}^{n-a}(n-b)!b!=n!+\sum_{k=1}^{n-1}k(n-k)!k!=n!+n!\sum_{k=1}^{n-1}\frac{k}{\binom{n}{k}},
\end{align*}
as claimed.
\end{proof}

\begin{theorem}\label{2.24}
$\pk_n(213, 312)=\sum_{k=0}^{n-1}\binom{n-1}{k}(k+1)!$.   
\end{theorem}
\begin{proof}
By Theorem 26 in \cite{AP}, $\Av_n(213, 312)=\{I_SnJ_{[n-1]\setminus S}\mid S\subset[n-1]\}$, so $\pk_n(213, 312)=\sum_{S\subset[n-1]}\ell(I_SnJ_{[n-1]\setminus S})=\sum_{S\subset[n-1]}(|S|+1)!=\sum_{k=0}^{n-1}\binom{n-1}{k}(k+1)!$. 
\end{proof}

\begin{theorem}\label{2.25}
$$\pk_n(231, 321)=\sum_{k=1}^{n+1}(-1)^{k+1}\sum_{\substack{a_1+\cdots+a_k=n+1\\a_1,\cdots,a_k\geq1}}a_1!\cdots a_k!.$$   
\end{theorem}
\begin{proof}
By Theorem 27 in \cite{AP}, $\Av_n(231, 321)=\{\sigma\oplus(1\ominus I_{k-1})\mid k\in[n], \sigma\in\Av_{n-k}(231, 321)\}$, so 
\begin{align*}
\pk_n(231, 321)&=\sum_{k=1}^n\sum_{\sigma\in\Av_{n-k}(231, 321)}\ell(\sigma\oplus(1\ominus I_{k-1}))=\sum_{k=1}^n\sum_{\sigma\in\Av_{n-k}(231, 321)}\ell(\sigma)(n-k+1)(k-1)!\\
&=\sum_{k=1}^n\pk_{n-k}(231, 321)(n-k+1)(k-1)!=\sum_{k=0}^{n-1}\pk_k(231, 321)(k+1)(n-k-1)!.    
\end{align*}
For brevity, let $p_n=\pk_n(231, 321)$. We use induction on $n$ to show that $p_n=(n+1)!-\sum_{k=0}^{n-1}p_k(n-k)!$ for all $n\geq1$. The case when $n=1$ is easily verified. For $n\geq 2$, by the recurrence relation above and induction hypothesis, we have
\begin{align*}
p_n+\sum_{k=0}^{n-1}p_k(n-k)!&=\sum_{k=0}^{n-1}p_k(k+1)(n-k-1)!+\sum_{k=0}^{n-1}p_k(n-k)!\\
&=\sum_{k=0}^{n-1}p_k(n-k-1)!(k+1+n-k)=(n+1)\sum_{k=0}^{n-1}p_k(n-k-1)!\\
&=(n+1)\left(p_{n-1}+\sum_{k=0}^{n-2}p_k(n-k-1)!\right)=(n+1)n!=(n+1)!,
\end{align*}
as required. The result then follows from induction on $n$.
\end{proof}

\begin{theorem}\label{2.26}
$$\pk_n(312, 321)=n!\sum_{k=1}^n\frac{\binom{n-1}{k-1}}{k!}.$$   
\end{theorem}
\begin{proof}
By Theorem 28 in \cite{AP}, $\Av_n(312, 321)=\{\sigma\oplus(I_{k-1}\ominus1)\mid k\in[n],\sigma\in\Av_{n-k}(312, 321)\}$, so 
\begin{align*}
\pk_n(312, 321)&=\sum_{k=1}^n\sum_{\sigma\in\Av_{n-k}(312, 321)}\ell(\sigma\oplus(I_{k-1}\ominus1))\\
&=\sum_{\sigma\in\Av_{n-1}(312, 321)}n\ell(\sigma)+\sum_{k=2}^n\sum_{\sigma\in\Av_{n-k}(312, 321)}\ell(\sigma)\frac{(n-1)!}{(n-k)!}\\
&=n\pk_{n-1}(312, 321)+\sum_{k=2}^n\pk_{n-k}(312, 321)\frac{(n-1)!}{(n-k)!}.    
\end{align*}
It follows that
\begin{align*}
\pk_{n+1}(312, 321)&=(n+1)\pk_n(312, 321)+\sum_{k=2}^{n+1}\pk_{n+1-k}(312, 321)\frac{n!}{(n+1-k)!}\\
&=(n+1)\pk_n(312, 321)+n\pk_{n-1}(312, 321)+n\sum_{k=2}^{n}\pk_{n-k}(312, 321)\frac{(n-1)!}{(n-k)!}\\
&=(2n+1)\pk_n(312, 321)-n(n-1)\pk_{n-1}(312, 321).
\end{align*}
For simplicity, let $p_n=\pk_n(312, 321)$, and consider the exponential generating function $P(x)=\sum_{n\geq0}\frac{p_n}{n!}x^n$. From the recurrence relation above and $p_0=p_1=1$, $p_2=3$, we have that 
\begin{align*}
P'(x)&=\sum_{n\geq0}\frac{p_{n+1}}{n!}x^n=1+3x+\sum_{n\geq2}\frac{p_{n+1}}{n!}x^n\\
&=1+3x+\sum_{n\geq2}\frac{(2n+1)p_n-n(n-1)p_{n-1}}{n!}x^n\\
&=1+3x+2x\sum_{n\geq1}\frac{p_{n+1}}{n!}x^n+\sum_{n\geq 2}\frac{p_n}{n!}x^n-x^2\sum_{n\geq0}\frac{p_{n+1}}{n!}x^n\\
&=1+3x+2x(P'(x)-1)+(P(x)-1-x)-x^2P'(x).
\end{align*}
This gives the differential equation $(x-1)^2P'(x)=P(x)$. Solving this and using $P(0)=1$, we get $P(x)=e^{\frac x{1-x}}$. Thus, 
\begin{align*}
\pk_n(312, 321)&=p_n=n![x^n]P(x)=n![x^n]e^{\frac x{1-x}}\\
&=n![x^n]\sum_{k\geq0}\frac{x^k}{k!(1-x)^k}=n!\sum_{k=1}^n[x^{n-k}]\frac1{k!(1-x)^k}=n!\sum_{k=1}^n\frac{\binom{n-1}{k-1}}{k!},
\end{align*}
as required.
\end{proof}

\subsection{Avoiding one pattern}
\begin{center}
\begin{tabular}{|c|c|c|c|}
\hline  Pattern $\sigma$  & $\pk_n(\sigma)$, $1\leq n\leq8$ & OEIS & Result \\\hline
  132   & \multirow{2}{*}{1, 3, 14, 85, 621, 5236, 49680, 521721} & \multirow{2}{*}{A088716} & \multirow{2}{*}{Theorem \ref{132,231}}\\
  231   & & &\\\hline
  123   & 1, 3, 10, 37, 146, 602, 2563, 11181 & A109081 & Theorem \ref{123}\\\hline
  213   & 1, 3, 13, 69, 421, 2867, 21477, 175769 & A088368 & Theorem \ref{213}\\\hline
  312   & 1, 3, 14, 87, 669, 6098, 64050, 759817 & A132624 & Theorem \ref{312}\\\hline
  321   & 1, 3, 15, 102, 860, 8553, 97331, 1241900 & New & Theorem \ref{321}\\\hline
\end{tabular}
\end{center}

\begin{theorem}\label{132,231}
$\pk_n(132)=\pk_n(231)=p_n$, where $p_0=1$ and $p_n=\sum_{k=1}^nkp_{k-1}p_{n-k}$ for all $n\geq1$. It follows that the generating function $P(x)=\sum_{n\geq 0}p_nx^n$ satisfies the differential equation $x^2P(x)P'(x)+x(P(x))^2-P(x)+1=0$.
\end{theorem}
\begin{proof}
For any $\sigma\in\Av_n(132)$, if $\sigma(i)=n$, then $\sigma(j)>\sigma(j')$ for all $1\leq j<i<j'\leq n$. From this, it follows that $\Av_n(132)=\{(\sigma_1\oplus1)\ominus\sigma_2\mid k\in[n],\sigma_1\in\Av_{k-1}(132),\sigma_2\in\Av_{n-k}(132)\}$. Therefore, 
\begin{align*}
\pk_n(132)&=\sum_{k=1}^n\sum_{\sigma_1\in\Av_{k-1}(132)}\sum_{\sigma_2\in\Av_{n-k}(132)}\ell((\sigma_1\oplus1)\ominus\sigma_2)\\
&=\sum_{k=1}^n\sum_{\sigma_1\in\Av_{k-1}(132)}\sum_{\sigma_2\in\Av_{n-k}(132)}k\ell(\sigma_1)\ell(\sigma_2)=\sum_{k=1}^nk\pk_{k-1}(132)\pk_{n-k}(132)    
\end{align*}

Similarly, for any $\sigma\in\Av_n(231)$, if $\sigma(i)=n$, then $\sigma(j)<\sigma(j')$ for all $1\leq j<i<j'\leq n$. From this, it follows that $\Av_n(231)=\{\sigma_1\oplus(1\ominus\sigma_2)\mid k\in[n],\sigma_1\in\Av_{k-1}(231),\sigma_2\in\Av_{n-k}(231)\}$. Therefore,
\begin{align*}
\pk_n(231)&=\sum_{k=1}^n\sum_{\sigma_1\in\Av_{k-1}(231)}\sum_{\sigma_2\in\Av_{n-k}(231)}\ell(\sigma_1\oplus(1\ominus\sigma_2))\\
&=\sum_{k=1}^n\sum_{\sigma_1\in\Av_{k-1}(231)}\sum_{\sigma_2\in\Av_{n-k}(231)}k\ell(\sigma_1)\ell(\sigma_2)=\sum_{k=1}^nk\pk_{k-1}(231)\pk_{n-k}(231)    
\end{align*}

Since $\pk_n(132)$ and $\pk_n(231)$ satisfy the same recurrence relation, and it is easy to verify that their initial values agree, their values must coincide for all $n\geq0$. Let their common values be $p_n$, it follows from the recurrence relation that the generating function $P(x)=\sum_{n\geq 0}p_nx^n$ satisfies the differential equation
\begin{align*}
P(x)&=1+\sum_{n\geq1}p_nx^n=1+\sum_{n\geq1}\sum_{k=1}^nkp_{k-1}p_{n-k}x^n\\
&=1+x\left(\sum_{a_1\geq0}(a_1+1)p_{a_1}x^{a_1}\right)\left(\sum_{a_2\geq0}p_{a_2}x^{a_2}\right)=1+x(xP(x))'P(x),   
\end{align*}
which is equivalent to the differential equation $x^2P(x)P'(x)+x(P(x))^2-P(x)+1=0$.
\end{proof}

The differential equation we obtained in Theorem \ref{132,231} is an example of a Chini's Equation, which has no known explicit solution. It remains open whether we can find a more explicit formula for $p_n$.

The next few results require some preparation. For each $\sigma\in\{123, 213, 312, 321\}$, we first use a bijection of Kratthenthaler in \cite{K} or its variants, which will be described in the proofs later, between $\Av_n(\sigma)$ and the set $\mathcal{C}_n$ of Catalan paths of length $2n$ to express $\pk_n(\sigma)$ as a sum over $\mathcal{C}_n$. For each $C\in\mathcal{C}_n$, let $\mathbf{u}(C)$ be the sequence of integers recording in order the lengths of each block of consecutive up-steps in $C$. For example, for the Catalan path $C=UDUUDUDD\in\mathcal{C}_4$, $\mathbf{u}(C)=(1,2,1)$. All terms in the sum over all $C\in\mathcal{C}_n$ will be expressions in terms of entries of $\mathbf{u}(C)$.

For $k\in[n]$, let $\mathcal{C}_{n,k}$ be the set of Catalan paths of length $2n$, which begins with a block of $k$ up-steps, or equivalently, $\mathcal{C}_{n,k}=\{C\in\mathcal{C}_n\mid \mathbf{u}(C)_1=k\}$. For each $C\in\mathcal{C}_{n,k}$, $C$ has a unique decomposition $C=U_1\cdots U_kD_1C_1\cdots D_kC_k$, where $U_1,\cdots,U_k$ are the first $k$ consecutive up-steps in $C$, $D_1,\cdots,D_k$ are down-steps, and $C_1,\cdots,C_k$ are Catalan paths, possibly of length 0. This decomposition is unique because for all $i\in[k]$, $D_i$ is the first down-step in $C$ that goes down from height $k-i+1$ to height $k-i$. We call this the \textit{canonical decomposition} of $C$. Also, note that $\mathbf{u}(C)$ is obtained by attaching $\mathbf{u}(C_1),\cdots,\mathbf{u}(C_k)$ together in order, and adding an entry of $k$ to the front. Moreover, for any $n\geq k\geq 1$ and any Catalan paths $C_1,\cdots,C_k$, possibly of length 0, with total length $2(n-k)$, $C=U_1\cdots U_kD_1C_1\cdots D_kC_k$ is a Catalan path in $\mathcal{C}_{n,k}$. Hence, the map $C\mapsto(C_1,\cdots,C_k)$ is a bijection.

Finally, to extract coefficients from generating functions satisfying certain functional equations, the well-known Lagrange's Implicit Function Theorem below is very useful, and a proof of it can be found, for example, in \cite{GJ}.
\begin{lemma}[Lagrange's Implicit Function Theorem]\label{lift}
Let $P(x), \phi(x), \psi(x)$ be formal power series with coefficients in $\mathbb{C}$, such that $[x^0]\phi(x)\not=0$ and $P(x)=x\phi(P(x))$. Then for $n\geq 1$,
$$[x^n]\psi(P(x))=\frac1{n}[x^{n-1}]\psi'(x)(\phi(x))^n.$$
In particular, if we set $\psi(x)=x$, then $[x^0]P(x)=0$ and for $n\geq 1$,
$$[x^n]P(x)=\frac1{n}[x^{n-1}](\phi(x))^n.$$
\end{lemma}

\begin{theorem}\label{123}
$$\pk_n(123)=\frac1{n+1}\sum_{k=1}^n\binom{n+1}{k}\binom{n+k-1}{2k-1}.$$
\end{theorem}
\begin{proof}
For simplicity, let $p_n=\pk_n(123)$, and let $P(x)=\sum_{n\geq0}p_nx^n$ be its generating function.

In \cite{K}, Krattenthaler showed that the map $\Omega_n:\Av_n(123)\to\mathcal{C}_n$ described below is a bijection. For each permutation $\sigma=\sigma(1)\cdots\sigma(n)\in\Av_n(123)$ and $i\in[n]$, we say that $\sigma(i)$ is a \textit{right-to-left maximum} if $\sigma(i)>\sigma(j)$ for all $i<j\leq n$. Then $\sigma$ uniquely decomposes as $\sigma=s_1x_1\cdots s_kx_k$, where $x_1,\cdots,x_k$ are the right-to-left maxima of $\sigma$, and $s_1,\cdots,s_k$ are possibly empty segments. Let $\Omega_n(\sigma)$ be the Catalan path of length $2n$ obtained by, for each $i\in[k]$ in order, drawing a block of $|s_i|+1$ consecutive up-steps followed by a block of $x_i-x_{i+1}$ consecutive down-steps, where $|s_i|$ is the length of the segment $s_i$ and we view $x_{k+1}=0$.  

Note that from the definition of right-to-left maxima, $x_1>\cdots>x_k$, $x_i$ is larger than every entry in $s_{i+1}$ for all $i\in[k-1]$, and $x_i$ is also larger than every entry in $s_i$ for all $i\in[k]$. Since $\sigma$ avoids the pattern 123, each segment $s_i$ is decreasing. Thus, we have $\ell(\sigma)=\prod_{i=1}^k(|s_i|+1)$. Moreover, note that $|s_i|+1=\mathbf{u}(\Omega_n(\sigma))_i$. Therefore, since $\Omega_n:\Av_n(123)\to\mathcal{C}_n$ is a bijection, we have
$$p_n=\pk_n(123)=\sum_{\sigma\in\Av_n(123)}\ell(\sigma)=\sum_{C\in\mathcal{C}_n}\prod_{i=1}^{|\mathbf{u}(C)|}\mathbf{u}(C)_i.$$

Let $W(C)=\prod_{i=1}^{|\mathbf{u}(C)|}\mathbf{u}(C)_i$. For all $k\in[n]$ and $C\in\mathcal{C}_{n,k}$, let $C=U_1\cdots U_kD_1C_1\cdots D_kC_k$ be the canonical decomposition of $C$. Then we have $W(C)=k\prod_{j=1}^kW(C_j)$. Moreover, recall that the map $C\mapsto(C_1,\cdots,C_k)$ is a bijection, so we have for $n\geq1$
\begin{align*}
p_n&=\sum_{C\in\mathcal{C}_n}W(C)=\sum_{k=1}^nk\sum_{\substack{C_1\in\mathcal{C}_{a_1},\cdots,C_k\in\mathcal{C}_{a_k}\\a_1+\cdots+a_k=n-k\\a_1,\cdots,a_k\geq0}}\prod_{j=1}^kW(C_j)=\sum_{k=1}^nk\sum_{\substack{a_1+\cdots+a_k=n-k\\a_1,\cdots,a_k\geq0}}\prod_{j=1}^kp_{a_j}.
\end{align*}
It follows that 
\begin{align*}
P(x)&=1+\sum_{n\geq1}p_nx^n=1+\sum_{n\geq1}\left(\sum_{k=1}^nk\sum_{\substack{a_1+\cdots+a_k=n-k\\a_1,\cdots,a_k\geq0}}\prod_{j=1}^kp_{a_j}\right)x^n\\
&=1+\sum_{k\geq1}kx^k\left(\sum_{a_1,\cdots,a_k\geq0}\prod_{j=1}^np_{a_j}x^{a_j}\right)=1+\sum_{k\geq1}kx^k(P(x))^k=1+\frac{xP(x)}{(1-xP(x))^2}.
\end{align*}
Now let $\overline{P}(x)=xP(x)$, then $\overline{P}(x)=x(1+\overline{P}(x)(1-\overline{P}(x))^{-2})$. Thus, by Lemma \ref{lift}, we have for $n\geq 1$,
\begin{align*}
p_n&=[x^n]P(x)=[x^{n+1}]\overline{P}(x)=\frac1{n+1}[x^n]\left(1+\frac{x}{(1-x)^2}\right)^{n+1}\\
&=\frac{1}{n+1}\sum_{k=1}^n\binom{n+1}{k}[x^{n-k}]\frac{1}{(1-x)^{2k}}\\
&=\frac{1}{n+1}\sum_{k=1}^n\binom{n+1}{k}\binom{n+k-1}{2k-1},
\end{align*}
as required.
\end{proof}

\begin{theorem}\label{213}
$$\pk_n(213)=\frac{1}{n+1}\sum_{\substack{a_1+\cdots+a_{n+1}=n\\a_1,\cdots,a_{n+1}\geq0}}a_1!\cdots a_{n+1}!.$$
\end{theorem}
\begin{proof}
For simplicity, let $p_n=\pk_n(213)$, and let $P(x)=\sum_{n\geq0}p_nx^n$ be its generating function. 

Similar to the proof of Theorem \ref{123}, the map $\Omega_n:\Av_n(213)\to\mathcal{C}_n$ described below is a bijection. For each permutation $\sigma\in\Av_n(213)$, $\sigma$ uniquely decomposes as $\sigma=s_1x_1\cdots s_kx_k$, where $x_1,\cdots,x_k$ are the right-to-left maxima of $\sigma$, and $s_1,\cdots,s_k$ are possibly empty segments. Let $\Omega_n(\sigma)$ be the Catalan path of length $2n$ obtained by, for each $i\in[k]$ in order, drawing a block of $|s_i|+1$ consecutive up-steps followed by a block of $x_i-x_{i+1}$ consecutive down-steps, where we view $x_{k+1}=0$.  

Note that from the definition of right-to-left maxima, $x_1>\cdots>x_k$, $x_i$ is larger than every entry in $s_{i+1}$ for all $i\in[k-1]$, and $x_i$ is also larger than every entry in $s_i$ for all $i\in[k]$. Since $\sigma$ avoids the pattern 213, each segment $s_i$ is increasing. Thus, we have $\ell(\sigma)=\prod_{i=1}^k(|s_i|+1)!$. Moreover, we have again that $|s_i|+1=\mathbf{u}(\Omega_n(\sigma))_i$, so
$$p_n=\pk_n(213)=\sum_{\sigma\in\Av_n(213)}\ell(\sigma)=\sum_{C\in\mathcal{C}_n}\prod_{i=1}^{|\mathbf{u}(C)|}(\mathbf{u}(C)_i)!.$$

Let $W(C)=\prod_{i=1}^{|\mathbf{u}(C)|}(\mathbf{u}(C)_i)!$. For all $k\in[n]$ and $C\in\mathcal{C}_{n,k}$, let $C=U_1\cdots U_kD_1C_1\cdots D_kC_k$ be the canonical decomposition of $C$. Then we have $W(C)=k!\prod_{j=1}^kW(C_j)$. Moreover, recall that the map $C\mapsto(C_1,\cdots,C_k)$ is a bijection, so we have
\begin{align*}
p_n&=\sum_{k=1}^nk!\sum_{\substack{C_1\in\mathcal{C}_{a_1},\cdots,C_k\in\mathcal{C}_{a_k}\\a_1+\cdots+a_k=n-k\\a_1,\cdots,a_k\geq0}}\prod_{j=1}^kW(C_j)=\sum_{k=1}^nk!\sum_{\substack{a_1+\cdots+a_k=n-k\\a_1,\cdots,a_k\geq0}}\prod_{j=1}^kp_{a_j}.
\end{align*}
It follows that 
\begin{align*}
P(x)&=1+\sum_{n\geq1}p_nx^n=1+\sum_{n\geq1}\left(\sum_{k=1}^nk!\sum_{\substack{a_1+\cdots+a_k=n-k\\a_1,\cdots,a_k\geq0}}\prod_{j=1}^np_{a_j}\right)x^n\\
&=1+\sum_{k\geq1}k!x^k\left(\sum_{a_1,\cdots,a_k\geq0}\prod_{j=1}^np_{a_j}x^{a_j}\right)=1+\sum_{k\geq1}k!x^k(P(x))^k.
\end{align*}
Now let $\overline{P}(x)=xP(x)$, then $\overline{P}(x)=x\sum_{k\geq0}k!(\overline{P}(x))^k$. Thus, by Lemma \ref{lift}, we have for $n\geq 0$,
\begin{align*}
p_n&=[x^n]P(x)=[x^{n+1}]\overline{P}(x)=\frac1{n+1}[x^n]\left(\sum_{k\geq0}k!x^k\right)^{n+1}\\
&=\frac{1}{n+1}\sum_{\substack{a_1+\cdots+a_{n+1}=n\\a_1,\cdots,a_{n+1}\geq0}}a_1!\cdots a_{n+1}!,
\end{align*}
as required.
\end{proof}

As will be seen below, when we express $\pk_n(312)$ and $\pk_n(321)$ as sums over $\mathcal{C}_n$, similar to the proofs of Theorem \ref{123}, \ref{213} above, the terms in these sums will be products of cumulative sums of entries in $\mathbf{u}(C)$, instead of just products of individual entries in $\mathbf{u}(C)$. As a result, our method above using canonical decomposition of Catalan paths does not translate. Instead, we use the following observation to obtain recursive formulas for $\pk_n(312)$ and $\pk_n(321)$.

For $n>k\geq1$ and $C\in\mathcal{C}_{n,k}$, we have $|\mathbf{u}(C)|\geq2$. Let $i'\in[k]$ be the length of the first block of down-steps in $C$ and let $C'$ be obtained by deleting the last $i'$ up-steps in the first block of up-steps in $C$, as well as the first $i'$ down-steps following it. It is easy to verify that $C'$ is a Catalan path of length $2(n-i')$, whose first block of up-steps have length $j=\mathbf{u}(C)_1+\mathbf{u}(C)_2-i'$, which is between $k-i'+1$ and $n-i'$. We say that $C'$ is obtained by 
\textit{deleting the first peak} of $C$. It is also easy to verify that this process is reversible.

\begin{theorem}\label{312}
$$\pk_n(312)=\sum_{C\in\mathcal{C}_n}\prod_{i=2}^{|\mathbf{u}(C)|}\left(1+\sum_{j=i}^{|\mathbf{u}(C)|}\mathbf{u}(C)_j\right).$$
Consequently, $\pk_n(312)=\sum_{k=1}^np_{n,k}$, where
$$p_{n,k}=\begin{cases}
1, &\text{if }k=n,\\
\displaystyle (n-k+1)\sum_{i=n-k}^{n-1}\sum_{j=k+1-n+i}^ip_{i,j}, &\text{if }1\leq k\leq n-1.\\
\end{cases}$$
Furthermore, the sequence $\pk_n(312)$ satisfies $$\frac{x}{1-x}=\sum_{n=1}^\infty\pk_n(312)\frac{x^n(1-x)^n}{\prod_{\ell=1}^n(1+\ell x)}.$$
\end{theorem}
\begin{proof}
For brevity, let $p_n=\pk_n(312)$. 

Similar to the proof of Theorem \ref{213}, the map $\Omega_n:\Av_n(312)\to\mathcal{C}_n$ described below is a bijection. Each $\sigma\in\Av_n(312)$ can be uniquely decomposed as $x_ks_k\cdots x_1s_1$, where $x_k<\cdots<x_1$ are the left-to-right maxima of $\sigma$, and $s_k,\cdots,s_1$ are possibly empty segments. Define $\Omega_n(\sigma)$ to be the Catalan path of length $2n$ obtained by, for each $i\in[k]$ in order, drawing a block of $|s_i|+1$ up-steps and followed by a block of $x_i-x_{i+1}$ down-steps, where we view $x_{k+1}$ as 0.

Note that from the definition, for all $i\in[k]$, $x_i$ is larger than every entry before it and every entry in the segment $s_i$. Moreover, each segment $s_i$ is decreasing as $\sigma$ avoids the pattern 312. It follows that $\ell(\sigma)=\prod_{i=1}^{k-1}\left(1+\sum_{j=i+1}^{k}(|s_j|+1)\right)$. Since $|s_i|+1=\mathbf{u}(\Omega_n(\sigma))_i$ and $\Omega_n$ is a bijection, it follows that
$$\pk_n(312)=p_n=\sum_{\sigma\in\Av_n(312)}\ell(\sigma)=\sum_{C\in\mathcal{C}_n}\prod_{i=2}^{|\mathbf{u}(C)|}\left(1+\sum_{j=i}^{|\mathbf{u}(C)|}\mathbf{u}(C)_j\right),$$
proving the first part. The second and third part follows by setting $m=1$ in Theorem \ref{metasylvestermultipark} below.
\end{proof}

\begin{theorem}\label{321}
$$\pk_n(321)=\sum_{C\in\mathcal{C}_n}\prod_{i=2}^{|\mathbf{u}(C)|}\left(1+\sum_{j=i}^{|\mathbf{u}(C)|}\mathbf{u}(C)_j\right)\prod_{i=1}^{|\mathbf{u}(C)|}(\mathbf{u}(C)_i-1)!.$$
Consequently, $\pk_n(321)=\sum_{k=1}^n(k-1)!p_{n,k}$, where
$$p_{n,k}=\begin{cases}
1, &\text{if }k=n,\\
\displaystyle (n-k+1)\sum_{i=n-k}^{n-1}\sum_{j=k+1-n+i}^i(n-i+j-k-1)!p_{i,j}, &\text{if }1\leq k\leq n-1.\\
\end{cases}$$
\end{theorem}
\begin{proof}
For brevity, let $p_n=\pk_n(321)$. 

Similar to the proof of Theorem \ref{123}, the map $\Omega_n:\Av_n(321)\to\mathcal{C}_n$ described below is a bijection. Each $\sigma\in\Av_n(321)$ can be uniquely decomposed as $x_ks_k\cdots x_1s_1$, where $x_k<\cdots<x_1$ are the left-to-right maxima of $\sigma$, and $s_k,\cdots,s_1$ are possibly empty segments. Define $\Omega_n(\sigma)$ to be the Catalan path of length $2n$ obtained by, for each $i\in[k]$ in order, drawing a block of $|s_i|+1$ up-steps and followed by a block of $x_i-x_{i+1}$ down-steps, where we view $x_{k+1}$ as 0.

Note that from the definition, for all $i\in[k]$, $x_i$ is larger than every entry before it and every entry in the segment $s_i$. Moreover, each segment $s_i$ is increasing as $\sigma$ as $\sigma$ avoids pattern 321. It follows that 
$$\ell(\sigma)=\prod_{i=1}^{k-1}\left(1+\sum_{j=i+1}^{k}(|s_j|+1)\right)\prod_{i=1}^k|s_i|!.$$ Since $|s_i|+1=\mathbf{u}(\Omega_n(\sigma))_i$ and $\Omega_n$ is a bijection, it follows that
$$\pk_n(321)=p_n=\sum_{\sigma\in\Av_n(321)}\ell(\sigma)=\sum_{C\in\mathcal{C}_n}\prod_{i=2}^{|\mathbf{u}(C)|}\left(1+\sum_{j=i}^{|\mathbf{u}(C)|}\mathbf{u}(C)_j\right)\prod_{i=1}^{|\mathbf{u}(C)|}(\mathbf{u}(C)_i-1)!,$$
proving the first part.

Let $W(C)=\prod_{i=2}^{|\mathbf{u}(C)|}(1+\sum_{j=i}^{|\mathbf{u}(C)|}\mathbf{u}(C)_j)\prod_{i=1}^{|\mathbf{u}(C)|}(\mathbf{u}(C)_i-1)!$, and let $p_{n,k}=\frac1{(k-1)!}\sum_{C\in\mathcal{C}_{n,k}}W(C)$, so that $p_n=\pk_n(321)=\sum_{k=1}^n(k-1)!p_{n,k}$. Note that $p_{n,n}=1$ as $\mathcal{C}_{n,n}$ contains only the Catalan path consisting of $n$ up-steps followed by $n$ down-steps. For $n>k\geq1$ and $C\in\mathcal{C}_{n,k}$, we have $|\mathbf{u}(C)|\geq2$. Let $i'\in[k]$ be the length of the first block of down-steps in $C$ and let $C'\in\mathcal{C}_{n-i',j}$ be obtained by deleting the first peak of $C$. Note that $\mathbf{u}(C')_t=\mathbf{u}(C)_{t+1}$ for all $2\leq t\leq|\mathbf{u}(C')|=|\mathbf{u}(C)|-1$, while $j=\mathbf{u}(C')_1=\mathbf{u}(C)_1+\mathbf{u}(C)_2-i'=k+\mathbf{u}(C)_2-i'$. Therefore,
\begin{align*}
W(C)&=\prod_{i=2}^{|\mathbf{u}(C)|}\left(1+\sum_{j=i}^{|\mathbf{u}(C)|}\mathbf{u}(C)_j\right)\prod_{i=1}^{|\mathbf{u}(C)|}(\mathbf{u}(C)_i-1)!\\
&=\left(1+\sum_{j=2}^{|\mathbf{u}(C)|}\mathbf{u}(C)_j\right)(\mathbf{u}(C)_1-1)!(\mathbf{u}(C)_2-1)!\prod_{i=3}^{|\mathbf{u}(C)|}\left(1+\sum_{j=i}^{|\mathbf{u}(C)|}\mathbf{u}(C)_j\right)\prod_{i=3}^{|\mathbf{u}(C)|}(\mathbf{u}(C)_i-1)!\\
&=\left(1+\sum_{j=2}^{|\mathbf{u}(C)|}\mathbf{u}(C)_j\right)\frac{(\mathbf{u}(C)_1-1)!(\mathbf{u}(C)_2-1)!}{(\mathbf{u}(C')_1-1)!}W(C')\\
&=(n-k+1)\frac{(k-1)!(i'+j-k-1)!}{(j-1)!}W(C').
\end{align*}
Thus, we have
\begin{align*}
p_{n,k}&=\frac1{(k-1)!}\sum_{C\in\mathcal{C}_{n,k}}W(C)=\sum_{i'=1}^k\sum_{j=k-i'+1}^{n-i'}\sum_{C'\in\mathcal{C}_{n-i',j}}(n-k+1)\frac{(i'+j-k-1)!}{(j-1)!}W(C').\\
&=(n-k+1)\sum_{i'=1}^k\sum_{j=k-i'+1}^{n-i'}(i'+j-k-1)!p_{n-i',j}\\
&=(n-k+1)\sum_{i=n-k}^{n-1}\sum_{j=k+1-n+i}^i(n-i+j-k-1)!p_{i,j},
\end{align*}
as required.
\end{proof}

\section{Pattern avoidance in block permutations}
In this section, we provide bijective proofs of Theorem \ref{thm:main1} and \ref{thm:main2}, which were originally proved algebraically by Adeniran and Pudwell in \cite{AP}. 
\subsection{Bijective proof of Theorem \ref{thm:main1}}\label{123132}
We begin with some preliminary analysis of parking functions whose block permutations avoid the patterns 123 and 132, and introduce some concepts and notations that will be used in the proof below. Let $f\in\Pf_n(123,132)$ and let $\pi_f\in S_n$ be its associated block permutation. Since $\pi_f$ avoids 132, every entry in $\pi_f$ appearing before $n$ is larger than every entry appearing after $n$. Since $\pi_f$ avoids 123, the entries in $\pi_f$ before $n$ appear in decreasing order, and each block has size 0, 1 or 2. In particular, $\pi_f(1)=n$ or $n-1$. Also, using condition \ref{blockcondition} and that each block has size 0, 1 or 2, we see that there is always an equal number of size 2 blocks and empty blocks, and for all $i\geq 1$, the $i$-th size 2 block always appear before the $i$-th empty block, provided that both exist. Hence, the subsequence of size 2 blocks and empty blocks can be viewed as a sequence of correctly matched left and right brackets. As such, for every size 2 block, there exists a unique corresponding empty block that is matched with it in this way. 

We now recursively partition $f\in\Pf_n(123,132)$ written in block notation into what we called \textit{clusters}, each of which is a union of blocks in $f$, as follows. When $n=0$, the unique empty parking function contains no cluster. For $n\geq 1$:

\begin{itemize}
    \item If $\pi_f(1)=n$, let $0\leq k\leq n-1$ be the smallest integer such that $\pi_f(i)=n+1-i$ for all $1\leq i\leq n-k$. By condition \ref{blockcondition}, the first $n-k$ blocks of $f$ all have size 1, and contain $n,n-1,\cdots,k+1$ in order. The first cluster of $f$ is then defined to be the union of these $n-k$ blocks, and is called an \textit{extend cluster} of length $n-k$.

    \item If $\pi_f(1)=n-1$, let $0\leq k\leq n-1$ be such that $\pi_f(n-k)=n$. It follows that $\pi_f(i)=n-i$ for all $1\leq i\leq n-k-1$ as $\pi_f$ avoids 123 and 132. Moreover, one of the following must occur:
    \begin{itemize}
        \item The first $n-k$ blocks of $f$ all have size 1. In this case, the first cluster of $f$ is defined to be the union of the first $n-k$ blocks of $f$, and is called a \textit{branch cluster} of length $n-k$. 
        \item The first $n-k-2$ blocks of $f$ all have size 1, while the $(n-k-1)$-th block of $f$ have size 2 and contains $k+1$ and $n$. In this case, the first cluster of $f$ is defined to be the union of the first $n-k-1$ blocks of $f$, along with the unique empty block corresponding to the $(n-k-1)$-th block of $f$, and we call it a \textit{jump cluster} of length $n-k$.
    \end{itemize}
\end{itemize}
In all cases above, let $f'$ be obtained by removing all blocks in the first cluster from $f$, and note that $f'$ is a parking function in $\Pf_k(123,132)$. We then define the remaining clusters of $f$ to be the clusters of $f'$. Furthermore, we define the \textit{main portion} of each cluster to be the size 1 or 2 blocks inside the cluster. It follows from the definition above that if we ignore empty blocks, then the main portion of each cluster appears consecutively in $f$. Finally, for a jump cluster, we say that its unique empty block lies \textit{inside} another cluster if it lies ahead of a block in the main portion of this cluster. See Figure \ref{fig:123132big} for an example of how the blocks of a parking function $f\in\Pf_{25}(123,132)$ are partitioned into clusters. 

We are now ready to give a bijective proof of Theorem \ref{thm:main1}. 

\begin{proof}[Proof of Theorem \ref{thm:main1}]
Let $\mathcal{T}^{\text{odd}}_{n+1}$ be the set of ordered rooted trees with $n+1$ edges and odd root degrees. We find a bijective correspondence between $\Pf_n(123,132)$ and $\mathcal{T}^{\text{odd}}_{n+1}$ for all $n\geq 0$.

\mbox{}

\noindent$\bullet\mbox{ }\Phi_n:\Pf_n(123,132)\to\mathcal{T}^{\text{odd}}_{n+1}$.

We define $\Phi_n$ using induction on $n$. For each $f\in\Pf_n(123,132)$, in addition to defining the rooted tree $\Phi_n(f)$, we also label its $n+1$ non-root vertices with $0,1,\cdots,n$ to make the definitions of $\Phi_n$ and $\Psi_n$ easier to describe. The definition of $\Phi_0,\Phi_1,\Phi_2$ and $\Phi_3$ along with the vertex labels are given in Figure \ref{fig:123132example}. 

Now assume $n\geq 4$ and let $f\in\Pf_n(123,213)$. Let the first cluster of $f$ have length $n-k$, where $0\leq k\leq n-1$. Let $f'$ be obtained by deleting the first cluster from $f$, and note that $f'\in\Pf_k(123,132)$. Define two labelled graphs $P_{k+1,n}$, $Q_{k+1,n}$ as in Figure \ref{fig:pqexample} below. Depending on the first cluster of $f$, we split into several cases, but in all cases, $\Phi_n(f)$ will be obtained by attaching the unique unlabelled vertex of either $P_{k+1,n}$ or $Q_{k+1,n}$ to a vertex in $\Phi_k(f')$ to create either one or two new branches at left-most positions, respectively. 

\begin{figure}[t]
    \centering
    \input{123132example}
    \caption{Bijection between $\Pf_n(123,132)$ and $\mathcal{T}^{\text{odd}}_{n+1}$ defined by $\Phi_n$ and $\Psi_n$ for $n=0,1,2,3$.}\label{fig:123132example}
\end{figure}

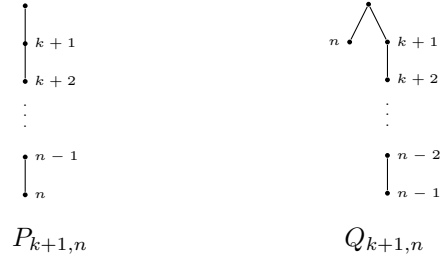
\begin{figure}[h]
    \centering
    \input{pqexample}
    \caption{The two labelled graphs $P_{k+1,n}, Q_{k+1,n}$ used in the definition of $\Phi_n$.}\label{fig:pqexample}
\end{figure}

\mbox{}

\textbf{Case 1.} The first cluster of $f$ is an extend cluster. Define $\Phi_n(f)$ to be the rooted tree obtained by attaching the unlabelled vertex in $P_{k+1,n}$ to the vertex in $\Phi_k(f')$ with label $k$. In this case, we say that we have performed an \textit{extend operation} to the vertex with label $k$. 

\mbox{}

\textbf{Case 2.} The first cluster of $f$ is a branch cluster. Define $\Phi_n(f)$ by attaching the unlabelled vertex in $Q_{k+1,n}$ to the vertex in $\Phi_k(f')$ with label $k$. In this case, we say that we have performed a \textit{branch operation} to the vertex with label $k$. 

\mbox{}

\textbf{Case 3.} The first cluster of $f$ is a jump cluster. In this case, we will perform a \textit{jump operation} to some vertex in $\Phi_k(f')$, which will depend on where the unique empty block in this jump cluster lies in $f$.

\hspace{\parindent}\textbf{Case 3.1.} The unique empty block in this cluster lies in an extend cluster consisting of blocks $\{a\},\cdots,\{b+1\}$, and appears ahead of the block $\{\ell\}$, where $b+1\leq\ell\leq a$. Define $\Phi_n(f)$ by attaching the unlabelled vertex in $Q_{k+1,n}$ to the vertex of $\Phi_k(f')$ with label $\ell-1$.

\hspace{\parindent}\textbf{Case 3.2.} The unique empty block in this cluster lies in a branch cluster consisting of blocks $\{a-1\},\cdots,\{b+1\},\{a\}$ ahead of the block $\{\ell\}$, where again $b+1\leq\ell\leq a$. If $b+1\leq\ell\leq a-1$, define $\Phi_n(f)$ by attaching the unlabelled vertex in $Q_{k+1,n}$ to the vertex in $\Phi_k(f')$ with label $\ell$, while if $\ell=a$, define $\Phi_n(f)$ by attaching the unlabelled vertex in $Q_{k+1,n}$ to the vertex in $\Phi_k(f')$ with label $b$.

\hspace{\parindent}\textbf{Case 3.3.} The unique empty block in this cluster lies in another jump cluster whose main portion consists of blocks $\{a-1\},\cdots,\{b+2\},\{b+1,a\}$, and appears ahead of the block containing $\ell$ for some $b+1\leq\ell\leq a-1$. Define $\Phi_n(f)$ by attaching the unlabelled vertex in $Q_{k+1,n}$ to the vertex in $\Phi_k(f')$ with label $\ell$. 

\hspace{\parindent}\textbf{Case 3.4.} The unique empty block in this cluster does not lie in another cluster of $f'$, or equivalently it is the last block of $f$. Define $\Phi_n(f)$ by attaching the unlabelled vertex in $Q_{k+1,n}$ to the root of $\Phi_k(f')$.

\mbox{}

See Figure \ref{fig:123132big} for an example where we compute $\Phi_{25}(f)$ for a parking function $f\in\Pf_{25}(123,132)$.

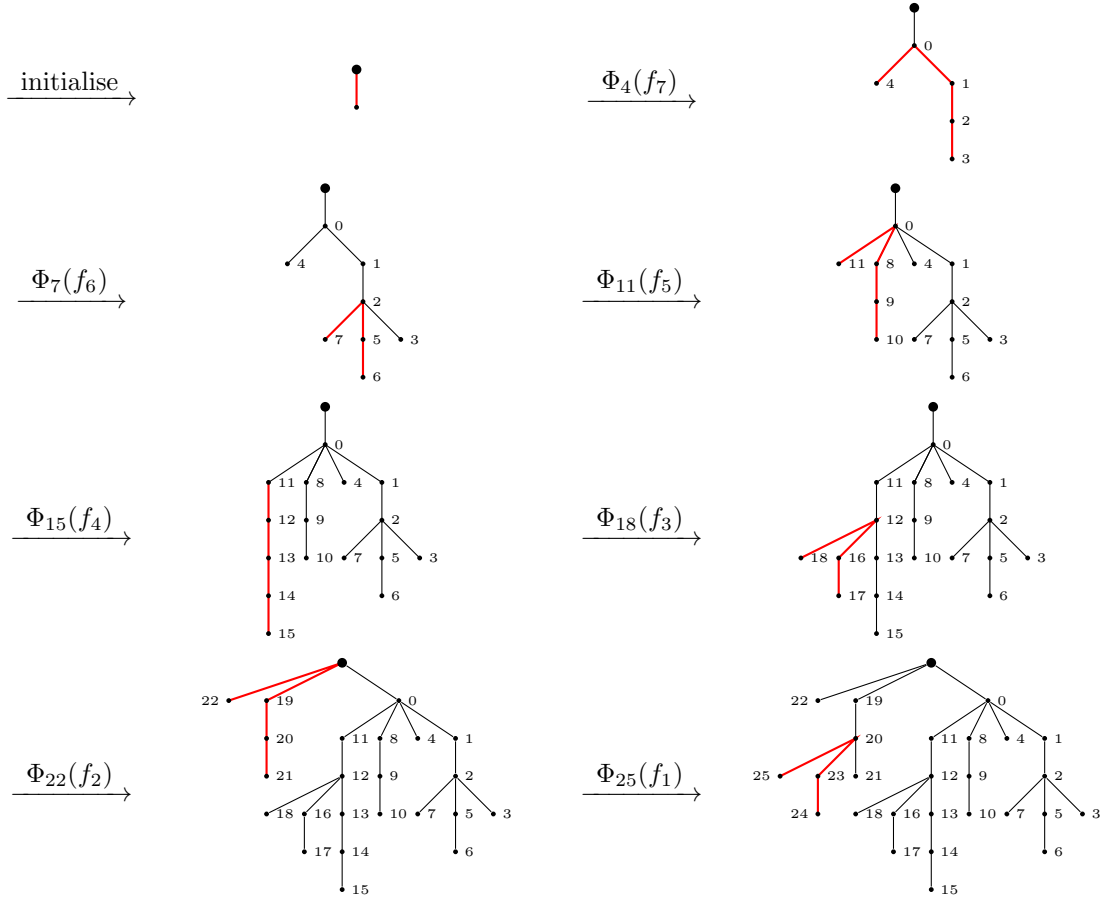
\begin{figure}
    \centering
    \input{123132big}
    \caption{An example computing $\Phi_{25}(f)$, where $f=(\{24\},\{23,25\},\{21\},\emptyset,\{20\},\{19,22\},\{17\},\{16,\\18\},\{15\},\{14\},\emptyset,\{13\},\{12\},\{10\},\{9\},\{8,11\},\{6\},\{5,7\},\{3\},\emptyset,\{2\},\{1\},\emptyset,\{4\},\emptyset)$. The cluster that each block belongs to and the cluster types are listed in the tables above. For each $i\in[7]$, $f_i$ is the parking function obtained by only retaining the $i$-th to 7-th clusters of $f$. In each step, the edges added are highlighted in red.}\label{fig:123132big}
\end{figure}

Before moving on, we prove two properties of the function $\Phi_n$ that we will use in the definitions of $\Psi_n$ and to justify bijectivity.
\begin{claim}\label{leftmost}
Suppose $\Phi_n(f)=T$ and $u\in T$ is a vertex with $m\geq 3$ branches $B_1,B_2,\cdots,B_m$, listed from left to right. Then every vertex in $B_1, B_2$ is created later than all vertices in $B_i$ for all $3\leq i\leq m$. 
\end{claim}
\begin{proof}[Proof of claim]
Since $B_1,B_2$ are the two left-most branches of $u$ which has at least three branches, from the definition of $\Phi_n$, we see that the first few vertices in branches $B_1,B_2$ must be created together by performing a jump operation to the vertex $u$, which corresponds to a jump cluster $C$ in $f$. Suppose for a contradiction that there exists a vertex in $B_i$ for some $3\leq i\leq m$ that is created after this jump operation, then this can happen only if we perform a jump operation to some vertex $u'\in B_i$, that corresponds to a jump cluster $C'$ in $f$ appearing before $C$. However, since $u$ is created before $u'$ in $T$, from the definition of $\Phi_n$, we see that the unique empty block in the jump cluster $C'$ appears in $f$ before the unique empty block in the jump cluster $C$. This is a contradiction to how the empty blocks are matched, proving the claim. 
\end{proof}

We say that a vertex in a rooted tree is a \textit{branching vertex} if it has at least two children. 

\begin{claim}\label{1or2}
Suppose $\Phi_n(f)=T$ and $u\in T$ is a vertex whose two left-most branches are $B_1,B_2$, with $B_1$ on the left of $B_2$. Suppose $B_2$ contains a branching vertex, then after the initial branch or jump operation that creates the first few vertices in $B_1$ and $B_2$, all remaining vertices in $B_1$ are created before any other vertices in $B_2$ are. In particular, the last jump or branch operation in the creation of $B_1$ and $B_2$ is performed to a vertex in $B_2$.
\end{claim}
\begin{proof}[Proof of claim]
From the definition of $\Phi_n$, the first vertices in branches $B_1,B_2$ must be created together by performing an operation that corresponds to a branch or jump cluster $C$ in $f$. After this operation, we have created a single vertex $u'$ in $B_1$ and a path $P$ in $B_2$. Since $B_2$ contains a branching vertex, at some point we must apply a jump operation to some vertex on $P$ to create new branches. Let $C_2$ be the jump cluster in $f$ corresponding to the first such jump operation, and assume the operation is performed to $u_2\in P$. If after this, for a contradiction, some more vertices in $B_1$ are created, then there must be a jump operation performed to some vertex $u_1$ in $B_1$. Let $C_1$ be the corresponding jump cluster in $f$. Note that $C_1$ appears before $C_2$ in $f$, the unique empty block in $C_2$ lies inside the cluster $C$, while the unique empty block in $C_1$ lies in a cluster appearing before $C$ in $f$. This is a contradiction to how the empty blocks are matched, proving the claim.
\end{proof}

\mbox{}

\noindent$\bullet\mbox{ }\Psi_n:\mathcal{T}^{\text{odd}}_{n+1}\to\Pf_n(123,132)$.

We define $\Psi_n$ inductively. The definition of $\Psi_0,\Psi_1,\Psi_2$ and $\Psi_3$ are given in Figure \ref{fig:123132example}. 

Now assume $n\geq 4$ and let $T\in\mathcal{T}^{\text{odd}}_{n+1}$. If $T$ is just a path of length $n+1$, let $\Psi_n(T)$ be the parking function consisting of a single extend cluster of length $n$. Otherwise, we run the following algorithm which, as we will show below, identifies the vertex in tree $T$ to which the operation corresponding to the first branch or jump cluster in $\Psi_n(T)$ is performed. 

To initialise, begin at the root of $T$.

Step 1: Descend until we first reach a branching vertex, then go to Step 2.

Step 2: Look at the two left-most branches $B_1$ and $B_2$ of the current vertex, where $B_1$ is to the left of $B_2$. If both $B_1$ and $B_2$ contain no branching vertex, stop. If $B_2$ contains no branching vertex but $B_1$ does, then repeat Step 1 in $B_1$. If $B_2$ contains a branching vertex, then repeat Step 1 in $B_2$. 

\begin{claim}\label{correctalgorithm}
If $\Phi_n(f)=T$, then the algorithm applied to $T$ identifies the vertex to which the operation corresponding to the first branch or jump cluster in $f$ is performed. 
\end{claim}
\begin{proof}
In each iteration of the algorithm, suppose the current vertex is $u$, with its branches listed from left to right being $B_1,B_2,\cdots,B_m$. By Claim \ref{leftmost}, since $B_1$ and $B_2$ are created later than everything in $B_i$ for $3\leq i\leq m$, the last branch or jump operation is performed to either $u$ if $B_1,B_2$ both contain no branching vertex, in which case the algorithm terminates and $u$ is correctly returned, or to a vertex in $B_1$ or $B_2$ if at least one of them contains a branching vertex. In the latter case, if $B_2$ contains no branching vertex, then $B_1$ does and the last branch or jump operation is therefore performed to a vertex in $B_1$, and the algorithm correctly moves to $B_1$ and iterates. Otherwise, $B_2$ contains a branching vertex and by Claim \ref{1or2}, the last branch or jump operation is performed to a vertex in $B_2$, and the algorithm correctly moves to $B_2$ and iterates, proving the claim.
\end{proof}

Note that the output of this algorithm will always be a branching vertex $v$ of $T$ whose two left-most branches both contain no branching vertex, in other words, these two branches are paths. Say they are $P_1,P_2$ with $\ell_1,\ell_2$ vertices, respectively, with $P_1$ to the left of $P_2$ and $\ell_1+\ell_2=n-k$ for some $0\leq k\leq n-2$. Let $T'$ be the subtree of $T$ obtained by deleting the vertices in $P_1,P_2$. Note that $T'\in\mathcal{T}^{\text{odd}}_{k+1}$. From induction hypothesis, if $f'=\Psi_{k}(T')$, then $T'=\Phi_{k}(f')$. Hence, from the definition of $\Phi_{k}(f')$, we obtain a vertex labelling of $T'$. We split into several cases depending on the size of $\ell_1$ and the label of $v$ in $T'$. 

\mbox{}

\textbf{Case 1.} $\ell_1>1$. Let $T_1$ be the subtree of $T$ obtained by deleting the all vertices in $P_1$, except the one adjacent to $v$. Note that $T_1\in\mathcal{T}^{\text{odd}}_{n-\ell_1+1}$. Define $\Psi_n(T)$ to be parking function on $[n]$ obtained by attaching an extend cluster consisting of blocks $\{n\},\{n-1\},\cdots,\{n-\ell_1+2\}$ to the front of $\Psi_{n-\ell_1+1}(T_1)$.

\mbox{}

\textbf{Case 2.} $\ell_1=1$ and the label of $v$ in $T'$ is $k$. Let $\Psi_n(T)$ be the parking function on $[n]$ obtained by attaching a branch cluster consisting of blocks $\{n-1\},\cdots,\{k+1\},\{n\}$ to the front of $f'=\Psi_{k}(T')$.

\mbox{}

\textbf{Case 3.} $\ell_1=1$ and the label of $v$ in $T'$ is not $k$. It follows that if the label of $v$ in $T'$ is the largest element in a cluster of $f'$, then this cluster is not the first cluster in $f'$.

\hspace{\parindent}\textbf{Case 3.1.} The label of $v$ in $T'$ is $\ell$, and $\ell+1$ is in an extend cluster in $f'$. Let $\Psi_n(T)$ be the parking function on $[n]$ obtained by attaching a jump cluster to $f'=\Psi_{k}(T')$, where we attach the main part consisting of the blocks $\{n-1\},\cdots,\{k+2\},\{k+1,n\}$ to the front of $f'$ and insert the unique empty block ahead of the block $\{\ell+1\}$ in $f'$.

\hspace{\parindent}\textbf{Case 3.2.} The label of $v$ in $T'$ is $\ell$, and $\ell+1$ is in a branch cluster in $f'$. If $\ell+1$ is the smallest element in this branch cluster, let $\Psi_n(T)$ be the parking function on $[n]$ obtained by attaching a jump cluster to $f'=\Psi_{k}(T')$, where we attach the main portion consisting of the blocks $\{n-1\},\cdots,\{k+2\},\{k+1,n\}$ to the front of $f'$ and insert the unique empty block ahead of the last block (the one containing the largest element) in this branch cluster. Otherwise, let $\Psi_n(T)$ be the parking function on $[n]$ obtained by attaching a jump cluster to $f'=\Psi_{k}(T')$, where we attach the main portion consisting of the blocks $\{n-1\},\cdots,\{k+2\},\{k+1,n\}$ to the front of $f'$ and insert the unique empty block ahead of the block $\{\ell\}$. 

\hspace{\parindent}\textbf{Case 3.3.} The label of $v$ in $T'$ is $\ell$, and $\ell+1$ is in a jump cluster $C$ in $f'$. We claim that $\ell+1$ cannot be the smallest element in $C$. Indeed, this would mean that $\ell$ is the largest element in the cluster after $C$, and the jump operation corresponding to the cluster $C$ jumps away from the vertex with label $\ell$, namely $v$, to some other vertex $u$ and creates two new branches there. Note that from the definition of $\Phi_k$ and that $\ell$ is the largest element in its cluster, $v$ must be a leaf in $T'$. Let $w$ be the common ancestor of $u$ and $v$ in $T'$ farthest away from the root. If $w=u$, then $u$ is an ancestor of $v$ in $T$ with at least three branches, with $v$ not in the two left-most ones. This contradicts that the algorithm applied to $T$ returns $v$. We cannot have $w=v$ as $v$ is a leaf in $T'$. Hence, $v$ and $u$ are strict descendants of $w$, and are in different branches $B_u$ and $B_v$ below $w$ by the choice of $w$. 

If $B_u$ is to the left of $B_v$, and $B_v$ is at least the third left-most branch below $w$, then we have a contradiction to Claim \ref{leftmost} as $u$ in $B_u$ is created earlier than $v$ in $B_v$ in $T'$. Thus, $B_u$ and $B_v$ are the two left-most branches below $w$. Since the label of $v$ is the largest in its cluster, $v$ is not created by the initial branch or jump operation that creates the first vertices of $B_u$ and $B_v$. Thus, $B_v$ must contain a branching vertex, as otherwise $v$ cannot be created. But then we have a contradiction to Claim \ref{1or2} as the jump operation corresponding to cluster $C$ attaches new branches to $u\in B_u$ after $v\in B_v$ is created.

If $B_u$ is to the right of $B_v$, and $B_u$ is at least the third left-most branch below $w$, then we have a contradiction to Claim \ref{leftmost} as the vertices created by the jump operation corresponding to $C$ in $B_u$ are created later than $v$ in $B_v$ in $T'$. Thus, $B_v$ and $B_u$ are the two left-most branches below $w$. But this contradicts that the algorithm applied to $T$ returns $v$, as $B_u$ contains a branching vertex $u$, so the algorithm should have returned a vertex inside $B_u$. This proves the claim.  

Hence, it follows that $\ell$ and $\ell+1$ are both in the jump cluster $C$. Let $\Psi_n(T)$ be the parking function on $[n]$ obtained by attaching a jump cluster to $f'=\Psi_{k}(T')$, where we attach the main portion consisting of the blocks $\{n-1\},\cdots,\{k+2\},\{k+1,n\}$ to the front of $f'$ and insert the unique empty block ahead of the block in $C$ containing $\ell$.

\hspace{\parindent}\textbf{Case 3.4.} $\ell_1=1$ and $v$ is the root vertex of $T'$. Let $\Psi_n(T)$ be the parking function on $[n]$ obtained by attach a jump cluster to $f'=\Psi_{k}(T')$, where we attach the main portion consisting of the blocks $\{n-1\},\cdots,\{k+2\},\{k+1,n\}$ to the front of $f'$ and insert the unique empty block at the end of $f'$.

\mbox{}

\noindent$\bullet$\mbox{ } Proof of bijectivity.

We first show that every $T\in\mathcal{T}^{\text{odd}}_{n+1}$ satisfies $\Phi_n(\Psi_n(T))=T$. This can be checked easily for $n\leq 3$, and when $T$ is just a path of length $n+1$. Otherwise, let $v,k,T',\ell_1$ be as in the definition of $\Psi_n(T)$ above. From induction hypothesis, $\Phi_k(\Psi_k(T'))=T'$. Note that from the definition above, $\Psi_n(T)$ is obtained by appropriately attaching a branch cluster or a jump cluster $C$ to the front of $\Psi_k(T')$, and then attach a further extend cluster $C'$ to the front if $\ell_1>1$. Then by the definition of $\Phi_n$, $\Phi_n(\Psi_n(T))$ is obtained by performing the branch or jump operation corresponding to $C$ to some vertex in $\Phi_k(\Psi_k(T'))=T'$ according to the cluster $C$, and then appropriately perform the extend operation corresponding to $C'$. It is easy to see from the cases above that performing these operations on $T'$ gives us exactly $T$, thus $\Phi_n(\Psi_n(T))=T$. 

We now show that every $f\in\Pf_n(123,132)$ satisfies $\Psi_n(\Phi_n(f))=f$. Note that from the definition there cannot be two or more consecutive extend clusters in $f$. Let $f'$ be obtained by removing from $f$ its first branch or jump cluster along with the at most one extend cluster that might exist before it, and note that $f'\in\Pf_k(123,132)$ for some $0\leq k\leq n-2$. Let $T'$ be the subtree of $T=\Phi_n(f)$ obtained by removing the two branches $P_1,P_2$ of the vertex $v\in T$ output by the algorithm. By Claim \ref{correctalgorithm}, $v$ is the vertex to which the last branch or jump operation is performed, which corresponds to the first branch or jump cluster in $f$. Thus, $P_1$ and $P_2$ are exactly the part of $T$ created by the first branch or jump cluster in $f$ and the extend cluster that may exist before it. It follows that $\Phi_k(f')=T'$ as $T'$ is the part of $T$ created by the remaining clusters in $f$, or in other words those in $f'$. Thus, $\Psi_k(T')=\Psi_k((\Phi_k(f'))=f'$ by induction hypothesis. Finally, from the definition of $\Psi_n$, $\Psi_n(\Phi_n(f))=\Psi_n(T)$ is obtained by appropriately attaching to the front of the parking function $\Psi_k(T')=f'$ a branch or jump cluster and possibly another extend cluster that corresponds to the branches $P_1,P_2$, which are exactly the clusters removed from $f$ to obtain $f'$. Hence, $\Psi_n(\Phi_n(f))=f$, completing the proof.
\end{proof}

\subsection{Bijective proof of Theorem \ref{thm:main2}}\label{123213}
Again, we begin with some preliminary analysis of parking functions whose block permutations avoid the patterns 123 and 213, and introduce some concepts and notations that will be used in the proof below. Let $f\in\Pf_n(123,213)$, and let $\pi_f\in S_n$ be its associated block permutation. Suppose $\pi_f(1)=k+1$, where $0\leq k\leq n-1$. We claim that $\pi_f(i)=n+2-i$ for all $2\leq i\leq n-k$. Indeed, since $\pi_f$ avoids $123$, $n,n-1,\cdots,k+2$ must appear in $\pi_f$ in decreasing order, and if there exists some $m<k+1$ appearing between $k+1$ and $k+2$ in $\pi_f$, then $\pi_f$ contains the pattern 213, a contradiction. It follows from the definition of blocks that each of $k+1,n,n-1,\cdots,k+2$ appear in a block of size 1, except possibly $k+1$ and $n$ may appear together in a block of size 2. Moreover, since $\pi_f$ avoids $123$, we again have that every block in $f$ has size 0,1 or 2, and thus using condition \ref{blockcondition}, like in the 123,132 avoiding case, every size 2 block is matched with a unique empty block. 

We now recursively partition $f$ into what we called \textit{clusters}, each of which is a union of blocks in $f$, as follows. When $n=0$, the unique empty parking function contains no cluster. For $n\geq 1$:
\begin{itemize}
    \item If $k+1,n,n-1,\cdots,k+2$ all appear in blocks of size 1, then the first cluster of $f$ is defined to be the union of these $n-k$ blocks, and is called a \textit{closed cluster} with length $n-k$ and \textit{parameter} $n-k-1$.

    \item If $k+1$ and $n$ appear together in a block of size 2, followed by $n-1,\cdots,k+2$ each in individual blocks, then consider the unique empty block in $f$ matched with the size 2 block $\{k+1,n\}$. 
    \begin{itemize}
        \item If this empty block appears right after the block containing $k+2+\ell$, where $0\leq\ell\leq n-k-2$, then the first cluster of $f$ is defined to be the union of the first $n-k$ blocks of $f$ (which includes the empty block matched with $\{k+1,n\}$), and is called a \textit{closed cluster} with length $n-k$ and \textit{parameter} $\ell$.
        \item If this empty block appears after the block $\{k+2\}$ and not immediately after it, then the first cluster of $f$ is defined to be the union of the first $n-k-1$ blocks of $f$, along with this empty block, and we call it an \textit{open cluster} of length $n-k$.
    \end{itemize}
\end{itemize}
In all cases above, let $f'$ be obtained by removing all blocks in the first cluster, and note that $f'$ is a parking function in $\Pf_k(123,213)$. We then define the remaining clusters of $f$ to be the clusters of $f'$. Furthermore, we define the \textit{main portion} of each cluster to be the size 1 or 2 blocks inside the cluster. It follows from the definition above that if we ignore empty blocks, then the main portion of each cluster appears consecutively. Finally, for an open cluster, we say that its unique empty block lies \textit{inside} another cluster if it lies after a block in the main portion of this cluster. Note that this is different from the corresponding definition in the 123,132 case. See Figure \ref{fig:123213big} for an example of how the blocks of a parking function $f\in\Pf_{20}(123,213)$ are partitioned into clusters.

We define another concept that will be used in the proof below. Let $T$ be a rooted tree with its root $v$ having degree at least 2. We say that a rooted subtree $\overline{T}$ with root $\overline{v}$ is a \textit{full right subtree} of $T$ if the following conditions hold:
\begin{itemize}
    \item $\overline{v}$ is either equal to $v$ or is reachable from $v$ in $T$ by always going down the right-most branch of every branching vertex on the way.
    \item $\overline{T}$ is induced by $\overline{v}$ and a non-empty collection of consecutive branches below $\overline{v}$ starting from the right-most one.
\end{itemize}

Now we are ready to give a bijective proof of Theorem \ref{thm:main2}. 

\begin{proof}[Proof of Theorem \ref{thm:main2}]
Let $\mathcal{T}^{\geq2}_{n+1}$ be the set of ordered rooted trees with $n+1$ edges and root degrees at least 2, except for when $n=0$ we let $\mathcal{T}^{\geq2}_{1}$ be the set consisting of the unique rooted tree with 1 edge. We find a bijective correspondence between $\Pf_n(123,213)$ and $\mathcal{T}^{\geq2}_{n+1}$ for all $n\geq 0$.

\mbox{}

\noindent$\bullet\mbox{ }\Phi_n:\Pf_n(123,213)\to\mathcal{T}^{\geq2}_{n+1}$.

The definition of $\Phi_0,\Phi_1,\Phi_2$ and $\Phi_3$ are given in Figure \ref{fig:123213example}.  We define $\Phi_n$ for $n\geq 4$ inductively, and include the following condition in the induction hypothesis, which needs to be maintained in the definition below. 

\begin{figure}[t]
    \centering
    \input{123213example}
    \caption{Bijection between $\Pf_n(123,213)$ and $\mathcal{T}^{\geq2}_{n+1}$ defined by $\Phi_n$ and $\Psi_n$ for $n=0,1,2,3$.}\label{fig:123213example}
\end{figure}

\stepcounter{propcounter}
\begin{enumerate}[label = \textbf{\Alph{propcounter}}]
\item\label{condition} Suppose $f\in\Pf_n(123,213)$, and $\overline{f}\in\Pf_k(123,213)$ is obtained by removing the first few clusters $C_1,\cdots,C_m$ of total size $n-k$ from $f$, with $C_m$ being a closed cluster of parameter $\ell'$. Furthermore, let $0\leq\ell\leq\ell'$ and suppose that there does not exist $i\in[m-1]$, such that $C_i$ is an open cluster and its unique empty block either lies inside $\overline{f}$ or inside $C_m$ and to the left of less than $\ell$ blocks in $C_m$. Then, the tree obtained by attaching a path of length $\ell$ to the root of $\Phi_k(\overline{f})$, with the other end of this path as the new root, is a full right subtree of $\Phi_n(f)$.
\end{enumerate}

Now assume $n\geq 4$ and let $f\in\Pf_n(123,213)$. Depending on the first cluster of $f$, we split into two cases. 

\mbox{}

\textbf{Case 1.} The first cluster of $f$ is a closed cluster of length $n-k$ and parameter $\ell$, where $0\leq k\leq n-1$ and $0\leq\ell\leq n-k-1$. Let $f'$ be obtained from $f$ by removing this first cluster, and observe that $f'\in\Pf_{k}(123,213)$. Define $\Phi_n(f)$ as follows. First, attach a path of length $\ell$ to the root of $\Phi_k(f')$ and make the other end of this path the new root. Then, attach a path of length $n-k-\ell$ to this new root as a new left branch. 

\mbox{}

\textbf{Case 2.} The first cluster $C$ of $f$ is an open cluster of length $n-k$, where $0\leq k\leq n-2$. Suppose the unique empty block in this cluster lies inside a cluster $C'$. Note that if $C'$ is an open cluster, then the empty block in $C'$ appears in $f$ after the empty block in $C$, despite $C'$ appearing after $C$. This contradicts how the empty blocks are matched, so $C'$ must be a closed cluster. 

Suppose the elements in the blocks in $C'$ are $b+1,\cdots,a$, for some $0\leq b<a\leq k$, and suppose the empty block in $C$ is located to the left of exactly $\ell$ blocks in $C'$, where $0\leq\ell\leq a-b-1$. Note that if the closed cluster $C'$ itself contains an empty block, then the empty block in $C$ must appear after the empty block in $C'$ by the way empty blocks are matched, so $C'$ has parameter at least $\ell$. Let $f'$ be obtained from $f$ by removing the first cluster, and observe that $f'\in\Pf_{k}(123,213)$. Let $\overline{f}$ be obtained from $f$ by removing all clusters up to and including $C'$, and observe that $\overline{f}\in\Pf_{b}(123,213)$. If there exists some cluster $C''$ in $f'$ but not in $\overline{f}$ that is an open cluster, such that its empty block is in $\overline{f}$, then the empty block in $C''$ appears in $f$ after the one in $C$, despite $C''$ appearing in $f$ after $C$. This contradicts how the empty blocks are matched, hence such cluster $C''$ does not exist. Similarly, no such open cluster can have its empty block lying in $C'$ and to the left of less than $\ell$ elements in $C'$. Hence, the assumptions in condition \ref{condition} hold. 

We can finally define $\Phi_n(f)$. Let $\overline{T}$ be the rooted tree obtained by attaching a path of length $\ell$ to the root of $\Phi_b(\overline{f})$ and letting the other end of this path, say $v$, be the new root. By condition \ref{condition}, $\overline{T}$ is a full right subtree of $\Phi_k(f')$. Let $T'$ be the subtree of $\Phi_k(f')$ induced by $v$ and all vertices not in $\overline{T}$. From the definition of full right subtree, the root of $\Phi_k(f')$ is in $T'$. Let $\Phi_n(f)$ be the rooted tree defined as follows. Starting with $\Phi_k(f')$, we view $v$ as the new root, and detach $T'$ from $\Phi_k(f')$. Then, we create a new left branch at $v$ that consists of a single edge, and reattach $T'$ to the other end of this edge. Finally, attach a path of length $n-k-1$ to the old root of $\Phi_k(f')$, which is in $T'$, as a new left-most branch.

\mbox{}

We need to check that condition \ref{condition} still holds for $\Phi_n(f)$ in both cases. This is clear in Case 1, as the operation corresponding to the additional closed cluster in front does not modify $\Phi_k(f')$. In Case 2, since the first cluster of $f$ has its empty block in $C'$, the assumptions in condition \ref{condition} can only hold if all clusters (and perhaps some more) up to and including $C'$ are removed. It follows that condition \ref{condition} holds for $\Phi_n(f)$ because it holds for $\Phi_k(f')$ and the operations in Case 2 leaves $\overline{T}$ unchanged.

See Figure \ref{fig:123213big} for an example where we compute $\Phi_{20}(f)$ for a parking function $f\in\Pf_{20}(123,213)$.

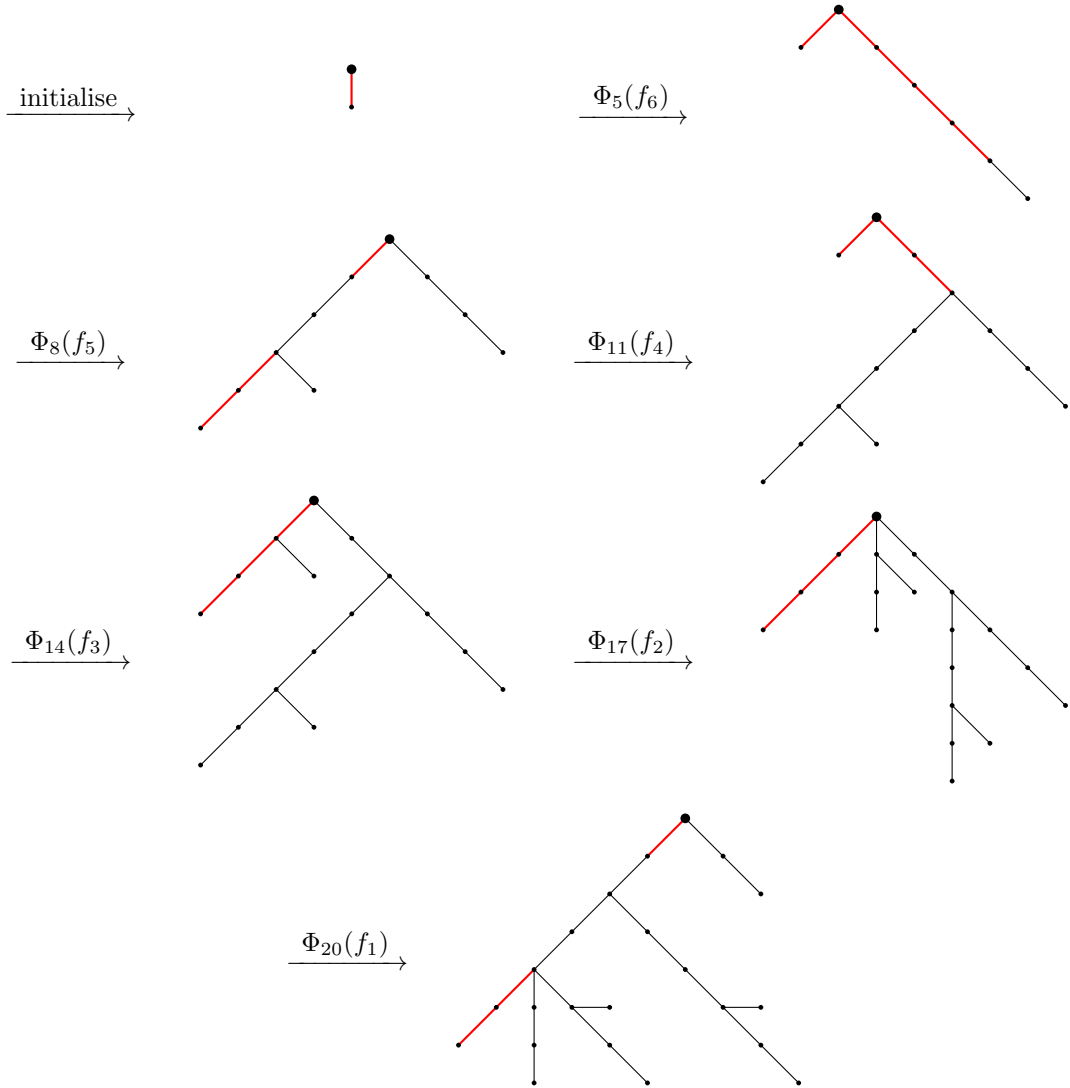
\begin{figure}
    \centering
    \input{123213big}
    \caption{An example computing $\Phi_{20}(f)$, where $f=(\{18,20\},\{19\},\{15,17\},\{16\},\emptyset,\{12,14\},\{13\},\\\{9\},\emptyset,\{11\},\{10\},\{6,8\},\{7\},\{1\},\{5\},\{4\},\emptyset,\{3\},\emptyset,\{2\})$. The cluster that each block belongs to and the cluster types are listed in the tables above. For each $i\in[6]$, $f_i$ is the parking function obtained by only retaining the $i$-th to 6-th clusters of $f$. In each step, the edges added are highlighted in red.}\label{fig:123213big}
\end{figure}

\mbox{}

\noindent$\bullet\mbox{ }\Psi_n:\mathcal{T}^{\geq2}_{n+1}\to\Pf_n(123,213)$.

The definition of $\Psi_0,\Psi_1,\Psi_2$ and $\Psi_3$ are given in Figure \ref{fig:123213example}. We define $\Psi_n$ inductively for $n\geq4$, and include the following condition, analogous to the one in the definition of $\Phi_n$, as part of the induction hypothesis, which needs to be maintained throughout.

\stepcounter{propcounter}
\begin{enumerate}[label = \textbf{\Alph{propcounter}}]
\item\label{condition2} Suppose $T\in\mathcal{T}_{n+1}^{\geq2}$, $\ell\geq 0$, and $\overline{T}\in\mathcal{T}_{k+1}^{\geq2}$ with root $\overline{v}$ satisfies that the tree obtained by attaching a path of length $\ell$ to $\overline{v}$ and declaring the other end of this path the new root is a full right subtree of $T$. Then $\overline{f}=\Psi_k(\overline{T})$ is obtained by deleting a few leading clusters of total size $n-k$ from $f=\Psi_n(T)$. Moreover, if $f\not=\overline{f}$, the cluster $C$ in $f$ immediately preceeding $\overline{f}$ is a closed cluster of parameter at least $\ell$, and no open cluster in $f$ that is not in $\overline{f}$ has its empty block either in $\overline{f}$, or in $C$ and to the left of less than $\ell$ blocks in $C$.
\end{enumerate}

Now assume $n\geq 4$ and let $T\in\mathcal{T}^{\geq2}_{n+1}$. Identify the left-most leaf of $T$ (the one obtained by starting at the root of $T$ and going left at every branching vertex until we reach a leaf) and look at the unique path connecting it to the root of $T$. Depending on whether this path contains any branching vertex, we split into two cases. 

\mbox{}

\textbf{Case 1.} The path connecting the left-most leaf of $T$ to the root of $T$ contains no branching vertex. In other words, the left-most branch of $T$ is just a path, say of length $n-k$. From assumption, the root of $T$ has degree at least 2. 

If the root has degree 2, and the right branch is also just a path, then this path has length $k+1$ and we define $\Psi_n(T)$ to be the parking function on $[n]$ consisting of a single closed cluster of size $n$ and parameter $k$. 

If the root has degree 2, and the right branch contains a branching vertex, let $1\leq\ell\leq k-1$ be the number of edges we need to traverse down the right branch until we encounter a branching vertex, and let $T'$ be the subtree of $T$ obtained by deleting the left branch of $T$ and the initial path of length $\ell$ on the right branch. Note that $T'\in \mathcal{T}^{\geq2}_{k-\ell+1}$. If the root has degree at least 3, set $\ell=0$, and let $T'$ be the subtree obtained by removing the left-most branch from $T$, and note that $T'\in \mathcal{T}^{\geq2}_{k-\ell+1}$ as well. In both of these cases, define $\Psi_n(T)$ to be the parking function on $[n]$ obtained by attaching a closed cluster of length $n-k+\ell$ and parameter $\ell$ to the front of $\Psi_{k-\ell}(T')$.

\mbox{}

\textbf{Case 2.} The path connecting the left-most leaf of $T$ to the root of $T$ contains a branching vertex. Suppose we traversed through $n-k-1$ edges on this path starting at the left-most leaf before first encountering a branching vertex $v$, where $1\leq k\leq n-2$. Let $T'$ be the tree obtained from $T$ by removing this path of length $n-k-1$ starting from the left-most leaf of $T$, removing the first edge in the left-most branch of $T$ and reattaching the remaining branch to the root of $T$, and finally declaring the $v$ to be the new root. Since $v$ is a branching vertex in $T$, the degree of $v$ in $T'$ is at least 2 and thus $T'\in\mathcal{T}^{\geq2}_{k+1}$. From assumption, the root of $T$ has degree at least 2. 

Suppose the root of $T$ has degree 2, and the right branch is just a path of length $\ell+1\geq1$. Note that this right branch must be created by the last cluster in $\Psi_k(T')$, which must be a closed cluster of parameter at least $\ell$. In this case, define $\Psi_n(T)$ to be the parking function on $[n]$ obtained by adding the main portion of an open cluster of length $n-k$ to the front of $\Psi_k(T')$, and placing its unique empty block to the left of exactly $\ell$ blocks in the last cluster of $\Psi_k(T')$.

If the root of $T$ has degree 2, and the right branch contains a branching vertex, let $\ell\geq1$ be the number of edges we need to traverse down the right branch until we encounter a branching vertex $u$, let $T_1$ be the subtree of $T$ induced by $u$ and all vertices below, and suppose $T_1$ contains $b+1$ edges. Then, $T_1\in\mathcal{T}^{\geq2}_{b+1}$, and the rooted tree obtained by attaching a path of length $\ell$ to $u$ and declaring the other end of this path the new root is a full right subtree of both $T$ and $T'$. If the root of $T$ has degree at least 3, set $\ell=0$, and let $T_1$ be the subtree of $T$ obtained by deleting the left-most branch. Again, note that if $T_1$ contains $b+1$ edges, then $T_1\in\mathcal{T}^{\geq2}_{b+1}$ and is a full right subtree of $T$ and $T'$. In both cases, by condition \ref{condition2}, the clusters in $\Psi_b(T_1)$ appear at the end of $\Psi_k(T')$, and the cluster preceeding $\Psi_b(T_1)$ in $\Psi_k(T')$ is a closed cluster $C$ with parameter at least $\ell$, and no open cluster in $\Psi_k(T')$ that is not in $\Psi_b(T_1)$ has its unique empty block either in $\Psi_b(T_1)$ or in $C$ and to the left of less than $\ell$ blocks in $C$. In both of these cases, define $\Psi_n(T)$ to be the parking function on $[n]$ obtained by adding the main portion of an open cluster of length $n-k$ to the front of $\Psi_k(T')$, and placing the unique empty block to the left of exactly $\ell$ blocks in $C$.

\mbox{}

Note that condition \ref{condition2} still holds for $\Psi_n(T)$ by induction hypothesis and the observation that a full right subtree of $T$ is either $T$ itself or a full right subtree of $T'$.

\mbox{}

\noindent$\bullet$ Proof of bijectivity.

We first show that every $T\in\mathcal{T}^{\geq2}_{n+1}$ satisfies $\Phi_n(\Psi_n(T))=T$. Let $T'$ be as in the definition of $\Psi_n(T)$, and let $f'$ be $\Psi_n(T)$ with the first cluster removed. Note that from the definition of $\Psi_n(T)$, $f'=\Psi_k(T')$ if $T'$ has $k+1$ edges, so from induction hypothesis, $T'=\Phi_k(\Psi_k(T'))=\Phi_k(f')$. From the definition of $\Phi_n$, $\Phi_n(\Psi_n(T))$ is obtained by performing to $\Phi_k(f')=\Phi_k(\Psi_k(T'))=T'$ the operation corresponding to the first cluster of $\Psi_n(T)$, which gives us exactly $T$. 

Now we show that every $f\in\Pf_n(123,213)$ satisfies $\Psi_n(\Phi_n(f))=f$. Let $f'$ be $f$ with first cluster removed and let $k$ be such that $f'\in\Pf_k(123,213)$. Let $T=\Phi_n(f)$ and let $T'$ be as in the definition of $\Psi_n(T)$. By looking through the cases in the definition of $\Phi_n$ and $\Psi_n$, we see that $\Phi_k(f')=T'$, so we have from induction hypothesis that $f'=\Psi_k(\Phi_k(f'))=\Psi_k(T')$. From the definition of $\Psi_n$, $\Psi_n(\Phi_n(f))=\Psi_n(T)$ is obtained by adding to the front of $\Psi_k(T')=f'$ the cluster corresponding to the operation changing $T'=\Phi_k(f')$ to $T=\Phi_n(f)$, which is exactly the first cluster of $f$, finishing the proof.
\end{proof}

\subsection{An explicit formula for $\pf_n(312,321)$}\label{compute312321}
In \cite{AP}, Adeniran and Pudwell obtained a formula for $\pf_n(312,321)$ expressed in terms of a sum over $\mathcal{C}_n$, similar to the ones we obtained during the proofs of Theorem \ref{123} and \ref{213}. Using the same method, we prove a more explicit formula for $\pf_n(312,321)$.

\begin{center}
\begin{tabular}{|c|c|c|}
\hline Patterns $\sigma_1,\sigma_2$ & $\pf_n(\sigma_1,\sigma_2)$, $1\leq n\leq8$ & OEIS \\\hline
 312, 321 & 1, 3, 13, 63, 324, 1736, 9589, 54223 & A362744\\\hline
\end{tabular}
\end{center}
\begin{theorem}\label{312321}
$$\pf_n(312, 321)=\frac{\binom{3n+1}{n}}{n+1}-\sum_{k=0}^{n-1}\frac{\binom{3n-3k+1}{n-k}}{2^{k+1}(n-k+1)}.$$
\end{theorem}
\begin{proof}
For $n\geq 1$, let
$$q_n=\sum_{C\in\mathcal{C}_n}\prod_{i=1}^{|\mathbf{u}(C)|}(\mathbf{u}(C)_i+1),\quad\quad
p_n=\sum_{C\in\mathcal{C}_n}\prod_{i=1}^{|\mathbf{u}(C)|-1}(\mathbf{u}(C)_i+1),$$
and let $Q(x)=1+\sum_{n\geq1}q_nx^n$ and $P(x)=\sum_{n\geq1}p_nx^n$ be their generating functions. By Theorem 28 in \cite{AP}, $\pf_n(312, 321)=p_n$ for all $n\geq1$. 

Let $W(C)=\prod_{i=1}^{|\mathbf{u}(C)|}(\mathbf{u}(C)_i+1)$. Similar to the proofs of Theorem \ref{123} and \ref{213} above, and using the canonical decomposition, we have for $n\geq 1$,
\begin{align*}
q_n&=\sum_{k=1}^n(k+1)\sum_{\substack{C_1\in\mathcal{C}_{a_1},\cdots,C_k\in\mathcal{C}_{a_k}\\a_1+\cdots+a_k=n-k\\a_1,\cdots,a_k\geq0}}\prod_{j=1}^kW(C_j)=\sum_{k=1}^n(k+1)\sum_{\substack{a_1+\cdots+a_k=n-k\\a_1,\cdots,a_k\geq0}}\prod_{j=1}^nq_{a_j}.
\end{align*}
And thus, 
\begin{align*}
Q(x)&=1+\sum_{n\geq1}q_nx^n=1+\sum_{n\geq1}\left(\sum_{k=1}^n(k+1)\sum_{\substack{a_1+\cdots+a_k=n-k\\a_1,\cdots,a_k\geq0}}\prod_{j=1}^nq_{a_j}\right)x^n\\
&=1+\sum_{k\geq1}(k+1)x^k\left(\sum_{a_1,\cdots,a_k\geq0}\prod_{j=1}^nq_{a_j}x^{a_j}\right)=1+\sum_{k\geq1}(k+1)x^k(Q(x))^k.
\end{align*}
Let $\overline{Q}(x)=xQ(x)$, then $\overline{Q}(x)=x(1-\overline{Q}(x))^{-2}$, so by Lemma \ref{lift}, we have $q_n=[x^n]Q(x)=[x^{n+1}]\overline{Q}(x)=\frac{1}{n+1}[x^n](1-x)^{-2n-2}=\frac{1}{n+1}\binom{3n+1}{n}$.

Now let $W'(C)=\prod_{i=1}^{|\mathbf{u}(C)|-1}(\mathbf{u}(C)_i+1)$, so that $p_n=\sum_{C\in\mathcal{C}_n}W'(C)$. For all $n>k\geq 1$ and $C\in\mathcal{C}_{n,k}$, let the canonical decomposition of $C$ be $C=U_1\cdots U_kD_1C_1\cdots D_kC_k$. As usual, $\mathbf{u}(C)$ is formed by attaching $\mathbf{u}(C_1),\cdots,\mathbf{u}(C_k)$ together and adding an entry of $k$ to the front. Moreover, there exists a unique $j\in[k]$ such that $C_j$ is the last among $C_1,\cdots,C_k$ to have non-zero length, so the last entry of $\mathbf{u}(C)$ is the last entry of $\mathbf{u}(C_j)$ as $\mathbf{u}(C_{j'})$ is empty for all $j<j'\leq k$. Thus, $W'(C)=(k+1)(\prod_{i=1}^{j-1}W(C_i))\cdot W'(C_j)$. When $n=k$, the unique Catalan path $C$ in $\mathcal{C}_{n,n}$ is the one with $n$ up-steps followed by $n$ down-steps, which satisfies $W'(C)=1$.

It follows that
\begin{align*}
P(x)&=\sum_{n\geq1}p_nx^n=\sum_{n\geq1}\left(\sum_{C\in\mathcal{C}_n}W'(C)\right)x^n\\
&=\sum_{n\geq1}x^n+\sum_{n\geq 1}x^n\sum_{k=1}^{n-1}(k+1)\sum_{j=1}^k\sum_{\substack{C_1\in\mathcal{C}_{a_1},\cdots,C_j\in\mathcal{C}_{a_j}\\a_1+\cdots+a_j=n-k\\a_1,\cdots,a_{j-1}\geq0,a_j\geq1}}\left(\prod_{i=1}^{j-1}W(C_i)\right)\cdot W'(C_j)\\
&=\sum_{n\geq1}x^n+\sum_{k\geq1}(k+1)x^k\sum_{j=1}^k\sum_{\substack{C_1\in\mathcal{C}_{a_1},\cdots,C_j\in\mathcal{C}_{a_j}\\a_1,\cdots,a_{j-1}\geq0,a_j\geq1}}\left(\prod_{i=1}^{j-1}W(C_i)\right)\cdot W'(C_j)x^{a_1+\cdots+a_j}\\
&=\sum_{n\geq1}x^n+\sum_{k\geq1}(k+1)x^k\sum_{j=1}^k\sum_{a_1,\cdots,a_{j-1}\geq0,a_j\geq1}\left(\prod_{i=1}^{j-1}q_{a_i}x^{a_i}\right)\cdot p_{a_j}x^{a_j}\\
&=\sum_{n\geq1}x^n+\sum_{k\geq 1}(k+1)x^k\sum_{j=1}^k(Q(x))^{j-1}P(x)\\
&=\frac{x}{1-x}+\frac{P(x)}{Q(x)-1}\sum_{k\geq1}(k+1)x^k((Q(x))^k-1)\\
&=\frac{x}{1-x}+\frac{P(x)}{Q(x)-1}((1-xQ(x))^{-2}-(1-x)^{-2})\\
&=\frac{x}{1-x}+\frac{P(x)}{Q(x)-1}(Q(x)-(1-x)^{-2}),
\end{align*}
where in the last step we used $Q(x)=(1-xQ(x))^{-2}$, which follows from the relation $\overline{Q}(x)=x(1-\overline{Q}(x))^{-2}$ we derived above. Solving this for $P(x)$, we get $P(x)=\frac{(Q(x)-1)(1-x)}{2-x}$. Hence, for $n\geq 2$, we have
\begin{align*}
p_n&=[x^n]P(x)=\frac12[x^n]\frac{(Q(x)-1)(1-x)}{1-\frac12x}\\
&=\frac12[x^n](Q(x)-1)\sum_{k\geq0}\frac{x^k}{2^k}-\frac12[x^{n-1}](Q(x)-1)\sum_{k\geq0}\frac{x^k}{2^k}\\
&=\sum_{k=0}^{n-1}\frac{q_{n-k}}{2^{k+1}}-\sum_{k=0}^{n-2}\frac{q_{n-1-k}}{2^{k+1}}=\sum_{k=0}^{n-1}\frac{\binom{3n-3k+1}{n-k}}{2^{k+1}(n-k+1)}-\sum_{k=0}^{n-2}\frac{\binom{3n-3k-2}{n-k-1}}{2^{k+1}(n-k)}\\
&=\frac{\binom{3n+1}{n}}{2(n+1)}+\sum_{k=0}^{n-2}\left(\frac{\binom{3n-3k-2}{n-k-1}}{2^{k+2}(n-k)}-\frac{\binom{3n-3k-2}{n-k-1}}{2^{k+1}(n-k)}\right)\\
&=\frac{\binom{3n+1}{n}}{2(n+1)}-\sum_{k=0}^{n-2}\frac{\binom{3n-3k-2}{n-k-1}}{2^{k+2}(n-k)}=\frac{\binom{3n+1}{n}}{n+1}-\sum_{k=0}^{n-1}\frac{\binom{3n-3k+1}{n-k}}{2^{k+1}(n-k+1)},
\end{align*}
as required. The case when $n=1$ is easy to verify. 
\end{proof}

\section{Congruence classes of generalised parking functions}\label{related}
In this section, we apply similar algebraic techniques as in the previous sections to the work of Novelli and Thibon in \cite{NT} on certain Hopf algebras of generalised parking functions, and obtain results on the graded dimensions of these Hopf algebras, which is also the number of congruence classes of generalised parking functions under several different congruence relations. We begin with some definitions.

For a function $f:[n]\to[N]$, the \textit{evaluation} of $f$ is the sequence $\ev(f)$ of length $N$, whose $j$-th entry is the number of $i\in[n]$ with $f(i)=j$. The \textit{packed evaluation} of $f$ is the sequence $\pev(f)$ obtained by removing all the zero entries from $\ev(f)$. 

In \cite{NT}, Novelli and Thibon studied Hopf algebras of two types of generalised parking functions. For $m\geq1$, an $m$\textit{-multiparking function} of \textit{size} $mn$ is a function $f:[mn]\to[n]$ such that there exists an ordinary parking function $\overline{f}$ satisfying $\ev(f)=m\ev(\overline{f})$, and an $m$\textit{-parking function} of \textit{size} $n$ is a function $f:[n]\to[1+m(n-1)]$ satisfying $f(i)\leq1+m(i-1)$ for all $i\in[n]$. 

One way to obtain a new Hopf algebra is to take the quotient under certain congruences of generalised parking functions. Therefore, the number of congruence classes of generalised parking functions is of interest. Since our work is mostly computational, we will apply relevant results in \cite{NT} directly as black boxes, and refer interested readers to the paper of Novelli and Thibon \cite{NT} for detailed definitions of these different congruences. 

We begin with two results in the setting of $m$-multiparking functions.

\begin{center}
\begin{tabular}{|c|c|c|}
\hline  \multirow{2}{*}{$m$}   & hyposylvester classes of & \multirow{2}{*}{OEIS}\\
   & $m$-multiparking functions of size $mn$
, $1\leq n\leq8$ & \\\hline
%\hline  $m$  & hyposylvester classes of $m$-multiparking functions of size $mn$, $1\leq n\leq8$ & OEIS \\\hline
  1   & 1, 3, 12, 55, 273, 1428, 7752, 43263 & A001764\\\hline
  2   & 1, 4, 21, 126, 818, 5594, 39693, 289510 & A003168\\\hline
  3   & 1, 5, 32, 233, 1833, 15180, 130392, 1151057 & A364922=A243693\\\hline
  4   & 1, 6, 45, 382, 3498, 33696, 336549, 3453750 & A243694\\\hline
  5   & 1, 7, 60, 579, 6017, 65732, 744264, 8656795 & A243695\\\hline
\end{tabular}
\end{center}

\begin{theorem}\label{hyposylvestermultipark}
For every $m\geq1$, the number $p_n^{(m)}$ of hyposylvester classes of $m$-multiparking functions of size $mn$ is
$$\frac1n\sum_{k=0}^{n-1}\binom{n}{k}\binom{3n-k}{2n+1}(m-1)^k.$$
\end{theorem}
\begin{proof}
Note that every ordinary parking function of size $n$ is the permutation of an increasing ordinary parking function $\overline{f}$ with the same evaluation. Also, recall that every increasing ordinary parking function $\overline{f}$ of size $n$ corresponds bijectively to a Catalan path $C$ of size $2n$, and note that this bijection satisfies $\pev(\overline{f})=\mathbf{u}(C)$. From the definition, for every $m$-multiparking function $f$ of size $mn$, there exists an ordinary parking function $\overline{f}$ of size $n$, which can be chosen to be increasing, that satisfies $\ev(f)=m\ev(\overline{f})$, and thus $\pev(f)=m\pev(\overline{f})=m\mathbf{u}(C)$. 

By Lemma 3.2 and Theorem 3.4 in \cite{NT}, $m$-multiparking functions in the same hyposylvester class have the same evaluation and thus the same packed evaluation. Moreover, for every possible evaluation $\alpha$ of an $m$-multiparking function, the number of hyposylvester classes that contain $m$-multiparking functions with common evaluation $\alpha$ depends only on their common packed evaluation $\beta$, and is equal to $\prod_{i=2}^{|\beta|}(1+\beta_i)$. Putting all of these together, we see that the number of hyposylvester classes of $m$-multiparking functions of size $mn$ is 
$$p_n^{(m)}=\sum_{C\in\mathcal{C}_n}\prod_{i=2}^{|\mathbf{u}(C)|}\left(1+m\mathbf{u}(C)_i\right).$$

We now proceed similarly to the proofs of Theorem \ref{123}, \ref{213} and \ref{312321} above. For $n\geq 1$, let
$q_n^{(m)}=\sum_{C\in\mathcal{C}_n}\prod_{i=1}^{|\mathbf{u}(C)|}(1+m\mathbf{u}(C)_i)$, and consider the generating functions $Q_m(x)=1+\sum_{n\geq1}q_n^{(m)}x^n$ and $P_m(x)=\sum_{n\geq1}p_n^{(m)}x^n$. By using the canonical decompositions of Catalan paths, we have 
\begin{align*}
Q_m(x)&=1+\sum_{k\geq1}(1+mk)x^k(Q_m(x))^k\\
&=1+\frac{xQ_m(x)}{1-xQ_m(x)}+\frac{mxQ_m(x)}{(1-xQ_m(x))^2}=\frac{1+(m-1)xQ_m(x)}{(1-xQ_m(x))^2}.
\end{align*}
Thus, $\overline{Q}_m(x)=xQ_m(x)$ satisfies $\overline{Q}_m(x)=x\cdot\frac{1+(m-1)\overline{Q}_m(x)}{(1-\overline{Q}_m(x))^2}$. Using the canonical decompositions again, we get $P_m(x)=\sum_{k\geq1}x^k(Q_m(x))^k=\frac{\overline{Q}_m(x)}{1-\overline{Q}_m(x)}$. Therefore, by Lemma \ref{lift}, we have 
\begin{align*}
p_n^{(m)}&=[x^n]P_m(x)=\frac1n[x^{n-1}]\left(\frac{x}{1-x}\right)'\left(\frac{1+(m-1)x}{(1-x)^2}\right)^n=\frac1n[x^{n-1}]\frac{(1+(m-1)x)^n}{(1-x)^{2n+2}}\\
&=\frac1n\sum_{k=0}^{n-1}\binom{n}{k}(m-1)^k[x^{n-1-k}]\frac1{(1-x)^{2n+2}}\\
&=\frac1n\sum_{k=0}^{n-1}\binom{n}{k}\binom{3n-k}{2n+1}(m-1)^k,
\end{align*}
as required.
\end{proof}

\begin{center}
\begin{tabular}{|c|c|c|}
\hline  \multirow{2}{*}{$m$}   & metasylvester classes of & \multirow{2}{*}{OEIS}\\
   & $m$-multiparking functions of size $mn$
, $1\leq n\leq8$ & \\\hline
%\hline  $m$  & hyposylvester classes of $m$-multiparking functions of size $mn$, $1\leq n\leq8$ & OEIS \\\hline
  1   & 1, 3, 14, 87, 669, 6098, 64050, 759817 & A132624\\\hline
  2   & 1, 4, 27, 254, 3048, 44328, 755681, 14750646 & A243696\\\hline
  3   & 1, 5, 44, 551, 8919, 176634, 4130208, 111222029 & A243697\\\hline
  4   & 1, 6, 65, 1014, 20598, 514604, 15240261, 521457190 & A243698\\\hline
  5   & 1, 7, 90, 1679, 40977, 1234002, 44162294, 1829650545 & A243699\\\hline
\end{tabular}
\end{center}
\begin{theorem}\label{metasylvestermultipark}
For every $m\geq 1$, the number $p_n^{(m)}$ of metasylvester classes of $m$-multiparking functions of size $mn$ is equal to
$$\sum_{C\in\mathcal{C}_n}\prod_{i=2}^{|\mathbf{u}(C)|}\left(1+m\sum_{j=i}^{|\mathbf{u}(C)|}\mathbf{u}(C)_j\right).$$
Consequently, $p_n^{(m)}=\sum_{k=1}^np_{n,k}$, where
$$p_{n,k}=\begin{cases}
1, &\text{if }k=n,\\
\displaystyle (1+m(n-k))\sum_{i=n-k}^{n-1}\sum_{j=k+1-n+i}^ip_{i,j}, &\text{if }1\leq k\leq n-1.\\
\end{cases}$$
Furthermore, the sequence $p_n^{(m)}$ satisfies 
$$\frac{x}{1-x}=\sum_{n=1}^\infty p_n^{(m)}\frac{x^n(1-x)^n}{\prod_{\ell=1}^n(1+m\ell x)}.$$
\end{theorem}
\begin{proof}
By Theorem 3.9 in \cite{NT}, for every possible evaluation of an $m$-multiparking function $\alpha$, the number of metasylvester classes containing $m$-multiparking function with common evaluation $\alpha$ depends only on their common packed evaluation $\beta$, and is equal to $\prod_{i=2}^{|\beta|}(1+\sum_{j=i}^{|\beta|}\beta_j)$. Consequently, similar to the proof of Theorem \ref{hyposylvestermultipark} above, we see that the number $p_n^{(m)}$ of metasylvester classes of $m$-multiparking functions of size $mn$ is equal to 
$$\sum_{C\in\mathcal{C}_n}\prod_{i=2}^{|\mathbf{u}(C)|}\left(1+m\sum_{j=i}^{|\mathbf{u}(C)|}\mathbf{u}(C)_j\right).$$

We now proceed similarly to the proof of Theorem \ref{321}. Let $W(C)=\prod_{i=2}^{|\mathbf{u}(C)|}(1+m\sum_{j=i}^{|\mathbf{u}(C)|}\mathbf{u}(C)_j)$, and let $p_{n,k}=\sum_{C\in\mathcal{C}_{n,k}}W(C)$, so that $p_n^{(m)}=\sum_{k=1}^np_{n,k}$. Note that $p_{n,n}=1$ as $\mathcal{C}_{n,n}$ contains only the Catalan path consisting of $n$ up-steps followed by $n$ down-steps. For $n>k\geq1$ and $C\in\mathcal{C}_{n,k}$, we have $|\mathbf{u}(C)|\geq2$. Let $i'\in[k]$ be the length of the first block of down-steps in $C$ and let $C'\in\mathcal{C}_{n-i',j}$ be obtained by deleting the first peak of $C$. Note that $\mathbf{u}(C')_t=\mathbf{u}(C)_{t+1}$ for all $2\leq t\leq|\mathbf{u}(C')|=|\mathbf{u}(C)|-1$. Therefore,
\begin{align*}
W(C)&=\prod_{i=2}^{|\mathbf{u}(C)|}\left(1+m\sum_{j=i}^{|\mathbf{u}(C)|}\mathbf{u}(C)_j\right)=\left(1+m\sum_{j=2}^{|\mathbf{u}(C)|}\mathbf{u}(C)_j\right)\prod_{i=3}^{|\mathbf{u}(C)|}\left(1+m\sum_{j=i}^{|\mathbf{u}(C)|}\mathbf{u}(C)_j\right)\\
&=\left(1+m\sum_{j=2}^{|\mathbf{u}(C)|}\mathbf{u}(C)_j\right)W(C')=(1+m(n-k))W(C').
\end{align*}
Thus, we have the recurrence relation
\begin{align*}
p_{n,k}&=\sum_{C\in\mathcal{C}_{n,k}}W(C)=\sum_{i'=1}^k\sum_{j=k-i'+1}^{n-i'}\sum_{C'\in\mathcal{C}_{n-i',j}}(1+m(n-k))W(C')\\
&=(1+m(n-k))\sum_{i'=1}^k\sum_{j=k-i'+1}^{n-i'}p_{n-i,j}=(1+m(n-k))\sum_{i=n-k}^{n-1}\sum_{j=k+1-n+i}^ip_{i,j}.
\end{align*}
In particular, we have $p_{n,1}=(1+m(n-1))\sum_{j=1}^{n-1}p_{n-1,j}=(1+m(n-1))p_{n-1}^{(m)}$.

For every $k\geq 0$, let $P_k(x)=\sum_{n=k+1}^\infty p_{n,n-k}x^{n-k-1}$ be the generating function of the terms lying on the $k$-th diagonal. Then, $P_k(0)=[x^0]P_k(x)=p_{k+1,1}=(1+mk)p_k^{(m)}$. For $k\geq 0$, using the recurrence above, we have 
\begin{align*}
P_{k+1}(x)&=\sum_{n=k+2}^\infty p_{n,n-k-1}x^{n-k-2}\\
&=\sum_{n=k+2}^\infty(1+m(k+1))\sum_{i=k+1}^{n-1}\sum_{j=0}^kp_{i,i-j}x^{n-k-2}\\
&=(1+m(k+1))x^{-k}\sum_{j=0}^k\sum_{i=k+1}^\infty p_{i,i-j}x^{i-1}\sum_{n=i+1}^\infty x^{n-i-1}\\
&=\frac{1+m(k+1)}{x^k(1-x)}\sum_{j=0}^k\sum_{i=k+1}^\infty p_{i,i-j}x^{i-1}.
\end{align*}
Let $S_k(x)=\sum_{j=0}^k\sum_{i=k+1}^\infty p_{i,i-j}x^{i-1}$, then we have $x^kP_{k+1}(x)=\frac{1+m(k+1)}{1-x}S_k(x)$. Furthermore, we have
\begin{align*}
S_{k+1}(x)&=\sum_{j=0}^{k+1}\sum_{i=k+2}^\infty p_{i,i-j}x^{i-1}\\
&=S_k(x)-\sum_{j=0}^kp_{k+1,k+1-j}x^k+\sum_{i=k+2}^\infty p_{i,i-k-1}x^{i-1}\\
&=S_k(x)-\sum_{j=1}^{k+1}p_{k+1,j}x^k+x^{k+1}P_{k+1}(x)\\
&=S_k(x)-p^{(m)}_{k+1}x^k+\frac{(1+m(k+1))x}{1-x}S_k(x)\\
&=\left(\frac{1+m(k+1)x}{1-x}\right)S_k(x)-p^{(m)}_{k+1}x^k,
\end{align*}
and therefore by telescoping,
\begin{align*}
&\sum_{k=1}^\infty p_k^{(m)}\frac{x^k(1-x)^k}{\prod_{\ell=1}^k(1+m\ell x)}=\sum_{k=0}^\infty p_{k+1}^{(m)}\frac{x^{k+1}(1-x)^{k+1}}{\prod_{\ell=1}^{k+1}(1+m\ell x)}\\
=&\sum_{k=0}^\infty\left(\left(\frac{1+m(k+1)x}{1-x}\right)S_k(x)-S_{k+1}(x)\right)\frac{x(1-x)^{k+1}}{\prod_{\ell=1}^{k+1}(1+m\ell x)}\\
=&\frac{1+mx}{1-x}\frac{x(1-x)}{(1+mx)}S_0(x)+\sum_{k=1}^\infty S_k(x)\left(\frac{1+m(k+1)x}{1-x}\frac{x(1-x)^{k+1}}{\prod_{\ell=1}^{k+1}(1+m\ell x)}-\frac{x(1-x)^k}{\prod_{\ell=1}^{k}(1+m\ell x)}\right)\\
=&xS_0(x)=\sum_{i=1}^\infty p_{i,i}x^i=\frac{x}{1-x},
\end{align*}
as claimed.
\end{proof}

Note that setting $m=1$ finishes the proof of Theorem \ref{312}, and shows that both the sequence $\pk_n(312)$ of size $n$ parking functions whose parking permutations avoid the pattern 312, and the sequence $p_n$ of the number of metasylvester classes of size $n$ parking functions match the sequence A132624 in OEIS. The latter confirms a conjecture of Novelli and Thibon in \cite{NT}. It remains to be seen whether more explicit formulas exist for $p_n^{(m)}$.

We now prove two results on $m$-parking functions. Let $\mathcal{C}_n^{(m)}$ be the set of $m$-Catalan paths of length $(m+1)n$, each consists of $n$ up-steps of size $m$, and $mn$ down-steps of size 1, and never goes below the $x$-axis. Setting $m=1$ recovers the usual Catalan paths. For every increasing $m$-parking function $f$ of size $n$ and for every $j\in[mn]$ in increasing order, draw $t$ up-steps of size $m$ followed by a down-step of size 1 if $|f^{-1}(j)|=t$. It is easy to see using the definition that this gives a bijective correspondence between increasing $m$-parking functions of size $n$ and the set $\mathcal{C}_n^{(m)}$.

By \cite{NT}, the number of hyposylvester classes of $m$-parking functions of size $n$ is $\frac1{2mn+1}\binom{(2m+1)n}{n}$. In the following theorem, analogous to results above, we express the number of metasylvester classes of $m$-parking functions of size $n$ as a sum over $\mathcal{C}_n^{(m)}$ of terms involving entries in $\mathbf{u}(C)$, which is defined in the same way as before for ordinary Catalan paths. Unfortunately, the method of deleting the first peak that we used in the proof of Theorem \ref{metasylvestermultipark} does not translate well to the setting of $m$-Catalan paths, and we could not obtain more explicit expressions. We thank Jared León for his programming help for the following table of values. 

\begin{center}
\begin{tabular}{|c|c|c|}
\hline  \multirow{2}{*}{$m$}   & metasylvester classes of & \multirow{2}{*}{OEIS}\\
   & $m$-parking functions of size $n$
, $1\leq n\leq8$ & \\\hline
  1   & 1, 3, 14, 87, 669, 6098, 64050, 759817 & A132624\\\hline
  2   & 1, 5, 45, 585, 9944, 208783, 5218212, 151283473 & A243678\\\hline
  3   & 1, 7, 94, 1879, 50006, 1663866, 66483078, 3101878511 & A243679\\\hline
  4   & 1, 9, 161, 4353, 158035, 7212505, 396783811, 25558807077 & A243682\\\hline
  5   & 	1, 11, 246, 8391, 386211, 22414326, 1571290734, 129166342089 & A243683\\\hline
\end{tabular}
\end{center}
\begin{theorem}\label{metasylvestermpark}
The number of metasylvester classes of $m$-parking functions of size $n$ is 
$$\sum_{C\in\mathcal{C}_n^{(m)}}\prod_{i=2}^{|\mathbf{u}(C)|}\left(1+\sum_{j=i}^{|\mathbf{u}(C)|}\mathbf{u}(C)_j\right).$$
\end{theorem}
\begin{proof}
Every $m$-parking function $f$ of size $n$ is the permutation of an increasing $m$-parking function $\overline{f}$, and every increasing $m$-parking function $\overline{f}$ of size $n$ corresponds bijectively to an $m$-Catalan path $C$ of length $(m+1)n$, with $\pev(\overline{f})=\mathbf{u}(C)$. Consequently, similar to the proof of Theorem \ref{metasylvestermultipark} and using Theorem 3.9 in \cite{NT}, we see that the number of metasylvester classes of $m$-multiparking functions of size $n$ is equal to 
$$\sum_{C\in\mathcal{C}_n^{(m)}}\prod_{i=2}^{|\mathbf{u}(C)|}\left(1+\sum_{j=i}^{|\mathbf{u}(C)|}\mathbf{u}(C)_j\right),$$ as required.
\end{proof}

Finally, we give an explicit formula for the number of hypoplactic classes of $m$-parking functions of size $n$, which turns out to coincide with the small $(m+1)$-Schr\"{o}der numbers defined by Yang and Jiang in \cite{YJ}, confirming some conjectures about this family of sequences on OEIS.

\begin{center}
\begin{tabular}{|c|c|c|}
\hline  \multirow{2}{*}{$m$}   & hypoplactic classes of & \multirow{2}{*}{OEIS}\\
   & $m$-parking functions of size $n$
, $1\leq n\leq8$ & \\\hline
%\hline  $m$  & hyposylvester classes of $m$-multiparking functions of size $mn$, $1\leq n\leq8$ & OEIS \\\hline
  1   & 1, 3, 11, 45, 197, 903, 4279, 20793 & A001003\\\hline
  2   & 1, 5, 33, 249, 2033, 17485, 156033, 1431281 & A034015\\\hline
  3   & 1, 7, 67, 741, 8909, 113107, 1492103, 20251945 & A243675=A371398\\\hline
  4   & 1, 9, 113, 1649, 26225, 440985, 7711009, 138792929 & A243676\\\hline
  5   & 1, 11, 171, 3101, 61381, 1285663, 28015735, 628599577 & A243677\\\hline
\end{tabular}
\end{center}
\begin{theorem}\label{hypoplacticmpark}
The number of hypoplactic classes of $m$-parking functions of size $n$ is 
$$\frac1n\sum_{k=1}^n\binom{mn}{k-1}\binom{n}{k}2^{k-1}.$$
\end{theorem}
\begin{proof}
By Section 2.6 in \cite{NT}, the number of hypoplactic classes containing $m$-parking functions with common evaluation $\alpha$ depends only on their common packed evaluation $\beta$, and is equal to $2^{|\beta|-1}$. As a result, similar to the proof of Theorem \ref{metasylvestermpark} and using the bijection between increasing $m$-parking functions and $m$-Catalan paths, the number of hypoplactic classes of $m$-parking functions of size $n$ is equal to 
$$\sum_{C\in\mathcal{C}_n^{(m)}}2^{|\mathbf{u}(C)|-1}.$$

Note that for every $C\in\mathcal{C}_n^{(m)}$, $|\mathbf{u}(C)|$ is exactly the number of peaks in $C$. Since the number of $m$-Catalan paths with length $(m+1)n$ and exactly $k$ peaks is known \cite{YJ} to be the $m$-Narayana numbers $N^{(m)}_{n,k}=\frac1n\binom{mn}{k-1}\binom{n}{k}$, it follows that the number of hypoplactic classes of $m$-parking functions of size $n$ is equal to 
$$\frac1n\sum_{k=1}^n\binom{mn}{k-1}\binom{n}{k}2^{k-1},$$
by splitting the sum over $m$-Catalan paths according to the number of peaks.
\end{proof}

\bibliographystyle{abbrv}
\bibliography{bibliography}
\end{document}

%% file: parkingfunctionexample.tex
\begin{tabular}{|c|c|}
\hline\rule{0pt}{4ex}  A parking function $f:[7]\to[7]$  &  \begin{tabular}{|c||c|c|c|c|c|c|c|c|}
 \hline$i$ & 1 & 2 & 3 & 4 & 5 & 6 & 7\\\hline
 $f(i)$ & 4 & 4 & 6 & 4 & 2 & 2 & 1\\\hline 
\end{tabular}\\[2ex]\hline
\rule{0pt}{3ex}  The parking positions of $f$  & \begin{tabular}{|c|c|c|c|c|c|c|c|}
 \hline 7 & 5 & 6 & 1 & 2 & 3 & 4\\\hline
\end{tabular}  \\[1ex]\hline
\rule{0pt}{3ex}   The parking permutation $\rho_f$ associated to $f$  & $\rho_f=7561234$ \\[1ex]\hline
\rule{0pt}{3ex}  $f$ in block notation   & $(\{7\},\{5,6\},\emptyset,\{1,2,4\},\emptyset,\{3\},\emptyset)$ \\[1ex]\hline
\rule{0pt}{3ex}   The block permutation $\pi_f$ associated to $f$  & $\pi_f=7561243$ \\[1ex]\hline
\end{tabular}

%% file: 123132example.tex
\renewcommand{\arraystretch}{1.2}
\newcolumntype{C}{ >{\centering\arraybackslash} m{0.8cm} }
\newcolumntype{D}{ >{\centering\arraybackslash} m{1.8cm} }
\newcolumntype{E}{ >{\centering\arraybackslash} m{2.4cm} }
\newcolumntype{F}{ >{\centering\arraybackslash} m{3cm} }

\begin{tabular}{F>{\centering\arraybackslash} m{0.1cm}F}
\begin{tabular}{|C|C|}
\hline $f$ & $()$\\\hline
$T$ & \begin{tikzpicture}
\draw (0,0.2) node{};
\draw (0,0) node(root)[inner sep=0.3ex,circle,fill=black]{};
\draw (0,-0.5) node(0)[inner sep=0.15ex,circle,fill=black]{};
\node [left] at (0,-0.5) {\tiny 0};
\draw (root) -- (0);
\end{tikzpicture}\\\hline
\end{tabular}  & \quad & \begin{tabular}{|C|C|}
\hline$f$ & $(\{1\})$\\\hline
$T$ & \begin{tikzpicture}
\draw (0,0.2) node{};
\draw (0,0) node(root)[inner sep=0.3ex,circle,fill=black]{};
\draw (0,-0.25) node(0)[inner sep=0.15ex,circle,fill=black]{};
\node [left] at (0,-0.25) {\tiny 0};
\draw (0,-0.5) node(1)[inner sep=0.15ex,circle,fill=black]{};
\node [left] at (0,-0.5) {\tiny 1};
\draw (root) -- (1);
\end{tikzpicture}\\\hline
\end{tabular}\\
\end{tabular}

\vspace{0.2cm}

\begin{tabular}{|C|D|D|D|}
\hline$f$ & $(\{1\},\{2\})$ & $(\{1,2\},\emptyset)$ & $(\{2\},\{1\})$ \\\hline
$T$ & \begin{tikzpicture}
\draw (0,0.2) node{};
\draw (0,0) node(root)[inner sep=0.3ex,circle,fill=black]{};
\draw (0,-0.4) node(0)[inner sep=0.15ex,circle,fill=black]{};
\node [left] at (0,-0.4) {\tiny 0};
\draw (0.5,-0.75) node(1)[inner sep=0.15ex,circle,fill=black]{};
\node [right] at (0.5,-0.75) {\tiny 1};
\draw (-0.5,-0.75) node(2)[inner sep=0.15ex,circle,fill=black]{};
\node [left] at (-0.5,-0.75) {\tiny 2};
\draw (root) -- (0);
\draw (1) -- (0);
\draw (2) -- (0);
\end{tikzpicture} & \begin{tikzpicture}
\draw (0,0.2) node{};
\draw (0,0) node(root)[inner sep=0.3ex,circle,fill=black]{};
\draw (0,-0.5) node(1)[inner sep=0.15ex,circle,fill=black]{};
\node [left] at (0,-0.5) {\tiny 1};
\draw (0.5,-0.5) node(0)[inner sep=0.15ex,circle,fill=black]{};
\node [right] at (0.5,-0.5) {\tiny 0};
\draw (-0.5,-0.5) node(2)[inner sep=0.15ex,circle,fill=black]{};
\node [left] at (-0.5,-0.5) {\tiny 2};
\draw (root) -- (0);
\draw (root) -- (1);
\draw (root) -- (2);
\end{tikzpicture} & \begin{tikzpicture}
\draw (0,0.2) node{};
\draw (0,0) node(root)[inner sep=0.3ex,circle,fill=black]{};
\draw (0,-0.25) node(0)[inner sep=0.15ex,circle,fill=black]{};
\node [left] at (0,-0.25) {\tiny 0};
\draw (0,-0.5) node(1)[inner sep=0.15ex,circle,fill=black]{};
\node [left] at (0,-0.5) {\tiny 1};
\draw (0,-0.75) node(2)[inner sep=0.15ex,circle,fill=black]{};
\node [left] at (0,-0.75) {\tiny 2};
\draw (root) -- (2);
\end{tikzpicture}\\\hline
\end{tabular}

\vspace{0.2cm}

\begin{tabular}{|C|E|E|E|E|}
\hline$f$ & $(\{2\},\{1\},\{3\})$ & $(\{2\},\{1,3\},\emptyset)$ & $(\{2\},\{3\},\{1\})$ & $(\{2,3\},\{1\},\emptyset)$ \\\hline
$T$ & \begin{tikzpicture}
\draw (0,0.2) node{};
\draw (0,0) node(root)[inner sep=0.3ex,circle,fill=black]{};
\draw (0,-0.3) node(0)[inner sep=0.15ex,circle,fill=black]{};
\node [left] at (0,-0.3) {\tiny 0};
\draw (0.5,-0.6) node(1)[inner sep=0.15ex,circle,fill=black]{};
\node [right] at (0.5,-0.6) {\tiny 1};
\draw (0.5,-0.9) node(2)[inner sep=0.15ex,circle,fill=black]{};
\node [right] at (0.5,-0.9) {\tiny 2};
\draw (-0.5,-0.6) node(3)[inner sep=0.15ex,circle,fill=black]{};
\node [left] at (-0.5,-0.6) {\tiny 3};
\draw (root) -- (0);
\draw (1) -- (0);
\draw (3) -- (0);
\draw (1) -- (2);
\end{tikzpicture} & \begin{tikzpicture}
\draw (0,0.2) node{};
\draw (0,0) node(root)[inner sep=0.3ex,circle,fill=black]{};
\draw (0,-0.45) node(1)[inner sep=0.15ex,circle,fill=black]{};
\node [left] at (0,-0.45) {\tiny 1};
\draw (0,-0.9) node(2)[inner sep=0.15ex,circle,fill=black]{};
\node [left] at (0,-0.9) {\tiny 2};
\draw (0.5,-0.45) node(0)[inner sep=0.15ex,circle,fill=black]{};
\node [right] at (0.5,-0.45) {\tiny 0};
\draw (-0.5,-0.45) node(3)[inner sep=0.15ex,circle,fill=black]{};
\node [left] at (-0.5,-0.45) {\tiny 3};
\draw (root) -- (0);
\draw (root) -- (3);
\draw (root) -- (2);
\end{tikzpicture} & \begin{tikzpicture}
\draw (0,0.2) node{};
\draw (0,0) node(root)[inner sep=0.3ex,circle,fill=black]{};
\draw (0,-0.3) node(0)[inner sep=0.15ex,circle,fill=black]{};
\node [left] at (0,-0.3) {\tiny 0};
\draw (0,-0.6) node(1)[inner sep=0.15ex,circle,fill=black]{};
\node [left] at (0,-0.6) {\tiny 1};
\draw (0.5,-0.9) node(2)[inner sep=0.15ex,circle,fill=black]{};
\node [right] at (0.5,-0.9) {\tiny 2};
\draw (-0.5,-0.9) node(3)[inner sep=0.15ex,circle,fill=black]{};
\node [left] at (-0.5,-0.9) {\tiny 3};
\draw (root) -- (1);
\draw (2) -- (1);
\draw (3) -- (1);
\end{tikzpicture} & \begin{tikzpicture}
\draw (0,0.2) node{};
\draw (0,0) node(root)[inner sep=0.3ex,circle,fill=black]{};
\draw (0,-0.45) node(2)[inner sep=0.15ex,circle,fill=black]{};
\node [left] at (0,-0.45) {\tiny 2};
\draw (0.5,-0.9) node(1)[inner sep=0.15ex,circle,fill=black]{};
\node [right] at (0.5,-0.9) {\tiny 1};
\draw (0.5,-0.45) node(0)[inner sep=0.15ex,circle,fill=black]{};
\node [right] at (0.5,-0.45) {\tiny 0};
\draw (-0.5,-0.45) node(3)[inner sep=0.15ex,circle,fill=black]{};
\node [left] at (-0.5,-0.45) {\tiny 3};
\draw (root) -- (0);
\draw (0) -- (1);
\draw (root) -- (3);
\draw (root) -- (2);
\end{tikzpicture}\\
\hline$f$ & $(\{2,3\},\emptyset,\{1\})$ & $(\{3\},\{2\},\{1\})$ & $(\{3\},\{1\},\{2\})$ & $(\{3\},\{1,2\},\emptyset)$ \\\hline
$T$ & \begin{tikzpicture}
\draw (0,0.2) node{};
\draw (0,0) node(root)[inner sep=0.3ex,circle,fill=black]{};
\draw (0,-0.45) node(0)[inner sep=0.15ex,circle,fill=black]{};
\node [left] at (0,-0.45) {\tiny 0};
\draw (0,-0.9) node(2)[inner sep=0.15ex,circle,fill=black]{};
\node [left] at (0,-0.9) {\tiny 2};
\draw (0.5,-0.9) node(1)[inner sep=0.15ex,circle,fill=black]{};
\node [right] at (0.5,-0.9) {\tiny 1};
\draw (-0.5,-0.9) node(3)[inner sep=0.15ex,circle,fill=black]{};
\node [left] at (-0.5,-0.9) {\tiny 3};
\draw (1) -- (0);
\draw (0) -- (3);
\draw (root) -- (2);
\end{tikzpicture} & \begin{tikzpicture}
\draw (0,0.2) node{};
\draw (0,0) node(root)[inner sep=0.3ex,circle,fill=black]{};
\draw (0,-0.225) node(0)[inner sep=0.15ex,circle,fill=black]{};
\node [left] at (0,-0.225) {\tiny 0};
\draw (0,-0.45) node(1)[inner sep=0.15ex,circle,fill=black]{};
\node [left] at (0,-0.45) {\tiny 1};
\draw (0,-0.675) node(2)[inner sep=0.15ex,circle,fill=black]{};
\node [left] at (0,-0.675) {\tiny 2};
\draw (0,-0.9) node(3)[inner sep=0.15ex,circle,fill=black]{};
\node [left] at (0,-0.9) {\tiny 3};
\draw (root) -- (3);
\end{tikzpicture} & \begin{tikzpicture}
\draw (0,0.2) node{};
\draw (0,0) node(root)[inner sep=0.3ex,circle,fill=black]{};
\draw (0,-0.3) node(0)[inner sep=0.15ex,circle,fill=black]{};
\node [left] at (0,-0.3) {\tiny 0};
\draw (0.5,-0.6) node(1)[inner sep=0.15ex,circle,fill=black]{};
\node [right] at (0.5,-0.6) {\tiny 1};
\draw (-0.5,-0.9) node(3)[inner sep=0.15ex,circle,fill=black]{};
\node [left] at (-0.5,-0.9) {\tiny 3};
\draw (-0.5,-0.6) node(2)[inner sep=0.15ex,circle,fill=black]{};
\node [left] at (-0.5,-0.6) {\tiny 2};
\draw (root) -- (0);
\draw (1) -- (0);
\draw (3) -- (2);
\draw (0) -- (2);
\end{tikzpicture} & \begin{tikzpicture}
\draw (0,0.2) node{};
\draw (0,0) node(root)[inner sep=0.3ex,circle,fill=black]{};
\draw (0,-0.45) node(1)[inner sep=0.15ex,circle,fill=black]{};
\node [left] at (0,-0.45) {\tiny 1};
\draw (-0.5,-0.9) node(3)[inner sep=0.15ex,circle,fill=black]{};
\node [left] at (-0.5,-0.9) {\tiny 3};
\draw (0.5,-0.45) node(0)[inner sep=0.15ex,circle,fill=black]{};
\node [right] at (0.5,-0.45) {\tiny 0};
\draw (-0.5,-0.45) node(2)[inner sep=0.15ex,circle,fill=black]{};
\node [left] at (-0.5,-0.45) {\tiny 2};
\draw (root) -- (0);
\draw (root) -- (1);
\draw (2) -- (3);
\draw (root) -- (2);
\end{tikzpicture}\\\hline
\end{tabular}

%% file: pqexample.tex
\renewcommand{\arraystretch}{1.2}
\newcolumntype{F}{ >{\centering\arraybackslash} m{4cm} }

\begin{tabular}{FF}
\begin{tikzpicture}
\draw (0,0.5) node{};
\draw (0,0) node(root)[inner sep=0.15ex,circle,fill=black]{};
\draw (0,-0.5) node(0)[inner sep=0.15ex,circle,fill=black]{};
\draw (0,-1) node(1)[inner sep=0.15ex,circle,fill=black]{};

\draw (0,-2) node(3)[inner sep=0.15ex,circle,fill=black]{};
\draw (0,-2.5) node(4)[inner sep=0.15ex,circle,fill=black]{};
\draw (root) -- (1);
\draw (3) -- (4); 

\node [right] at (0,-0.5) {\tiny $k+1$};
\node [right] at (0,-1) {\tiny $k+2$};
\node at (0,-1.35) {\tiny $\vdots$};
\node [right] at (0,-2) {\tiny $n-1$};
\node [right] at (0,-2.5) {\tiny $n$};
\end{tikzpicture} 
& \begin{tikzpicture}
\draw (0,0.5) node{};
\draw (0,0) node(root)[inner sep=0.15ex,circle,fill=black]{};
\draw (0.25,-0.5) node(0)[inner sep=0.15ex,circle,fill=black]{};
\draw (0.25,-1) node(1)[inner sep=0.15ex,circle,fill=black]{};

\draw (0.25,-2) node(3)[inner sep=0.15ex,circle,fill=black]{};
\draw (0.25,-2.5) node(4)[inner sep=0.15ex,circle,fill=black]{};
\draw (-0.25,-0.5) node(5)[inner sep=0.15ex,circle,fill=black]{};
\draw (5) -- (root) -- (0) -- (1);
\draw (3) -- (4); 

\node [right] at (0.25,-0.5) {\tiny $k+1$};
\node [right] at (0.25,-1) {\tiny $k+2$};
\node at (0.25,-1.35) {\tiny $\vdots$};
\node [right] at (0.25,-2) {\tiny $n-2$};
\node [right] at (0.25,-2.5) {\tiny $n-1$};
\node [left] at (-0.25,-0.5) {\tiny $n$};
\end{tikzpicture}\\
$P_{k+1,n}$ & $Q_{k+1,n}$
\end{tabular}

%% file: 123132big.tex
\renewcommand{\arraystretch}{1.2}
\newcolumntype{C}{ >{\centering\arraybackslash} m{1.8cm} }
\newcolumntype{D}{ >{\centering\arraybackslash} m{4.9cm} }

\begin{tabular}{c||c|c|c|c|c|c|c|c}
  block & $\{24\}$ & $\{23,25\}$ & $\{21\}$ & $\emptyset$ & $\{20\}$ & $\{19,22\}$ & $\{17\}$ & $\{16,18\}$ \\\hline
  cluster number & 1 & 1 & 2 & 1 & 2 & 2 & 3 & 3
\end{tabular}

\mbox{}

\mbox{}

\begin{tabular}{c||c|c|c|c|c|c|c|c}
  block & $\{15\}$ & $\{14\}$ & $\emptyset$ & $\{13\}$ & $\{12\}$ & $\{10\}$ & $\{9\}$ & $\{8,11\}$\\\hline
  cluster number & 4 & 4 & 3 & 4 & 4 & 5 & 5 & 5
\end{tabular}

\mbox{}

\mbox{}

\begin{tabular}{c||c|c|c|c|c|c|c|c|c}
  block & $\{6\}$ & $\{5,7\}$ & $\{3\}$ & $\emptyset$ & $\{2\}$ & $\{1\}$ & $\emptyset$ & $\{4\}$ & $\emptyset$\\\hline
  cluster number & 6 & 6 & 7 & 6 & 7 & 7 & 5 & 7 & 2
\end{tabular}

\mbox{}

\mbox{}

\begin{tabular}{c||c|c|c|c|c|c|c}
  cluster number & 1 & 2 & 3 & 4 & 5 & 6 & 7\\\hline
  cluster type & jump & jump & jump & extend & jump & jump & branch
\end{tabular}

\mbox{}

\mbox{}

\begin{tabular}{CDCD}
$\xrightarrow{\mbox{    }\mbox{\normalsize initialise}\mbox{    }}$ & 
\begin{tikzpicture}
\draw[red,thick] (0,0) -- (0,-0.5);

\draw (0,0) node(root)[inner sep=0.3ex,circle,fill=black]{};
\draw (0,-0.5) node(0)[inner sep=0.15ex,circle,fill=black]{};
\end{tikzpicture}
& $\xrightarrow{\mbox{    }\mbox{\normalsize $\Phi_4(f_7)$}\mbox{    }}$ & 
\begin{tikzpicture}
\draw (0,0) -- (0,-0.5);
\draw[red,thick] (0,-0.5) -- (0.5,-1) -- (0.5,-2);
\draw[red,thick] (0,-0.5) -- (-0.5,-1);

\draw (0,0) node(root)[inner sep=0.3ex,circle,fill=black]{};
\draw (0,-0.5) node(0)[inner sep=0.15ex,circle,fill=black]{};
\draw (0.5,-1) node(1)[inner sep=0.15ex,circle,fill=black]{};
\draw (0.5,-1.5) node(2)[inner sep=0.15ex,circle,fill=black]{};
\draw (0.5,-2) node(3)[inner sep=0.15ex,circle,fill=black]{};
\draw (-0.5,-1) node(4)[inner sep=0.15ex,circle,fill=black]{};

\node [right] at (0) {\tiny 0};
\node [right] at (1) {\tiny 1};
\node [right] at (2) {\tiny 2};
\node [right] at (3) {\tiny 3};
\node [right] at (4) {\tiny 4};
\end{tikzpicture}\\

$\xrightarrow{\mbox{    }\mbox{\normalsize $\Phi_7(f_6)$}\mbox{    }}$ & 
\begin{tikzpicture}
\draw (0,0) -- (0,-0.5);
\draw (-0.5,-1) -- (0,-0.5) -- (0.5,-1) -- (0.5,-1.5) -- (1,-2);
\draw[red,thick] (0,-2) -- (0.5,-1.5) -- (0.5,-2.5);

\draw (0,0) node(root)[inner sep=0.3ex,circle,fill=black]{};
\draw (0,-0.5) node(0)[inner sep=0.15ex,circle,fill=black]{};
\draw (0.5,-1) node(1)[inner sep=0.15ex,circle,fill=black]{};
\draw (0.5,-1.5) node(2)[inner sep=0.15ex,circle,fill=black]{};
\draw (1,-2) node(3)[inner sep=0.15ex,circle,fill=black]{};
\draw (-0.5,-1) node(4)[inner sep=0.15ex,circle,fill=black]{};
\draw (0.5,-2) node(5)[inner sep=0.15ex,circle,fill=black]{};
\draw (0.5,-2.5) node(6)[inner sep=0.15ex,circle,fill=black]{};
\draw (0,-2) node(7)[inner sep=0.15ex,circle,fill=black]{};

\node [right] at (0) {\tiny 0};
\node [right] at (1) {\tiny 1};
\node [right] at (2) {\tiny 2};
\node [right] at (3) {\tiny 3};
\node [right] at (4) {\tiny 4};
\node [right] at (5) {\tiny 5};
\node [right] at (6) {\tiny 6};
\node [right] at (7) {\tiny 7};
\end{tikzpicture} & $\xrightarrow{\mbox{    }\mbox{\normalsize $\Phi_{11}(f_5)$}\mbox{    }}$ &
\begin{tikzpicture}
\draw (0,0) -- (0,-0.5) -- (0.25,-1);
\draw (-0.25,-1) -- (0,-0.5) -- (0.75,-1) -- (0.75,-1.5) -- (1.25,-2);
\draw (0.25,-2) -- (0.75,-1.5) -- (0.75,-2.5);
\draw[red,thick] (-0.75,-1) -- (0,-0.5) -- (-0.25,-1) -- (-0.25,-2);

\draw (0,0) node(root)[inner sep=0.3ex,circle,fill=black]{};
\draw (0,-0.5) node(0)[inner sep=0.15ex,circle,fill=black]{};
\draw (0.75,-1) node(1)[inner sep=0.15ex,circle,fill=black]{};
\draw (0.75,-1.5) node(2)[inner sep=0.15ex,circle,fill=black]{};
\draw (1.25,-2) node(3)[inner sep=0.15ex,circle,fill=black]{};
\draw (0.25,-1) node(4)[inner sep=0.15ex,circle,fill=black]{};
\draw (0.75,-2) node(5)[inner sep=0.15ex,circle,fill=black]{};
\draw (0.75,-2.5) node(6)[inner sep=0.15ex,circle,fill=black]{};
\draw (0.25,-2) node(7)[inner sep=0.15ex,circle,fill=black]{};
\draw (-0.25,-1) node(8)[inner sep=0.15ex,circle,fill=black]{};
\draw (-0.25,-1.5) node(9)[inner sep=0.15ex,circle,fill=black]{};
\draw (-0.25,-2) node(10)[inner sep=0.15ex,circle,fill=black]{};
\draw (-0.75,-1) node(11)[inner sep=0.15ex,circle,fill=black]{};

\node [right] at (0) {\tiny 0};
\node [right] at (1) {\tiny 1};
\node [right] at (2) {\tiny 2};
\node [right] at (3) {\tiny 3};
\node [right] at (4) {\tiny 4};
\node [right] at (5) {\tiny 5};
\node [right] at (6) {\tiny 6};
\node [right] at (7) {\tiny 7};
\node [right] at (8) {\tiny 8};
\node [right] at (9) {\tiny 9};
\node [right] at (10) {\tiny 10};
\node [right] at (11) {\tiny 11};
\end{tikzpicture}\\

$\xrightarrow{\mbox{    }\mbox{\normalsize $\Phi_{15}(f_4)$}\mbox{    }}$ &
\begin{tikzpicture}
\draw (0,0) -- (0,-0.5) -- (0.25,-1);
\draw (-0.25,-1) -- (0,-0.5) -- (0.75,-1) -- (0.75,-1.5) -- (1.25,-2);
\draw (0.25,-2) -- (0.75,-1.5) -- (0.75,-2.5);
\draw (-0.75,-1) -- (0,-0.5) -- (-0.25,-1) -- (-0.25,-2);
\draw[red,thick] (-0.75,-1) -- (-0.75,-3);

\draw (0,0) node(root)[inner sep=0.3ex,circle,fill=black]{};
\draw (0,-0.5) node(0)[inner sep=0.15ex,circle,fill=black]{};
\draw (0.75,-1) node(1)[inner sep=0.15ex,circle,fill=black]{};
\draw (0.75,-1.5) node(2)[inner sep=0.15ex,circle,fill=black]{};
\draw (1.25,-2) node(3)[inner sep=0.15ex,circle,fill=black]{};
\draw (0.25,-1) node(4)[inner sep=0.15ex,circle,fill=black]{};
\draw (0.75,-2) node(5)[inner sep=0.15ex,circle,fill=black]{};
\draw (0.75,-2.5) node(6)[inner sep=0.15ex,circle,fill=black]{};
\draw (0.25,-2) node(7)[inner sep=0.15ex,circle,fill=black]{};
\draw (-0.25,-1) node(8)[inner sep=0.15ex,circle,fill=black]{};
\draw (-0.25,-1.5) node(9)[inner sep=0.15ex,circle,fill=black]{};
\draw (-0.25,-2) node(10)[inner sep=0.15ex,circle,fill=black]{};
\draw (-0.75,-1) node(11)[inner sep=0.15ex,circle,fill=black]{};
\draw (-0.75,-1.5) node(12)[inner sep=0.15ex,circle,fill=black]{};
\draw (-0.75,-2) node(13)[inner sep=0.15ex,circle,fill=black]{};
\draw (-0.75,-2.5) node(14)[inner sep=0.15ex,circle,fill=black]{};
\draw (-0.75,-3) node(15)[inner sep=0.15ex,circle,fill=black]{};

\node [right] at (0) {\tiny 0};
\node [right] at (1) {\tiny 1};
\node [right] at (2) {\tiny 2};
\node [right] at (3) {\tiny 3};
\node [right] at (4) {\tiny 4};
\node [right] at (5) {\tiny 5};
\node [right] at (6) {\tiny 6};
\node [right] at (7) {\tiny 7};
\node [right] at (8) {\tiny 8};
\node [right] at (9) {\tiny 9};
\node [right] at (10) {\tiny 10};
\node [right] at (11) {\tiny 11};
\node [right] at (12) {\tiny 12};
\node [right] at (13) {\tiny 13};
\node [right] at (14) {\tiny 14};
\node [right] at (15) {\tiny 15};
\end{tikzpicture} & $\xrightarrow{\mbox{    }\mbox{\normalsize $\Phi_{18}(f_3)$}\mbox{    }}$ &
\begin{tikzpicture}
\draw (0,0) -- (0,-0.5) -- (0.25,-1);
\draw (-0.25,-1) -- (0,-0.5) -- (0.75,-1) -- (0.75,-1.5) -- (1.25,-2);
\draw (0.25,-2) -- (0.75,-1.5) -- (0.75,-2.5);
\draw (-0.75,-1) -- (0,-0.5) -- (-0.25,-1) -- (-0.25,-2);
\draw (-0.75,-1) -- (-0.75,-3);
\draw[red,thick] (-1.75,-2) -- (-0.75,-1.5) -- (-1.25,-2) -- (-1.25,-2.5);

\draw (0,0) node(root)[inner sep=0.3ex,circle,fill=black]{};
\draw (0,-0.5) node(0)[inner sep=0.15ex,circle,fill=black]{};
\draw (0.75,-1) node(1)[inner sep=0.15ex,circle,fill=black]{};
\draw (0.75,-1.5) node(2)[inner sep=0.15ex,circle,fill=black]{};
\draw (1.25,-2) node(3)[inner sep=0.15ex,circle,fill=black]{};
\draw (0.25,-1) node(4)[inner sep=0.15ex,circle,fill=black]{};
\draw (0.75,-2) node(5)[inner sep=0.15ex,circle,fill=black]{};
\draw (0.75,-2.5) node(6)[inner sep=0.15ex,circle,fill=black]{};
\draw (0.25,-2) node(7)[inner sep=0.15ex,circle,fill=black]{};
\draw (-0.25,-1) node(8)[inner sep=0.15ex,circle,fill=black]{};
\draw (-0.25,-1.5) node(9)[inner sep=0.15ex,circle,fill=black]{};
\draw (-0.25,-2) node(10)[inner sep=0.15ex,circle,fill=black]{};
\draw (-0.75,-1) node(11)[inner sep=0.15ex,circle,fill=black]{};
\draw (-0.75,-1.5) node(12)[inner sep=0.15ex,circle,fill=black]{};
\draw (-0.75,-2) node(13)[inner sep=0.15ex,circle,fill=black]{};
\draw (-0.75,-2.5) node(14)[inner sep=0.15ex,circle,fill=black]{};
\draw (-0.75,-3) node(15)[inner sep=0.15ex,circle,fill=black]{};
\draw (-1.25,-2) node(16)[inner sep=0.15ex,circle,fill=black]{};
\draw (-1.25,-2.5) node(17)[inner sep=0.15ex,circle,fill=black]{};
\draw (-1.75,-2) node(18)[inner sep=0.15ex,circle,fill=black]{};

\node [right] at (0) {\tiny 0};
\node [right] at (1) {\tiny 1};
\node [right] at (2) {\tiny 2};
\node [right] at (3) {\tiny 3};
\node [right] at (4) {\tiny 4};
\node [right] at (5) {\tiny 5};
\node [right] at (6) {\tiny 6};
\node [right] at (7) {\tiny 7};
\node [right] at (8) {\tiny 8};
\node [right] at (9) {\tiny 9};
\node [right] at (10) {\tiny 10};
\node [right] at (11) {\tiny 11};
\node [right] at (12) {\tiny 12};
\node [right] at (13) {\tiny 13};
\node [right] at (14) {\tiny 14};
\node [right] at (15) {\tiny 15};
\node [right] at (16) {\tiny 16};
\node [right] at (17) {\tiny 17};
\node [right] at (18) {\tiny 18};
\end{tikzpicture}\\

$\xrightarrow{\mbox{    }\mbox{\normalsize $\Phi_{22}(f_2)$}\mbox{    }}$ &
\begin{tikzpicture}
\draw[red,thick] (-1.5,-0.5) -- (0,0) -- (-1,-0.5) -- (-1,-1.5);

\draw (0,0) node(root)[inner sep=0.3ex,circle,fill=black]{};
\draw (0.75,-0.5) node(0)[inner sep=0.15ex,circle,fill=black]{};
\draw (1.5,-1) node(1)[inner sep=0.15ex,circle,fill=black]{};
\draw (1.5,-1.5) node(2)[inner sep=0.15ex,circle,fill=black]{};
\draw (2,-2) node(3)[inner sep=0.15ex,circle,fill=black]{};
\draw (1,-1) node(4)[inner sep=0.15ex,circle,fill=black]{};
\draw (1.5,-2) node(5)[inner sep=0.15ex,circle,fill=black]{};
\draw (1.5,-2.5) node(6)[inner sep=0.15ex,circle,fill=black]{};
\draw (1,-2) node(7)[inner sep=0.15ex,circle,fill=black]{};
\draw (0.5,-1) node(8)[inner sep=0.15ex,circle,fill=black]{};
\draw (0.5,-1.5) node(9)[inner sep=0.15ex,circle,fill=black]{};
\draw (0.5,-2) node(10)[inner sep=0.15ex,circle,fill=black]{};
\draw (0,-1) node(11)[inner sep=0.15ex,circle,fill=black]{};
\draw (0,-1.5) node(12)[inner sep=0.15ex,circle,fill=black]{};
\draw (0,-2) node(13)[inner sep=0.15ex,circle,fill=black]{};
\draw (0,-2.5) node(14)[inner sep=0.15ex,circle,fill=black]{};
\draw (0,-3) node(15)[inner sep=0.15ex,circle,fill=black]{};
\draw (-0.5,-2) node(16)[inner sep=0.15ex,circle,fill=black]{};
\draw (-0.5,-2.5) node(17)[inner sep=0.15ex,circle,fill=black]{};
\draw (-1,-2) node(18)[inner sep=0.15ex,circle,fill=black]{};
\draw (-1,-0.5) node(19)[inner sep=0.15ex,circle,fill=black]{};
\draw (-1,-1) node(20)[inner sep=0.15ex,circle,fill=black]{};
\draw (-1,-1.5) node(21)[inner sep=0.15ex,circle,fill=black]{};
\draw (-1.5,-0.5) node(22)[inner sep=0.15ex,circle,fill=black]{};

\node [right] at (0) {\tiny 0};
\node [right] at (1) {\tiny 1};
\node [right] at (2) {\tiny 2};
\node [right] at (3) {\tiny 3};
\node [right] at (4) {\tiny 4};
\node [right] at (5) {\tiny 5};
\node [right] at (6) {\tiny 6};
\node [right] at (7) {\tiny 7};
\node [right] at (8) {\tiny 8};
\node [right] at (9) {\tiny 9};
\node [right] at (10) {\tiny 10};
\node [right] at (11) {\tiny 11};
\node [right] at (12) {\tiny 12};
\node [right] at (13) {\tiny 13};
\node [right] at (14) {\tiny 14};
\node [right] at (15) {\tiny 15};
\node [right] at (16) {\tiny 16};
\node [right] at (17) {\tiny 17};
\node [right] at (18) {\tiny 18};
\node [right] at (19) {\tiny 19};
\node [right] at (20) {\tiny 20};
\node [right] at (21) {\tiny 21};
\node [left] at (22) {\tiny 22};

\draw (root) -- (0);
\draw (0) -- (1);
\draw (0) -- (4);
\draw (1) -- (2);
\draw (2) -- (3);
\draw (2) -- (6);
\draw (2) -- (7);
\draw (0) -- (8);
\draw (8) -- (10);
\draw (0) -- (11);
\draw (11) -- (15);
\draw (12) -- (16);
\draw (12) -- (18);
\draw (16) -- (17);
\end{tikzpicture} & $\xrightarrow{\mbox{    }\mbox{\normalsize $\Phi_{25}(f_1)$}\mbox{    }}$ &
\begin{tikzpicture}
\draw[red,thick] (-2,-1.5) -- (-1,-1) -- (-1.5,-1.5) -- (-1.5,-2);

\draw (0,0) node(root)[inner sep=0.3ex,circle,fill=black]{};
\draw (0.75,-0.5) node(0)[inner sep=0.15ex,circle,fill=black]{};
\draw (1.5,-1) node(1)[inner sep=0.15ex,circle,fill=black]{};
\draw (1.5,-1.5) node(2)[inner sep=0.15ex,circle,fill=black]{};
\draw (2,-2) node(3)[inner sep=0.15ex,circle,fill=black]{};
\draw (1,-1) node(4)[inner sep=0.15ex,circle,fill=black]{};
\draw (1.5,-2) node(5)[inner sep=0.15ex,circle,fill=black]{};
\draw (1.5,-2.5) node(6)[inner sep=0.15ex,circle,fill=black]{};
\draw (1,-2) node(7)[inner sep=0.15ex,circle,fill=black]{};
\draw (0.5,-1) node(8)[inner sep=0.15ex,circle,fill=black]{};
\draw (0.5,-1.5) node(9)[inner sep=0.15ex,circle,fill=black]{};
\draw (0.5,-2) node(10)[inner sep=0.15ex,circle,fill=black]{};
\draw (0,-1) node(11)[inner sep=0.15ex,circle,fill=black]{};
\draw (0,-1.5) node(12)[inner sep=0.15ex,circle,fill=black]{};
\draw (0,-2) node(13)[inner sep=0.15ex,circle,fill=black]{};
\draw (0,-2.5) node(14)[inner sep=0.15ex,circle,fill=black]{};
\draw (0,-3) node(15)[inner sep=0.15ex,circle,fill=black]{};
\draw (-0.5,-2) node(16)[inner sep=0.15ex,circle,fill=black]{};
\draw (-0.5,-2.5) node(17)[inner sep=0.15ex,circle,fill=black]{};
\draw (-1,-2) node(18)[inner sep=0.15ex,circle,fill=black]{};
\draw (-1,-0.5) node(19)[inner sep=0.15ex,circle,fill=black]{};
\draw (-1,-1) node(20)[inner sep=0.15ex,circle,fill=black]{};
\draw (-1,-1.5) node(21)[inner sep=0.15ex,circle,fill=black]{};
\draw (-1.5,-0.5) node(22)[inner sep=0.15ex,circle,fill=black]{};
\draw (-1.5,-1.5) node(23)[inner sep=0.15ex,circle,fill=black]{};
\draw (-1.5,-2) node(24)[inner sep=0.15ex,circle,fill=black]{};
\draw (-2,-1.5) node(25)[inner sep=0.15ex,circle,fill=black]{};

\node [right] at (0) {\tiny 0};
\node [right] at (1) {\tiny 1};
\node [right] at (2) {\tiny 2};
\node [right] at (3) {\tiny 3};
\node [right] at (4) {\tiny 4};
\node [right] at (5) {\tiny 5};
\node [right] at (6) {\tiny 6};
\node [right] at (7) {\tiny 7};
\node [right] at (8) {\tiny 8};
\node [right] at (9) {\tiny 9};
\node [right] at (10) {\tiny 10};
\node [right] at (11) {\tiny 11};
\node [right] at (12) {\tiny 12};
\node [right] at (13) {\tiny 13};
\node [right] at (14) {\tiny 14};
\node [right] at (15) {\tiny 15};
\node [right] at (16) {\tiny 16};
\node [right] at (17) {\tiny 17};
\node [right] at (18) {\tiny 18};
\node [right] at (19) {\tiny 19};
\node [right] at (20) {\tiny 20};
\node [right] at (21) {\tiny 21};
\node [left] at (22) {\tiny 22};
\node [right] at (23) {\tiny 23};
\node [left] at (24) {\tiny 24};
\node [left] at (25) {\tiny 25};

\draw (root) -- (0);
\draw (0) -- (1);
\draw (0) -- (4);
\draw (1) -- (2);
\draw (2) -- (3);
\draw (2) -- (6);
\draw (2) -- (7);
\draw (0) -- (8);
\draw (8) -- (10);
\draw (0) -- (11);
\draw (11) -- (15);
\draw (12) -- (16);
\draw (12) -- (18);
\draw (16) -- (17);
\draw (root) -- (19);
\draw (19) -- (21);
\draw (root) -- (22);
\end{tikzpicture}
\end{tabular}

%% file: 123213example.tex
\renewcommand{\arraystretch}{1.2}
\newcolumntype{C}{ >{\centering\arraybackslash} m{0.8cm} }
\newcolumntype{D}{ >{\centering\arraybackslash} m{1.8cm} }
\newcolumntype{E}{ >{\centering\arraybackslash} m{2.4cm} }
\newcolumntype{F}{ >{\centering\arraybackslash} m{3cm} }

\begin{tabular}{F>{\centering\arraybackslash} m{0.1cm}F}
\begin{tabular}{|C|C|}
\hline $f$ & $()$\\\hline
$T$ & \begin{tikzpicture}
\draw (0,0.1) node{};
\draw (0,0) node(root)[inner sep=0.3ex,circle,fill=black]{};
\draw (0,-0.5) node(0)[inner sep=0.15ex,circle,fill=black]{};
\draw (root) -- (0);
\end{tikzpicture}\\\hline
\end{tabular}  & \quad & \begin{tabular}{|C|C|}
\hline$f$ & $(\{1\})$\\\hline
$T$ & \begin{tikzpicture}
\draw (0,0.1) node{};
\draw (0,0) node(root)[inner sep=0.3ex,circle,fill=black]{};
\draw (0.25,-0.5) node(0)[inner sep=0.15ex,circle,fill=black]{};
\draw (-0.25,-0.5) node(1)[inner sep=0.15ex,circle,fill=black]{};
\draw (root) -- (0);
\draw (root) -- (1);
\end{tikzpicture}\\\hline
\end{tabular}\\
\end{tabular}

\vspace{0.2cm}

\begin{tabular}{|C|D|D|D|}
\hline$f$ & $(\{1\},\{2\})$ & $(\{1,2\},\emptyset)$ & $(\{2\},\{1\})$ \\\hline
$T$ & \begin{tikzpicture}
\draw (0,0.1) node{};
\draw (0,0) node(root)[inner sep=0.3ex,circle,fill=black]{};
\draw (0.5,-0.25) node(0)[inner sep=0.15ex,circle,fill=black]{};
\draw (0.5,-0.5) node(1)[inner sep=0.15ex,circle,fill=black]{};
\draw (-0.5,-0.25) node(2)[inner sep=0.15ex,circle,fill=black]{};
\draw (root) -- (0);
\draw (1) -- (0);
\draw (2) -- (root);
\end{tikzpicture} & \begin{tikzpicture}
\draw (0,0.1) node{};
\draw (0,0) node(root)[inner sep=0.3ex,circle,fill=black]{};
\draw (0.5,-0.25) node(0)[inner sep=0.15ex,circle,fill=black]{};
\draw (-0.5,-0.25) node(1)[inner sep=0.15ex,circle,fill=black]{};
\draw (-0.5,-0.5) node(2)[inner sep=0.15ex,circle,fill=black]{};
\draw (root) -- (0);
\draw (1) -- (root);
\draw (2) -- (1);
\end{tikzpicture} & \begin{tikzpicture}
\draw (0,0.1) node{};
\draw (0,0) node(root)[inner sep=0.3ex,circle,fill=black]{};
\draw (0,-0.5) node(1)[inner sep=0.15ex,circle,fill=black]{};
\draw (0.5,-0.5) node(0)[inner sep=0.15ex,circle,fill=black]{};
\draw (-0.5,-0.5) node(2)[inner sep=0.15ex,circle,fill=black]{};
\draw (root) -- (0);
\draw (root) -- (1);
\draw (root) -- (2);
\end{tikzpicture}\\\hline
\end{tabular}

\vspace{0.2cm}

\begin{tabular}{|C|E|E|E|}
\hline$f$ & $(\{1\},\{3\},\{2\})$ & $(\{1,3\},\emptyset,\{2\})$ & $(\{1,3\},\{2\},\emptyset)$ \\\hline
$T$ & \begin{tikzpicture}
\draw (0,0.1) node{};
\draw (0,0) node(root)[inner sep=0.3ex,circle,fill=black]{};
\draw (0.5,-0.25) node(0)[inner sep=0.15ex,circle,fill=black]{};
\draw (0.5,-0.5) node(1)[inner sep=0.15ex,circle,fill=black]{};
\draw (0.5,-0.75) node(2)[inner sep=0.15ex,circle,fill=black]{};
\draw (-0.5,-0.25) node(3)[inner sep=0.15ex,circle,fill=black]{};
\draw (root) -- (0);
\draw (1) -- (0);
\draw (3) -- (root);
\draw (2) -- (1);
\end{tikzpicture} & \begin{tikzpicture}
\draw (0,0.1) node{};
\draw (0,0) node(root)[inner sep=0.3ex,circle,fill=black]{};
\draw (0.5,-0.25) node(0)[inner sep=0.15ex,circle,fill=black]{};
\draw (0.5,-0.5) node(1)[inner sep=0.15ex,circle,fill=black]{};
\draw (-0.5,-0.25) node(2)[inner sep=0.15ex,circle,fill=black]{};
\draw (-0.5,-0.5) node(3)[inner sep=0.15ex,circle,fill=black]{};
\draw (root) -- (0);
\draw (1) -- (0);
\draw (2) -- (root);
\draw (2) -- (3);
\end{tikzpicture} & \begin{tikzpicture}
\draw (0,0.1) node{};
\draw (0,0) node(root)[inner sep=0.3ex,circle,fill=black]{};
\draw (0.5,-0.25) node(0)[inner sep=0.15ex,circle,fill=black]{};
\draw (-0.5,-0.25) node(1)[inner sep=0.15ex,circle,fill=black]{};
\draw (-0.5,-0.5) node(2)[inner sep=0.15ex,circle,fill=black]{};
\draw (-0.5,-0.75) node(3)[inner sep=0.15ex,circle,fill=black]{};
\draw (root) -- (0);
\draw (1) -- (root);
\draw (2) -- (1);
\draw (2) -- (3);
\end{tikzpicture}\\
\hline$f$ & $(\{2\},\{3\},\{1\})$ & $(\{2,3\},\emptyset,\{1\})$ & $(\{2,3\},\{1\},\emptyset)$\\\hline
$T$ & \begin{tikzpicture}
\draw (0,0.1) node{};
\draw (0,0) node(root)[inner sep=0.3ex,circle,fill=black]{};
\draw (0.375,-0.25) node(0)[inner sep=0.15ex,circle,fill=black]{};
\draw (0.5,-0.5) node(1)[inner sep=0.15ex,circle,fill=black]{};
\draw (0.25,-0.5) node(2)[inner sep=0.15ex,circle,fill=black]{};
\draw (-0.5,-0.25) node(3)[inner sep=0.15ex,circle,fill=black]{};
\draw (root) -- (0);
\draw (1) -- (0);
\draw (3) -- (root);
\draw (2) -- (0);
\end{tikzpicture} & \begin{tikzpicture}
\draw (0,0.1) node{};
\draw (0,0) node(root)[inner sep=0.3ex,circle,fill=black]{};
\draw (0.5,-0.25) node(0)[inner sep=0.15ex,circle,fill=black]{};
\draw (0,-0.25) node(1)[inner sep=0.15ex,circle,fill=black]{};
\draw (-0.5,-0.25) node(2)[inner sep=0.15ex,circle,fill=black]{};
\draw (-0.5,-0.5) node(3)[inner sep=0.15ex,circle,fill=black]{};
\draw (root) -- (0);
\draw (1) -- (root);
\draw (2) -- (root);
\draw (2) -- (3);
\end{tikzpicture} & \begin{tikzpicture}
\draw (0,0.1) node{};
\draw (0,0) node(root)[inner sep=0.3ex,circle,fill=black]{};
\draw (-0.375,-0.25) node(1)[inner sep=0.15ex,circle,fill=black]{};
\draw (0.5,-0.25) node(0)[inner sep=0.15ex,circle,fill=black]{};
\draw (-0.25,-0.5) node(2)[inner sep=0.15ex,circle,fill=black]{};
\draw (-0.5,-0.5) node(3)[inner sep=0.15ex,circle,fill=black]{};
\draw (root) -- (0);
\draw (1) -- (3);
\draw (1) -- (root);
\draw (2) -- (1);
\end{tikzpicture}\\\hline
$f$ & $(\{3\},\{1\},\{2\})$ & $(\{3\},\{1,2\},\emptyset)$ & $(\{3\},\{2\},\{1\})$\\\hline
$T$ & \begin{tikzpicture}
\draw (0,0.1) node{};
\draw (0,0) node(root)[inner sep=0.3ex,circle,fill=black]{};
\draw (0.5,-0.25) node(0)[inner sep=0.15ex,circle,fill=black]{};
\draw (0,-0.25) node(1)[inner sep=0.15ex,circle,fill=black]{};
\draw (-0.5,-0.25) node(2)[inner sep=0.15ex,circle,fill=black]{};
\draw (0.5,-0.5) node(3)[inner sep=0.15ex,circle,fill=black]{};
\draw (root) -- (0);
\draw (1) -- (root);
\draw (2) -- (root);
\draw (0) -- (3);
\end{tikzpicture} & \begin{tikzpicture}
\draw (0,0.1) node{};
\draw (0,0) node(root)[inner sep=0.3ex,circle,fill=black]{};
\draw (0.5,-0.25) node(0)[inner sep=0.15ex,circle,fill=black]{};
\draw (0,-0.25) node(1)[inner sep=0.15ex,circle,fill=black]{};
\draw (-0.5,-0.25) node(2)[inner sep=0.15ex,circle,fill=black]{};
\draw (0,-0.5) node(3)[inner sep=0.15ex,circle,fill=black]{};
\draw (root) -- (0);
\draw (1) -- (root);
\draw (2) -- (root);
\draw (1) -- (3);
\end{tikzpicture} & \begin{tikzpicture}
\draw (0,0.1) node{};
\draw (0,0) node(root)[inner sep=0.3ex,circle,fill=black]{};
\draw (0.167,-0.5) node(1)[inner sep=0.15ex,circle,fill=black]{};
\draw (0.5,-0.5) node(0)[inner sep=0.15ex,circle,fill=black]{};
\draw (-0.167,-0.5) node(2)[inner sep=0.15ex,circle,fill=black]{};
\draw (-0.5,-0.5) node(3)[inner sep=0.15ex,circle,fill=black]{};
\draw (root) -- (0);
\draw (1) -- (root);
\draw (2) -- (root);
\draw (3) -- (root);
\end{tikzpicture}\\\hline 
\end{tabular}

%% file: 123213big.tex
\renewcommand{\arraystretch}{1.2}
\newcolumntype{C}{ >{\centering\arraybackslash} m{1.8cm} }
\newcolumntype{D}{ >{\centering\arraybackslash} m{4.8cm} }

\begin{tabular}{c||c|c|c|c|c|c|c|c|c|c}
  block & $\{18,20\}$ & $\{19\}$ & $\{15,17\}$ & $\{16\}$ & $\emptyset$ & $\{12,14\}$ & $\{13\}$ & $\{9\}$ & $\emptyset$ & $\{11\}$\\\hline
  cluster number & 1 & 1 & 2 & 2 & 2 & 3 & 3 & 4 & 3 & 4
\end{tabular}

\mbox{}

\mbox{}

\begin{tabular}{c||c|c|c|c|c|c|c|c|c|c}
  block & $\{10\}$ & $\{6,8\}$ & $\{7\}$ & $\{1\}$ & $\{5\}$ & $\{4\}$ & $\emptyset$ & $\{3\}$ & $\emptyset$ & $\{2\}$\\\hline
  cluster number & 4 & 5 & 5 & 6 & 6 & 6 & 5 & 6 & 1 & 6
\end{tabular}

\mbox{}

\mbox{}

\begin{tabular}{c||c|c|c|c|c|c}
  cluster number & 1 & 2 & 3 & 4 & 5 & 6\\\hline
  cluster type & open & closed & open & closed & open & closed\\\hline
  parameter & / & 0 & / & 2 & / & 4
\end{tabular}

\mbox{}

\mbox{}

\begin{tabular}{CDCD}
$\xrightarrow{\mbox{    }\mbox{\normalsize initialise}\mbox{    }}$ & \begin{tikzpicture}
\draw[red,thick] (0,0) -- (0,-0.5);

\draw (0,0) node(root)[inner sep=0.3ex,circle,fill=black]{};
\draw (0,-0.5) node(0)[inner sep=0.15ex,circle,fill=black]{};
\end{tikzpicture}

& $\xrightarrow{\mbox{    }\mbox{\normalsize $\Phi_5(f_6)$}\mbox{    }}$ &

\begin{tikzpicture}
\draw[red,thick] (-2,2) -- (0,0);
\draw[red,thick] (-2,2) -- (-2.5,1.5);
\draw (0,0) -- (0.5,-0.5);

\draw (0,0) node(root)[inner sep=0.15ex,circle,fill=black]{};
\draw (0.5,-0.5) node(0)[inner sep=0.15ex,circle,fill=black]{};
\draw (-0.5,0.5) node(1)[inner sep=0.15ex,circle,fill=black]{};
\draw (-1,1) node(2)[inner sep=0.15ex,circle,fill=black]{};
\draw (-1.5,1.5) node(3)[inner sep=0.15ex,circle,fill=black]{};
\draw (-2,2) node(4)[inner sep=0.3ex,circle,fill=black]{};
\draw (-2.5,1.5) node(5)[inner sep=0.15ex,circle,fill=black]{};
\end{tikzpicture}\\

$\xrightarrow{\mbox{    }\mbox{\normalsize $\Phi_8(f_5)$}\mbox{    }}$ &

\begin{tikzpicture}
\draw (-1,1) -- (0.5,-0.5);
\draw[red,thick] (-1,1) -- (-1.5,0.5);
\draw[red,thick] (-2.5,-0.5) -- (-3.5,-1.5);
\draw (-1.5,0.5) -- (-2.5,-0.5) -- (-2,-1);

\draw (0,0) node(root)[inner sep=0.15ex,circle,fill=black]{};
\draw (0.5,-0.5) node(0)[inner sep=0.15ex,circle,fill=black]{};
\draw (-0.5,0.5) node(1)[inner sep=0.15ex,circle,fill=black]{};
\draw (-1,1) node(2)[inner sep=0.3ex,circle,fill=black]{};
\draw (-1.5,0.5) node(6)[inner sep=0.15ex,circle,fill=black]{};
\draw (-2,0) node(3)[inner sep=0.15ex,circle,fill=black]{};
\draw (-2.5,-0.5) node(4)[inner sep=0.15ex,circle,fill=black]{};
\draw (-2,-1) node(5)[inner sep=0.15ex,circle,fill=black]{};
\draw (-3,-1) node(7)[inner sep=0.15ex,circle,fill=black]{};
\draw (-3.5,-1.5) node(8)[inner sep=0.15ex,circle,fill=black]{};
\end{tikzpicture}

& $\xrightarrow{\mbox{    }\mbox{\normalsize $\Phi_{11}(f_4)$}\mbox{    }}$ &

\begin{tikzpicture}
\draw (-1,1) -- (0.5,-0.5);
\draw (-1,1) -- (-3.5,-1.5);
\draw (-2.5,-0.5) -- (-2,-1);
\draw[red,thick] (-1,1) -- (-2,2) -- (-2.5,1.5);

\draw (0,0) node(root)[inner sep=0.15ex,circle,fill=black]{};
\draw (0.5,-0.5) node(0)[inner sep=0.15ex,circle,fill=black]{};
\draw (-0.5,0.5) node(1)[inner sep=0.15ex,circle,fill=black]{};
\draw (-1,1) node(2)[inner sep=0.15ex,circle,fill=black]{};
\draw (-1.5,0.5) node(6)[inner sep=0.15ex,circle,fill=black]{};
\draw (-2,0) node(3)[inner sep=0.15ex,circle,fill=black]{};
\draw (-2.5,-0.5) node(4)[inner sep=0.15ex,circle,fill=black]{};
\draw (-2,-1) node(5)[inner sep=0.15ex,circle,fill=black]{};
\draw (-3,-1) node(7)[inner sep=0.15ex,circle,fill=black]{};
\draw (-3.5,-1.5) node(8)[inner sep=0.15ex,circle,fill=black]{};
\draw (-1.5,1.5) node(9)[inner sep=0.15ex,circle,fill=black]{};
\draw (-2,2) node(10)[inner sep=0.3ex,circle,fill=black]{};
\draw (-2.5,1.5) node(11)[inner sep=0.15ex,circle,fill=black]{};
\end{tikzpicture}\\

$\xrightarrow{\mbox{    }\mbox{\normalsize $\Phi_{14}(f_3)$}\mbox{    }}$ &

\begin{tikzpicture}
\draw (-1,1) -- (0.5,-0.5);
\draw (-1,1) -- (-3.5,-1.5);
\draw (-2.5,-0.5) -- (-2,-1);
\draw (-1,1) -- (-2,2);
\draw[red,thick] (-2,2) -- (-3.5,0.5);
\draw (-2.5,1.5) -- (-2,1);

\draw (0,0) node(root)[inner sep=0.15ex,circle,fill=black]{};
\draw (0.5,-0.5) node(0)[inner sep=0.15ex,circle,fill=black]{};
\draw (-0.5,0.5) node(1)[inner sep=0.15ex,circle,fill=black]{};
\draw (-1,1) node(2)[inner sep=0.15ex,circle,fill=black]{};
\draw (-1.5,0.5) node(6)[inner sep=0.15ex,circle,fill=black]{};
\draw (-2,0) node(3)[inner sep=0.15ex,circle,fill=black]{};
\draw (-2.5,-0.5) node(4)[inner sep=0.15ex,circle,fill=black]{};
\draw (-2,-1) node(5)[inner sep=0.15ex,circle,fill=black]{};
\draw (-3,-1) node(7)[inner sep=0.15ex,circle,fill=black]{};
\draw (-3.5,-1.5) node(8)[inner sep=0.15ex,circle,fill=black]{};
\draw (-1.5,1.5) node(9)[inner sep=0.15ex,circle,fill=black]{};
\draw (-2,2) node(10)[inner sep=0.3ex,circle,fill=black]{};
\draw (-2.5,1.5) node(12)[inner sep=0.15ex,circle,fill=black]{};
\draw (-2,1) node(11)[inner sep=0.15ex,circle,fill=black]{};
\draw (-3,1) node(13)[inner sep=0.15ex,circle,fill=black]{};
\draw (-3.5,0.5) node(14)[inner sep=0.15ex,circle,fill=black]{};
\end{tikzpicture}

& $\xrightarrow{\mbox{    }\mbox{\normalsize $\Phi_{17}(f_2)$}\mbox{    }}$ &

\begin{tikzpicture}
\draw (-1,1) -- (0.5,-0.5);
\draw (-1,1) -- (-1,-1.5);
\draw (-1,-0.5) -- (-0.5,-1);
\draw (-1,1) -- (-2,2) -- (-2,0.5);
\draw (-2,1.5) -- (-1.5,1);
\draw[red,thick] (-2,2) -- (-3.5,0.5);

\draw (0,0) node(root)[inner sep=0.15ex,circle,fill=black]{};
\draw (0.5,-0.5) node(0)[inner sep=0.15ex,circle,fill=black]{};
\draw (-0.5,0.5) node(1)[inner sep=0.15ex,circle,fill=black]{};
\draw (-1,1) node(2)[inner sep=0.15ex,circle,fill=black]{};
\draw (-1,0.5) node(6)[inner sep=0.15ex,circle,fill=black]{};
\draw (-1,0) node(3)[inner sep=0.15ex,circle,fill=black]{};
\draw (-1,-0.5) node(4)[inner sep=0.15ex,circle,fill=black]{};
\draw (-0.5,-1) node(5)[inner sep=0.15ex,circle,fill=black]{};
\draw (-1,-1) node(7)[inner sep=0.15ex,circle,fill=black]{};
\draw (-1,-1.5) node(8)[inner sep=0.15ex,circle,fill=black]{};
\draw (-1.5,1.5) node(9)[inner sep=0.15ex,circle,fill=black]{};
\draw (-2,2) node(10)[inner sep=0.3ex,circle,fill=black]{};
\draw (-2,1.5) node(12)[inner sep=0.15ex,circle,fill=black]{};
\draw (-1.5,1) node(11)[inner sep=0.15ex,circle,fill=black]{};
\draw (-2,1) node(13)[inner sep=0.15ex,circle,fill=black]{};
\draw (-2,0.5) node(14)[inner sep=0.15ex,circle,fill=black]{};
\draw (-2.5,1.5) node(15)[inner sep=0.15ex,circle,fill=black]{};
\draw (-3,1) node(16)[inner sep=0.15ex,circle,fill=black]{};
\draw (-3.5,0.5) node(17)[inner sep=0.15ex,circle,fill=black]{};
\end{tikzpicture}\\
\end{tabular}

\begin{tabular}{CD}
   $\xrightarrow{\mbox{    }\mbox{\normalsize $\Phi_{20}(f_1)$}\mbox{    }}$  &  
\begin{tikzpicture}
\draw (-0.5,0.5) -- (0.5,-0.5);
\draw (-1,0) -- (-2.5,-1.5) -- (-2.5,-3);
\draw (0,-2) -- (-1.5,-0.5);
\draw (0.5,-2) -- (0,-2) -- (1,-3);
\draw (-2.5,-1.5) -- (-1,-3);
\draw (-2,-2) -- (-1.5,-2);
\draw[red,thick] (-2.5,-1.5) -- (-3.5,-2.5);
\draw[red,thick] (-0.5,0.5) -- (-1,0);

\draw (0,0) node(root)[inner sep=0.15ex,circle,fill=black]{};
\draw (0.5,-0.5) node(0)[inner sep=0.15ex,circle,fill=black]{};
\draw (-0.5,0.5) node(1)[inner sep=0.3ex,circle,fill=black]{};
\draw (-1,0) node(18)[inner sep=0.15ex,circle,fill=black]{};
\draw (-1.5,-0.5) node(2)[inner sep=0.15ex,circle,fill=black]{};
\draw (-1,-1) node(6)[inner sep=0.15ex,circle,fill=black]{};
\draw (-0.5,-1.5) node(3)[inner sep=0.15ex,circle,fill=black]{};
\draw (0,-2) node(4)[inner sep=0.15ex,circle,fill=black]{};
\draw (0.5,-2) node(5)[inner sep=0.15ex,circle,fill=black]{};
\draw (0.5,-2.5) node(7)[inner sep=0.15ex,circle,fill=black]{};
\draw (1,-3) node(8)[inner sep=0.15ex,circle,fill=black]{};
\draw (-2,-1) node(9)[inner sep=0.15ex,circle,fill=black]{};
\draw (-2.5,-1.5) node(10)[inner sep=0.15ex,circle,fill=black]{};
\draw (-2,-2) node(12)[inner sep=0.15ex,circle,fill=black]{};
\draw (-1.5,-2) node(11)[inner sep=0.15ex,circle,fill=black]{};
\draw (-1.5,-2.5) node(13)[inner sep=0.15ex,circle,fill=black]{};
\draw (-1,-3) node(14)[inner sep=0.15ex,circle,fill=black]{};
\draw (-2.5,-2) node(15)[inner sep=0.15ex,circle,fill=black]{};
\draw (-2.5,-2.5) node(16)[inner sep=0.15ex,circle,fill=black]{};
\draw (-2.5,-3) node(17)[inner sep=0.15ex,circle,fill=black]{};
\draw (-3,-2) node(19)[inner sep=0.15ex,circle,fill=black]{};
\draw (-3.5,-2.5) node(20)[inner sep=0.15ex,circle,fill=black]{};
\end{tikzpicture}\\
\end{tabular}